\documentclass[12pt]{article}
\usepackage[utf8]{inputenc}
\usepackage{latexsym, amsmath, amsthm, amsfonts}
\usepackage{bbm}
\usepackage{enumerate, amssymb, xspace}
\usepackage[mathscr]{eucal}
\usepackage{graphicx}
\usepackage{rotating}
\usepackage{hyperref} 
\usepackage{color,colordvi}

\usepackage{mathtools}  
\usepackage{stmaryrd}   
\usepackage{dutchcal}   

\usepackage{tikz}
\usepackage{tikz-cd}
\usetikzlibrary{matrix,arrows,tqft,calc,decorations,arrows,shapes}
\usetikzlibrary{decorations.markings}
\usetikzlibrary{decorations.pathmorphing}
\usetikzlibrary{arrows.meta}  

\usepackage{xargs} 

\def\simptexthickness{1.4pt}
\def\ddraw{\draw[postaction=decorate]}

\newcommand{\simptex}[2][]{
\begin{tikzpicture}[line width=\simptexthickness,rounded corners,decoration={markings,mark=at position 0.46 with {\arrowreversed{Stealth[length=6pt,width=6pt]}}},#1]
#2
\end{tikzpicture}
}

\def\onesimplexvertices #1-#2- {
\node (nx) at (0,0) {$#1$};
\node (ny) at (2,0) {$#2$};
}
\newcommandx{\onesimplex}[3][1={-},2={},3={<-}]{
\begin{tikzpicture}
\onesimplexvertices#1 - - 
\draw[#3] (nx.east)  to[edge label={$#2$}] (ny.west);
\end{tikzpicture}
}

\def\twosimplexvertices #1-#2-#3- {
\node (nx) at (0,0) {};
\node (ny) at (1,1.4) {};
\node (nz) at (2,0) {};
\draw (nx.west) node {\ensuremath{#1}};
\draw (ny.north) node {\ensuremath{#2}};
\draw (nz.east) node {\ensuremath{#3}};
}
\def\twosimplexedges #1;#2,#3,#4, {
\draw[postaction=decorate] (nx.center)  to[edge label={\ensuremath{#2}}] (ny.center);
\draw[postaction=decorate] (ny.center)  to[edge label={\ensuremath{#3}}] (nz.center);
\draw[postaction=decorate] (nx.center)  to[swap,edge label={\ensuremath{#4}}] (nz.center);
}
\newcommandx{\twosimplex}[4][1={--},2={,,},3={},4={line width=\simptexthickness,rounded corners,decoration={markings,mark=at position 0.46 with {\arrowreversed{Stealth[length=6pt,width=6pt]}}}}]{
\begin{tikzpicture}[#4]
\twosimplexvertices#1 - - - 
\twosimplexedges 4;#2 , , , 
\node (face) at (1,0.5) {$#3$};
\end{tikzpicture}
}
\newcommandx\internaltwosimplex[4][1={--},2={,,},3={},4={line width=\simptexthickness,rounded corners,decoration={markings,mark=at position 0.46 with {\arrowreversed{Stealth[length=6pt,width=6pt]}}}}]{
\begin{scope}[#4]
\twosimplexvertices#1 - - - 
\twosimplexedges 4;#2 , , , 
\node (face) at (1,0.5) {$#3$};
\end{scope}
}

\def\threesimplexvertices #1-#2-#3-#4- {
\node[left,inner sep=1.6pt,scale=1] (va) at (0,0){\scriptsize\ensuremath{#1}};
\node[above,inner sep=1.9pt,scale=1] (vb) at (3,4.5){\scriptsize\ensuremath{#2}};
\node[right,inner sep=1.6pt,scale=1] (vc) at (6,0){\scriptsize\ensuremath{#3}};
\node[inner sep=0.01pt,scale=1] (vd) at (3,2)["\scriptsize\ensuremath{#4}" below]{};
}
\def\threesimplexfaces #1,#2,#3,#4, {
\node[scale=1] (fa) at (3,0.85) {\reflectbox{\scriptsize\ensuremath{#4}}};
\node[scale=1] (fb) at (2.2,2.3) {\reflectbox{\scriptsize\ensuremath{#1}}};
\node[scale=1] (fc) at (3.8,2.3) {\reflectbox{\scriptsize\ensuremath{#2}}};
\node[scale=1] (fd) at (3,1.55) {};
\draw[scale=0.8] (3.2,1.4) node{\ensuremath{\color{\colorFm}#3}};
}
\def\Gsimplexfaces #1,#2,#3,#4, {
\node[scale=1] (fa) at (3,0.85) {\ensuremath{#4}};
\node[scale=1] (fb) at (2.2,2.3) {\ensuremath{#1}};
\node[scale=1] (fc) at (3.8,2.3) {\ensuremath{#2}};
\node[scale=1] (fd) at (3,1.55) {};
\draw (3.2,1.4) node{\reflectbox{\scriptsize\ensuremath{#3}}};
}
\def\threesimplexedges #1;#2,#3,#4,#5,#6,#7 {
\draw[#1] (va.east)  to[edge label={\scriptsize\ensuremath{#2}}] (vb.south); 
\draw[#1] (vb.south)  to[edge label={\scriptsize\ensuremath{#3}}] (vc.west); 
\draw[#1] (vc.west)  to[edge label={\scriptsize\ensuremath{#4}}] (vd); 
\draw[#1] (va.east)  to[swap,edge label={\scriptsize\ensuremath{#6}}] (vc.west); 
\draw[#1] (vb.south)  to[edge label={\scriptsize\ensuremath{#5}}] (vd); 
\draw[#1] (va.east)  to[swap,edge label={\scriptsize\ensuremath{#7}}] (vd); 
}
\def\threesimplexsmoothedges #1;#2,#3,#4,#5,#6,#7 {
\ddraw (va.east)  to[edge label={\scriptsize\ensuremath{#2}}] (vb.south); 
\ddraw (vb.south)  to[edge label={\scriptsize\ensuremath{#3}}] (vc.west); 
\ddraw (vc.west)  to[edge label={\scriptsize\ensuremath{#4}}] (vd.center); 
\ddraw (va.east)  to[swap,edge label={\scriptsize\ensuremath{#6}}] (vc.west); 
\ddraw (vb.south)  to[edge label={\scriptsize\ensuremath{#5}}] (vd.center); 
\ddraw (va.east)  to[swap,edge label={\scriptsize\ensuremath{#7}}] (vd.center); 
}
\def\Gsimplexsmoothedges #1;#2,#3,#4,#5,#6,#7 {
\draw[postaction={decorate}] (va.east)  to[edge label={\scriptsize\ensuremath{#2}}] (vb.south); 
\draw[postaction={decorate}]  (vb.south)  to[edge label={\scriptsize\ensuremath{#3}}] (vc.west); 
\draw[postaction={decorate},gray]  (vc.west)  to[edge label={\scriptsize\ensuremath{#4}}] (vd.center); 
\draw[postaction={decorate}]  (va.east)  to[swap,edge label={\scriptsize\ensuremath{#6}}] (vc.west); 
\draw[postaction={decorate},gray]  (vb.south)  to[edge label={\scriptsize\ensuremath{#5}}] (vd.center); 
\draw[postaction={decorate},gray]  (va.east)  to[swap,edge label={\scriptsize\ensuremath{#7}}] (vd.center); 
}
\newcommandx{\threesimplex}[4][1={---},2={ , , , , , },3={ , , , },4={line width=\simptexthickness,rounded corners,decoration={markings,mark=at position 0.46 with {\arrowreversed{Stealth[length=6pt,width=6pt]}}}}]{
\begin{tikzpicture}[#4]
\threesimplexvertices#1 - - - - 
\threesimplexfaces#3 , , , , 
\Gsimplexsmoothedges#4;#2 , , , , , , 
\end{tikzpicture}
}
\newcommandx{\Gsimplex}[4][1={---},2={ , , , , , },3={ , , , },4={line width=\simptexthickness,rounded corners,decoration={markings,mark=at position 0.46 with {\arrowreversed{Stealth[length=6pt,width=6pt]}}}}]{
\begin{tikzpicture}[#4]
\threesimplexvertices#1 - - - - 
\threesimplexsmoothedges#4;#2 , , , , , , 
\Gsimplexfaces#3 , , , , 
\end{tikzpicture}
}
\newcommandx{\internalthreesimplex}[4][1={---},2={ , , , , , },3={ , , , },4={line width=\simptexthickness,rounded corners,decoration={markings,mark=at position 0.46 with {\arrowreversed{Stealth[length=6pt,width=6pt]}}}}]{
\begin{scope}[#4]
\threesimplexvertices#1 - - - - 
\threesimplexfaces#3 , , , , 
\Gsimplexsmoothedges#4;#2 , , , , , , 
\end{scope}
}
\newcommandx{\internalGsimplex}[4][1={---},2={ , , , , , },3={ , , , },4={line width=\simptexthickness,rounded corners,decoration={markings,mark=at position 0.46 with {\arrowreversed{Stealth[length=6pt,width=6pt]}}}}]{
\begin{scope}[#4]
\threesimplexvertices#1 - - - - 
\threesimplexsmoothedges#4;#2 , , , , , , 
\Gsimplexfaces#3 , , , , 
\end{scope}
}

\makeatletter
\DeclareRobustCommand\simpY{\@ifnextchar^{\simpYup}{\@ifnextchar_{\simpYdown}{\Ysimplex}}}
\def\simpYup^#1#2_#3 {\Ysimplex[#1,#2,#3]}
\def\simpYdown_#1^#2#3 {\simpYup^#2#3_#1 }

\def\Ysimplexvertices#1,#2,#3,{
\node[above,scale=0.65] (i) at (0,1) {\ensuremath{#1}};
\node[above,scale=0.65] (j) at (0.5,1) {\ensuremath{#2}};
\node[below,scale=0.65] (k) at (0.25,0) {\ensuremath{#3}};}
\newcommandx{\Ysimplex}[2][1={ , , },2={}]{
\begin{tikzpicture}[scale=0.65]
\Ysimplexvertices#1, , , 
\node[right,scale=0.7] (morph) at (0.25,0.5) {\ensuremath{#2}};
\draw (i) to (0.25,0.5);
\draw (j) to (0.25,0.5) to (k);
\end{tikzpicture}
}

\makeatletter
\DeclareRobustCommand\simpA{\@ifnextchar^{\simpAup}{\@ifnextchar_{\simpAdown}{\Asimplex}}}
\def\simpAup^#1_#2#3 {\simpAdown_#2#3^#1 }
\def\simpAdown_#1#2^#3 {\Asimplex[#1,#2,#3] }

\def\Asimplexvertices#1,#2,#3,{
\node[below,scale=0.65] (i) at (0,0) {\ensuremath{#1}};
\node[below,scale=0.65] (j) at (0.5,0) {\ensuremath{#2}};
\node[above,scale=0.65] (k) at (0.25,1) {\ensuremath{#3}};
}
\newcommandx{\Asimplex}[2][1={ , , },2={}]{
\begin{tikzpicture}[scale=0.65]
\Asimplexvertices#1, , , 
\node[right,scale=0.7] (morph) at (0.25,0.5) {\ensuremath{#2}};
\draw (i) to (0.25,0.5);
\draw (j) to (0.25,0.5) to (k);
\end{tikzpicture}
}

\usepackage{pifont}
\usepackage{psfrag}
\usepackage{stackrel}

\def\arrowDefect   {stealth'}

\def\colorFm       {blue}  
\def\colorObject   {blue!50!black}

\def\widthObject   {1.4pt}

\newcommand\rigidarrowright[1]{ \scopeArrow{0.9}{\arrowDefect}
                   \draw[\colorObject,line width=\widthObject,postaction={decorate}]
                   (#1) -- +(-0.13,0); \end{scope}}

\newcommand\basmorph[1]{\path 
		   node at (#1) [shape=rectangle,inner sep=2.8pt,draw=\colorObject,fill=yellow] {};}
\newcommand\basmorpho[1]{\path 
                   node at (#1) [shape=rectangle,inner sep=2.8pt,draw=\colorObject,fill=yellow] {};
                   \node[below=1pt] at (#1) {$\oo$};}
\newcommand\basmorphO[1]{\path 
                   node at (#1) [shape=rectangle,inner sep=2.8pt,draw=\colorObject,fill=yellow] {};
                   \node[above=1pt] at (#1) {$\oo$};}

\newcommand\scopeArrow[2] {\begin{scope}[decoration={markings,mark=at position #1
                   with \arrow{#2}}]}  

\def\Bar           {\overline }
\def\barpi         {\varpi}
\def\be            {\begin{equation}}
\def\bearl         {\begin{array}{l}}
\def\bearll        {\begin{array}{ll}}
\def\bearlll       {\begin{array}{lll}}

\def\C             {{\ensuremath{\mathcal C}}}

\def\cc            {\overline }   
\def\cdo           {\,{\cdot}\,}

\def\cir           {\,{\circ}\,}

\def\coev          {{\mathrm{coev}\!}}

\def\complex       {{\ensuremath{\mathbbm C}}}

\def\cupp          {\cap}   
\def\Cupp          {{}^{\cupp\!}}
\def\CUpp          {{}^{\cupp\!\!}}
\def\D             {{\ensuremath{\mathcal D}}}
\newcommand\dd[1]  {{\delta_{#1}}}

\newcommand\dF[1]  {\mathrm d_{#1}^{\mathrm{FP}}}

\newcommand\di[1]  {{\mathrm d_{#1}}}
\def\dim           {\text{dim}}

\def\dimk          {\dim_\ko}

\def\dsty          {\displaystyle }

\def\ee            {\end{equation}}

\def\eear          {\end{array}}

\def\End           {{\ensuremath{\mathrm{End}}}}
\def\EndC          {{\ensuremath{\mathrm{End}_\C}}}

\def\eps           {\varepsilon}
\def\eq            {\,{=}\,}

\newcommand\equ[1] {\stackrel{\eqref{#1}}=}

\def\ev            {{\mathrm{ev}\!}}

\newcommand\F[3]   {\def\FpartA{#1}\def\FpartB{#2}\def\FpartC{#3}\FpartDEFGHIJ}

\def\FF            {{\mathrm F}}
\newcommand\FFm[6] {\FF^{(#1\, #2\, #3)\, #4}_{#5,#6}}

\newcommand\Fm[6]  {\mathbb F^{(#1\, #2\, #3)\, #4}_{#5,#6}}
\newcommand\FM[4]  {\FF^{({#1}\,{#2}\,{#3})\,{#4}}_{}}

\newcommand\Fo[1]  {\F{#1}{\Bar{#1}}{#1}{#1}\OO\OO{{}}{{}}{{}}{{}}}
\newcommand\FO[1]  {\F{#1}{{#1}}{#1}{#1}\OO\OO{{}}{{}}{{}}{{}}}
\newcommand\Fob[1] {\F{\Bar{#1}}{#1}{\Bar{#1}}{\Bar{#1}}\OO\OO{{}}{{}}{{}}{{}}}
\newcommand\Foo[1] {\F{#1}{\Bar{#1}}{#1}{#1}\OO\OO\oo\oo\oo\oo}

\newcommand\FpartDEFGHIJ[7]
                   {\FF^{(\FpartA\,\FpartB\,\FpartC)\,#1}_{{\ssm#4}#2{\ssm#5},{\ssm#6}#3{\ssm#7}}}

\def\fs            {\nu}      
\def\FS            {\mathcal V}
\def\Fta           {\FF^{(\ta\,\ta\,\ta)\,\ta}}
\def\Fte           {\FF^{(\ta\,\ta\,\ta)\,\OO}}
\def\Ftot          {\mathbb F}

\newcommand\Fx[1]  {\FF^{#1}_{}}
\newcommand\G[3]   {\def\GpartA{#1}\def\GpartB{#2}\def\GpartC{#3}\GpartDEFGHIJ}
\def\GG            {{\mathrm G}}
\newcommand\GGm[6] {\GG^{(#1\, #2\, #3)\, #4}_{#5,#6}}

\newcommand\Gm[6]  {\mathbb G^{(#1\, #2\, #3)\, #4}_{#5,#6}}
\newcommand\GM[4]  {\GG^{({#1}\,{#2}\,{#3})\,{#4}}_{}}

\newcommand\Go[1]  {\G{#1}{\Bar{#1}}{#1}{#1}\OO\OO{{}}{{}}{{}}{{}}}
\newcommand\Gob[1] {\G{\Bar{#1}}{#1}{\Bar{#1}}{\Bar{#1}}\OO\OO{{}}{{}}{{}}{{}}}
\newcommand\Goo[1] {\G{#1}{\Bar{#1}}{#1}{#1}\OO\OO\oo\oo\oo\oo}

\newcommand\GpartDEFGHIJ[7]
            {\mathrm G^{(\GpartA\,\GpartB\,\GpartC)\,#1}_{{\ssm#4}#2{\ssm#5},{\ssm#6}#3{\ssm#7}}}

\def\Gr            {{\ensuremath{\mathrm{Gr}}}}

\def\gusion        {veined near-fu\-sion{} }

\newcommand\Hd[3]  {{\HH^{#1,#2}_{\phantom{#1,}#3}}}
\def\HF            {{\mathrm H}}
\newcommand\HFm[6] {\HF^{(#1\, #2\, #3)\, #4}_{#5,#6}}
\def\HG            {\widetilde{\mathrm H}}
\newcommand\HGm[6] {\HG^{(#1\, #2\, #3)\, #4}_{#5,#6}}
\def\HH            {{\mathrm H}}
\newcommand\Hm[3]  {{\HH_{#1,#2}^{\phantom{#1,}#3}}}

\def\Hom           {{\ensuremath{\mathrm{Hom}}}}

\def\HomC          {{\ensuremath{\mathrm{Hom}_\C}}}
\def\HomD          {{\ensuremath{\mathrm{Hom}_\D}}}

\def\HtotF         {\HH^{(4)}}
\def\I             {{S}}
\def\ib            {{\Bar i}}
\def\id            {\text{\sl id}}
\def\Id            {\text{\sl Id}}

\def\iN            {\,{\in}\,}
\def\inv           {^{-1}}

\def\jb            {{\Bar j}}
\def\kb            {{\Bar k}}
\def\ko            {{\ensuremath{\Bbbk}}}

\def\lb            {{\Bar l}}
\def\le            {\,{\leq}\,}

\def\lhs           {left hand side}
\def\LL            {{\breve L}}

\def\MM            {{\mathrm M}}

\newcommand\N[3]   {N_{#1 #2}^{\phantom{#1#2}\!\!#3}}

\def\ndgusion      {veined fu\-sion{} }
\def\NDgusion      {\textit{veined} fu\-sion{} }

\newcommand\nxl[1] {\\[#1mm]}
\newcommand\Nxl[1] {\\[-1.3em]\\[#1mm]}

\def\Obj           {\mathrm{Obj}}

\def\one           {{\bf1}}

\def\oo            {{\ssm\circ}}   
\def\OO            {1}             
\newcommand\oob[1] {{}^{\Bar{#1}}{\vee}{}^{#1}}   
\newcommand\oobb[1]{{}^{#1}{\vee}{}^{\Bar{#1}}}
\newcommand\oocap[1] {{}_{\Bar{#1}}{\cap}{}_{#1}^{}}   
\newcommand\ooCap[1] {{}_{\Bar{#1}}{\cap}{}_{#1}}  
\newcommand\oocapp[1]{{}^{}_{#1}{\cap}{}_{\Bar{#1}}}
\newcommand\ooCapp[1]{{}_{#1}{\cap} {}_{\Bar{#1}}}
\newcommand\oocup[1] {{}^{\Bar{#1}}{\cup}{}^{#1}}   
\newcommand\oocupp[1]{{}^{#1}{\cup}{}^{\Bar{#1}}}
\newcommand\ood[1] {{}_{\Bar{#1}}{\wedge}{}_{#1}^{}}   
\newcommand\ooD[1] {{}_{\Bar{#1}}{\wedge}{}_{#1}}
\newcommand\oodd[1]{{}^{}_{#1}{\wedge}{}_{\Bar{#1}}}
\newcommand\ooDD[1]{{}_{#1}{\wedge}{}_{\Bar{#1}}}
\def\opp           {^\text{op}}
\def\oti           {\,{\otimes}\,}

\def\otik          {\,{\otimes_\ko}\,}

\def\pb            {{\Bar p}}
\def\piv           {\pi}             

\def\qb            {{\Bar q}}

\def\qquand        {\qquad{\rm and}\qquad}

\def\RR            {{\breve R}}
\def\RSC           {R\S\C}
\def\S             {\I}
\def\SC            {\S\C}
\def\sjs           {$6j$-sym\-bol}

\def\sss           {\scriptscriptstyle}
\newcommand\ssm[1] {\sss \color{\colorFm} #1}
\def\SSS           {\ensuremath{\mathfrak S_3}}
\def\SSSS          {\ensuremath{\mathfrak S_4}}
\newcommand\sumI[1]{\sum_{#1\in\I}}
\newcommand\sumN[4]{\sum_{#1=1}^{\N{#2}{#3}{#4}}}

\def\ta            {x}

\def\To            {\,{\xrightarrow{~}}\,}

\def\tr            {\text{tr}}
\def\Tr            {\text{tr}}

\def\vapi          {\sigma} 
\def\Vect          {\text{vect}}

\def\Vectk         {\ensuremath{\Vect_\ko}}

\def\wee           {\wedge}
\def\Wee           {{}^{\wee\!}}
\def\WEe           {{}^{\wee\!\!}_{}}

\theoremstyle{definition}
\newtheorem{thm}{Theorem}[section]
\newtheorem{conv}[thm]{Convention}
\newtheorem{cor}[thm]{Corollary}
\newtheorem{defi}[thm]{Definition}
\newtheorem{exa}[thm]{Example}
\newtheorem{lem}[thm]{Lemma}

\newtheorem{prop}[thm]{Proposition}

\newtheorem{rem}[thm]{Remark}

\definecolor{DarkViolet} {rgb}{0.580392,0.000000,0.827450}
\definecolor{ForestGreen}{rgb}{0.133333,0.545098,0.133333}

\newcommand{\gray}[1]{{\color{gray}{#1}}}

\catcode`_=\active
\newcommand_[1]{\ensuremath{\sb{{#1}}}}
\catcode`^=\active
\newcommand^[1]{\ensuremath{\sp{{#1}}}}

\begin{document}

\newcommand\Cite[2] {\cite[#1]{#2}}
\numberwithin{equation}{section}
\numberwithin{thm}{section}
                                     

~\vskip 3.3em

\begin{center}
{\bf \Large Tetrahedral symmetry of 6\boldmath{$j$}-symbols in fusion categories}

\vskip 18mm

{\large \  \ J\"urgen Fuchs\,$^{\,a} \quad$ and $\quad$ Tobias Gr\o sfjeld\,$^{\,b}$ }

\vskip 12mm

 \it$^a$
 Teoretisk fysik, \ Karlstads Universitet\\
 Universitetsgatan 21, \ S\,--\,651\,88\, Karlstad
 \\[9pt]
 \it$^b$
 Matematiska institutionen, \ Stockholms Universitet\\
 S\,--\,106\,91\, Stockholm

\end{center}

\vskip 3.2em

\noindent{\sc Abstract}\\[3pt]
We establish tetrahedral symmetries of \sjs s for arbitrary fusion categories under minimal assumptions. As a convenient tool for our calculations we introduce the notion of a \emph{veined} fusion category, which is generated by a finite set of simple objects but is larger than its skeleton. Every fusion category \C\ contains veined fusion subcategories that are monoidally equivalent to \C\ and which suffice to compute many categorical properties for \C.
The notion of a veined fusion category does not assume the presence of a pivotal structure, and thus in particular does not assume unitarity.
We also exhibit the geometric origin of the algebraic statements for the \sjs s.

\newpage
\tableofcontents
\newpage

\section{Introduction}    

Any monoidal category is monoidally equivalent to a strict one, i.e.\ to a monoidal category with trivial associativity and unit constraints. Crucially, under such an equivalence non-trivial information about the tensor product is preserved. Specifically, in case the monoidal category is linear and semisimple with simple monoidal unit and the tensor product of two simple objects is isomorphic to a finite direct sum of simples, the remnants of the associator are encoded in linear isomorphisms 
  \be
  \bigoplus_{p\in\S} \Hm ip{\,l} \oti \Hm jkp \,\xrightarrow{\phantom x}\,
  \bigoplus_{q\in\S} \Hm qk{\,l} \oti \Hm ijq
  \label{eq:1.1}
  \ee
between finite-dimensional vector spaces. Here $\S$ labels the isomorphism classes of simple objects and, with a choice of representatives $(X_i)_{i\in\S}$ in these classes, the vector spaces are the morphism spaces $\Hm ijk \,{=}\, \Hom(X_i\oti X_j,X_k)$ for $i,j,k \,{\in}\, \S$.
 
The \emph{\sjs s}, also called \emph{fusing matrices} \cite{mose3} or \emph{F-matrices} \cite{rosW2}, of such a monoidal category \C\ are the the matrix blocks of the linear isomorphisms \eqref{eq:1.1} with respect to a choice of bases in all spaces $\Hm ijk$ \Cite{Ch.\,4.9}{EGno}. The term \sjs\ is also used for the entries of those matrices. 
\sjs s are of direct relevance in various applications, ranging from recoupling theory in quantum mechanics \cite{FAra} and rational conformal field theory \cite{mose3,fefk3} to state-sum invariants of three-manifolds \cite{tuvi,dujn,bawe2} and the classification of semisimple monoidal categories with prescribed Grothendieck ring \cite{rosW2}.

The numerical values of the \sjs s depend, in general, on the choice of basis, albeit certain specific \sjs s, or combinations of \sjs s, are basis independent. \sjs s can be geometrically interpreted in terms of labeled tetrahedra, see e.g.\ Section VII.1.2 of \cite{TUra}. It is therefore natural to ask whether, or in what sense, the numerical values of \sjs s enjoy the symmetry \SSSS\ of a tetrahedron. Such tetrahedral symmetries have already been discussed a lot in the literature, for instance in the context of state-sum constructions and their Hamiltonian realization in quantum spin systems, see e.g.\ \cite{bawe2,leWe,haWo}.

However, in existing treatments typically restrictions on the class of categories considered are imposed which, while motivated by specific applications the authors have in mind, are not necessary. For instance, the treatment in \cite{bawe2} is in the setting of spherical fusion categories, while \cite{TUra} is in the setting of unimodal modular tensor categories. In applications in condensed matter physics one typically deals with unitary fusion categories (see e.g.\ \cite{kitaA3,szbv}), which in addition are sometimes required to be braided or to be multiplicity-free (i.e.\ $\dim(\Hm ijk) \,{\le}\, 1$).
 
Indeed, to the best of our knowledge, the tetrahedral symmetry of \sjs s has so far not been studied in the what would seem the most natural approach, namely the minimal setting in which the associator can be characterized through the isomorphisms \eqref{eq:1.1}. The purpose of the present paper is to fill this gap. More precisely, our approach is as follows. First of all we impose the indispensable conditions that the monoidal category \C\ in question is \ko-linear, with \ko\ a commutative ring, that it is semisimple and has simple monoidal unit, and that the spaces $\Hm ijk$ are free \ko-modules of finite rank. Further, in order to be able to sensibly manipulate \sjs s, we require that the Grothendieck ring of \C\ has a structure of a \emph{unital based ring} in the sense of Definition 3.1.3 of \cite{EGno}. Finally, we restrict our attention to the situation that \ko\ is actually an algebraically closed field of characteristic zero, and that the number $|\S|$ of isomorphism classes of simple objects is finite. Neither of these latter two restrictions (which are also present in the treatments cited above) are strictly necessary. But they allow us to invoke various pertinent results from the literature, in particular from the theory of fusion categories \cite{EGno}.

Imposing these (almost) minimal requirements leads us to a class of monoidal categories which we call \emph{\gusion}categories (see Definition \ref{def:ourfusion}) and generically denote by \D. A \gusion category is appreciably larger than its skeleton, but its isoclasses of objects are just small enough such that we can address the isomorphisms \eqref{eq:1.1} for a definite set $\S$ of simple objects in terms of a generating skeleton of the category, without any further isomorphisms to be accounted for.
It is worth noting that in this setting we do \emph{not} have to presume that \D\ is rigid. Rather, the role of rigidity is largely taken over by covector duality of finite-dimensional \ko-vector spaces. Provided that a certain non-degeneracy condition is satisfied -- namely, that the specific \sjs s $\Fo i$ are invertible for all $i \,{\in}\, \S$ -- right and left rigidity functors $(-)^\cupp$ and $\Cupp(-)$ for \D\ can be \emph{constructed} from the input data of a \gusion category.
Thereby \D\ naturally acquires the structure of a fusion category. We refer to the resulting class of fusion categories as \emph{\ndgusion}categories, see Definition \ref{def:ndgusion}. 
The so-obtained rigidity functors $(-)^\cupp$ and $\Cupp(-)$ automatically coincide on objects, and the double dual functors $(-)^{\cupp\cupp}$ and ${}^{\cupp\cupp\!}(-)$ are the identity on objects.

\medskip

This paper contains one basic result: Given a \ndgusion category \D\ such that the double dual $(-)^{\cupp\cupp}$ is the identity as a functor, there is a distinguished action of the symmetric group \SSSS\ on the vector space
  \be
  \bigoplus_{i,j,k,p,q,l \in \S} \! \big(
  \Hm jkp \oti \Hd qkl \oti \Hm ipl \oti \Hd ijq \,\oplus\,
  \Hd jkp \oti \Hm qkl \oti \Hd ipl \oti \Hm ijq \big)
  \label{eq:nodef:H4}
  \ee
which results in non-trivial identities for the \sjs s of \D, and thereby also for any fusion category \C\ that is monoidally equivalent to \D. In Definition \ref{def:tauij} this \SSSS-action is introduced in terms of generators. That these indeed obey the relations required for obtaining \SSSS\ is established in Proposition \ref{prop:genuineS4N}. 
 
A geometric interpretation of the so-obtained \SSSS-action on the space \eqref{eq:nodef:H4} in terms of labeled tetrahedra is given in Section \ref{sec:simplicialSSSS}. Those labeled tetrahedra can, in turn, be identified with \sjs s (the precise expressions are given in Equation \eqref{eq:tetrahedra4FandG}). Hereby the \SSSS-action translates to the desired tetrahedral symmetries of \sjs s. Similarly as in \cite{dujn,fgsv}, the resulting relations are obtained as an invariance property of a function $\Ftot$ from the space \eqref{eq:nodef:H4} to the ground field whose values on basis elements are given by rescaled \sjs s, as described in Definition \ref{def:Ftot}. Necessary and sufficient conditions for $\Ftot$ to be \SSSS-invariant are formulated in Theorem \ref{thm:sym} for \ndgusion categories and in Corollary \ref{cor:sym} for general fusion categories.

In the multiplicity-free case, in the literature these tetrahedral symmetries are frequently expressed as equalities between the numerical values of (rescaled) \sjs s. Such equalities are, however, directly valid only under additional conditions. Indeed, according to Definition \ref{def:tauij} the \SSSS-transformations involve a specific \SSS-action on the basic morphism spaces $\Hm ijk$ and their duals, and generically this \SSS-action is non-trivial even when these spaces are one-di\-men\-si\-onal. We specify this \SSS-action in terms of generators in Definitions \ref{def:LR-0} and \ref{def:LL,RR}. That they satisfy the \SSS-relations, under the same conditions as in Proposition \ref{prop:genuineS4N}, is seen in Proposition \ref{prop:genuineS3}.

We stress that unitarity of the fusion category does not play any direct role in our considerations -- after all, the ground field \ko\ can be any algebraically closed field of characteristic zero. Still, unitarity is of interest because, as follows from Lemma \ref{lem:pseudo-etc}(ii), the conditions of Corollary \ref{cor:sym} are fulfilled for every pseudo-unitary \complex-linear fusion category with its canonical spherical structure.
It is also worth pointing out that pivotality of the category is \emph{not} sufficient for the invariance property of $\Ftot$ to hold, compare Lemma \ref{lem:barpi}. On the other hand, even when the $\Ftot$ is not \SSSS-invariant -- and thus in particular, even in the absence of a pivotal structure -- according to the explicit formulas recorded in Proposition \ref{prop:tauij-explN} there are still non-trivial relations among \sjs\ that are connected by tetrahedral transformations. As compared to the case that $\Ftot$ is invariant, they contain additional sign factors. These sign factors have been studied, in a somewhat different setting and for $\ko \eq \complex$, in \cite{bart6}; they are obtained as eigenvalues of certain involutive matrices that are constructed from special \sjs s, see Equations \eqref{eq:involmatrix} and \eqref{eq:MMdiag}.

 \bigskip

To arrive at the results described above, we introduce a few technical tools and, along the way, discuss further aspects of \ndgusion categories, some of which may be of independent interest.
We organize our discussion as follows. We start in Section \ref{sec:Prel} by specifying our setting, including in particular our conventions for \sjs s, the definition of a \ndgusion category and pertinent aspects of covector duality. In Section \ref{sec:semantics4gusion} we summarize the graphical calculus using simplices that we use in our discussion of morphisms and \sjs s. We then explain, in Section \ref{sec:covector2rigid}, the construction of distinguished left and right rigidity functors $(-)^\cupp$ and $\Cupp(-)$ for a \ndgusion category. These allow us in particular to introduce, in Section \ref{sec:dim}, dimension functions. The definition of the functors $(-)^\cupp$ and $\Cupp(-)$ involves a choice of an invertible number $\mu_i \iN \ko$ for each $i \iN \S$; this freedom can be fixed in such a way (see Equation \eqref{eq:mu-choice}) that the left and right dimension functions coincide, but for the sake of generality we refrain from imposing this choice.  

The dimension functions do not rely on the existence of a pivotal structure. Their relation with pivotal dimensions, as well as other aspects of the double dual $(-)^{\cupp\cupp}$ and pivotality, are the subject of Section \ref{sec:piv}. As a final preparation for the discussion of tetrahedral symmetries, in Section \ref{sec:decon} we introduce the maps $\LL$ and $\RR$ which are the generators of the \SSS-action alluded to above. We refer to them as partial duals, owing to the fact that they can be composed so as to yield the rigidity duals of morphisms, as described in Lemma \ref{lem:cupp-va-LLRR}. Finally, Section \ref{sec:SSSS} provides our results on tetrahedral symmetry, as already described above, while in Appendix \ref{sec:app6j} and \ref{sec:app-A} we collect pertinent information about \sjs s and string diagrams, and about structures that underlying our use of simplicial diagrams, respectively.


\section{Preliminaries} \label{sec:Prel}

Let \ko\ be a commutative ring.

\begin{defi}\label{def:ourfusion}
A \textit{\gusion category} is a \ko-linear monoidal category \D\ with a finite set $\S \eq \S_\D$ of pairwise non-isomorphic simple objects such that the following holds:
 \def\leftmargini{2.1em}~\\[-1.65em]\begin{itemize}\addtolength\itemsep{-7pt}
\item[(i)]
The monoidal unit object $1$ is in $S$.
\item[(ii)]
For any pair $i,j \iN \S$, $\HomD(i,j)$ is a free \ko-module of rank $\delta_{i,j}$.
\item[(iii)]
For every $i \iN \S$ there exists a unique $\Bar{i} \iN \S$ such that for any $j \iN \S$, $\HomD(j\oti \Bar i,1)$ and $\HomD(\Bar i\oti j,1)$ are free \ko-modules of rank $\delta_{i,j}$.
\item[(iv)]
Every object is a finite direct sum of tensor products (with arbitrary bracketing) of finitely many objects in $\S$.
\item[(v)]
Every object is isomorphic to a direct sum of objects in $\S$.
\end{itemize}
\end{defi}

This notion of \gusion category makes sense for any commutative ring \ko. However, for the purposes of the present paper we restrict (as is commonly done, e.g.\ in most of Chapter 9 of \cite{EGno})
to the case that \ko\ is an algebraically closed field of characteristic zero. In this case the imposed conditions are redundant, in particular it follows directly from simplicity of $i$ that $\HomD(i,i) \,{\cong}\, \ko$. Still we give the definition as it stands, in order to cover the general case.

\begin{rem} \label{rem:gusion}
(i) By the uniqueness condition on the object $\Bar i$ one has $\Bar{\Bar i} \eq i$ for every $i \iN \S$.
\\[3pt]
(ii) As we take \ko\ to be an algebraically closed field of characteristic zero, an object $x$ is simple iff it is absolutely simple, i.e.\ iff $\HomC(x,x) \,{\cong}\, \ko$. This will no longer be true if one replaces \ko-li\-ne\-a\-rity by super-\ko-linearity, which allows one to consider superfusion categories (as is e.g.\ done in \cite{ushe,bghnprw}). The latter have recently become of interest in the context of fermionic topological orders in condensed matter physics (see e.g.\ \cite{aalW,waGu2}).
\\[3pt]
(iii) Finiteness of the set $\S$ guarantees that summations appearing in our constructions are finite. But provided that the latter summations are still finite, some statements (like the definition of \sjs s) remain valid for infinite $\S$, and thus e.g.\ for representation categories of finite-dimensional complex simple Lie groups.
\end{rem} 

    \bigskip

To explain our choice of terminology, we need to tell what the relevant non-degeneracy condition is. To this end we first recall the concept of \sjs s. Let \C\ be a monoidal category that is monoidally equivalent to a \gusion category. We denote by $\otimes$ the tensor product of \C\ and by $a$ its associativity constraint. For simplicity, and without loss of generality, we take the monoidal unit 1 to be strict.
In the sequel we also abbreviate
  \be
  \HomC(i\oti j,k) =: \Hm ijk \qquad\text{and} \qquad \HomC(k,i\oti j) =: \Hd ijk 
  \label{eq:def:Hm}
  \ee
for every triple $(i,j,k) \,{\in}\, \S^{\times 3}$.
Since \C\ is semisimple (by property (v) in Definition \ref{def:ourfusion}),
the associativity constraint $a$ amounts to a collection of linear isomorphisms $\FM ijkl$ such that the diagram
  \be
  \begin{tikzcd}[ampersand replacement=\&,column sep=5.2em,row sep=2.1em]
  \bigoplus_{p\in\S} \Hm ip{\,l} \otik \Hm jkp 
  \ar{d}[swap]{\circ} \ar{r}{\FM ijkl}[swap]{\cong}
  \& \bigoplus_{q\in\S} \Hm qk{\,l} \otik \Hm ijq \ar[d,"\circ"]
  \\
  \HomC(i\otimes (j \otimes k),l) \ar{r}{\HomC(a_{i,j,k},l)~}[swap]{\cong}
  \& \HomC((i\otimes j) \otimes k,l)
  \end{tikzcd}  
  \label{eq:HHtoHH}
  \ee
commutes for $i,j,k,l \iN \S$.

\begin{defi} \label{def:sjs}
Let \C\ be monoidally equivalent to a \gusion category. The \emph{\sjs s} of \C\ are the matrix blocks of the isomorphisms $\FM ijkl$ with respect to a chosen basis.
\end{defi}

Thus concretely, denoting by $\N ijk \,{:=}\, \dim_\ko(\Hm ijk) \eq \dim_\ko(\Hd ijk)$ the (finite) dimension of the space $\Hm ijk$, and choosing a framing set of basis vectors for each of these spaces, the \sjs s are represented by the collection
  \be
  \big\{ \F ijklpq\alpha\beta\gamma\delta \big\}
  \label{eq:Fnumbers}
  \ee
of numbers that are defined by
  \be
  \alpha \circ (\id_{i} \oti \beta)  \circ a_{i,j,k}
  = \sumI q \sumN \gamma ijq \sumN \delta qkl
  \F ijklpq\alpha\beta\gamma\delta \, \delta \circ (\gamma \oti \id_{k})
  \label{eq:ab=F.cd-}
  \ee
for each quadruple of basis elements
  \be
  \alpha \iN \Hm ip{\!l} \,,\quad \beta \iN \Hm jkp \,,\quad
  \gamma \iN \Hm ijq \,,\quad \delta \iN \Hm qkl \,.
  \ee
Note that for ease of notation here we do not distinguish between a basis morphism $\alpha$ and its `numbering' which is an integer in the range $\{1,2,...\,,\N ijk\}$. Also, by a common slight abuse of terminology, the term \sjs\ will be used both for the matrix blocks of the associativity constraint and for the collection \eqref{eq:Fnumbers} of matrix elements with respect to a given choice of bases of the spaces $\Hm ijk$.

\begin{rem}\label{rem:sjs-conventions}
The convention for the labeling of \sjs s chosen here is taken from \Cite{Eq.\,(2.36)}{fuRs4}. In the literature various other conventions are in use. For instance, the notation in \Cite{Ch.\,4.9}{EGno} is related to ours by $\big( \Phi^l_{ijk} \big)_{pq} \eq \GGm ijklpq$, where $\mathrm G$ denotes the matrix inverse of $\mathrm F$, and the relation with the notation in \cite{ushe} is $\big( F^{ijp,\alpha\beta}_{klq,\gamma\delta} \big)_{\!\text{Usher}} \eq \G ijklpq\alpha\beta\gamma\delta$.
\end{rem}

In the particular case of \sjs s for which $k \eq l \eq i$, $j \eq \Bar i$ and $p \eq q \eq 1$, each of the morphism labels $\alpha,\beta,\gamma,\delta$ can take only a single value. We then often drop those labels from the definition and just write $\Fo i$ for such a \sjs.

\begin{defi}\label{def:ndgusion}
A \textit{\ndgusion category} is a \gusion category \D\ for which the numbers $\Fo i \iN \ko$ are invertible for all $i \,{\in}\, \S_\D$.
\end{defi}

\begin{rem} \label{rem:whyndgusion}
Our terminology is motivated by the fact, to be formulated in Proposition \ref{prop:gusion-fusion}, that any \ndgusion category can be endowed with left and right rigidity structures, whereby it acquires the structure of a fusion category. The qualification \emph{veined} is chosen to indicate that the isomorphism classes of objects are much smaller than for a generic fusion category, but still considerably larger than for a skeletal fusion category.
Conversely, for any \ndgusion category \D\ there is a fusion category \C\ such that \D\ is a full and dense \ko-linear monoidal subcategory of \C. Indeed, given a fusion category \C, select any set $\S$ of representatives for the isoclasses of simple objects of \C\ such that $1 \iN \S$ and take $\D \eq \SC$ to be the category generated by the objects in $\S$ by taking formal finite direct sums and tensor products, see Definition \ref{def:SCetc}(ii).
The involution $i \,{\mapsto}\, \Bar i$ on the set $\S$ of simple objects of the category $\SC$ is induced by the rigidity of \C. 
\end{rem}

For discussing various aspects of \sjs s, it is convenient to make use of the graphical description of morphisms by string diagrams.\,%
 \footnote{~When using string diagrams it is common to replace the monoidal category under consideration by an equivalent strict monoidal category. In our context this is fully justified because the \sjs s retain the essential information about the associativity constraint when replacing a monoidal category by a monoidally equivalent one. Note, however, that as long as one is only dealing with equalities between morphisms, it is not really necessary to work with a strict monoidal category, see e.g.\ Section 2.1 of \cite{bart6},}
This is done in Appendix \ref{sec:app6j}, in which we collect pertinent information about \sjs s. Specifically, the defining formula \eqref{eq:ab=F.cd-} is graphically represented in Equation \eqref{eq:pic:ab=Fcd}.

The numerical values of the \sjs s depend on the arbitrary basis choices for the spaces $\Hm ijk$ or, in other words, on the reference frame or \textit{gauge}. It is important to keep track of the impact that this gauge choice may have in our investigations. Specifically, one can readily determine the effects of arbitrary transformations of the frame upon the \sjs s. For instance, whether or not the numbers $\Fo i$ are invertible is a gauge-independent statement.

  \medskip

Let now \D\ be a \ndgusion category. By semisimplicity of, given any basis $\{ \alpha \} \,{\subset}\, \Hm ijk$ there is a unique \textit{dual basis} $\{ \cc\alpha \} \,{\subset}\, \Hd ijk$ such that
  \be
  \alpha' \circ \cc\alpha = \delta_{\alpha,\alpha'}^{}\, \id_{k} \,.
  \label{eq:dualbasis}
  \ee
Dual bases satisfy the completeness relation
  \be
  \id_{i} \otimes \id_{j} = \sumI k \sumN\alpha ijk\, \cc\alpha \circ \alpha \,.
  \label{eq:dominance}
  \ee
Note that the chosen bases generate all morphisms of \D\ via sums, products and covector dualities. In particular, in the morphism spaces $\HomD(x,i)$ for $i \iN S$ and any $x \iN \D$ we can choose bases that consist of combinations (via tensor product and composition) of the elements of the chosen bases of the spaces $\Hm jkl$. Further, we can then choose coherently dual bases in the spaces $\HomD(i,x)$ by applying the linear \textit{covector duality map} that maps $\Hm jkl \,{\ni}\, \alpha \,{\mapsto}\, \Bar{\alpha} \,{\in}\, \Hd jkl$ to all constituents of each of the basis vectors in the chosen bases of the spaces $\Hm jkl$. From here on we assume that such basis choices have been made. This way the covector duality map extends to $\HomD(x,i) \,{\ni}\, f \,{\mapsto}\, \Bar{f} \,{\in}\, \HomD(i,x)$ for all $i \iN S$ and $x \iN \D$. We can then in particular endow each of these spaces with a non-degenerate \ko-bilinear form $\langle-,-\rangle$ by setting $\langle f,g \rangle \,{:=} f \cir \Bar{g} \iN \HomD(i,i) \,{\cong}\, \ko$ for $f,g \iN \HomD(x,i)$, and similarly for $\HomD(i,x)$.

It is worth noting that the covector duality map does not behave as well under arbitrary gauge transformations as the \sjs s do -- had we chosen $\xi\alpha$ instead of $\alpha$ as a basis vector, for $\xi \iN \ko^\times$, the resulting covector duality map would map $\xi\alpha$ to $\xi^{-1}\Bar{\alpha}$ rather than $\xi\Bar{\alpha}$ and so is a different map. To avoid any complications arising from this feature, from here on
we only admit gauge transformations that preserve the covector duality maps and, as a consequence, preserve the inner products $\langle-,-\rangle$.


\section{Duality in \ndgusion categories}

\subsection{Simplicial semantics for \gusion categories} \label{sec:semantics4gusion}

One of the tools in our study of tetrahedral symmetries of fusion categories will be to manipulate tetrahedra that represent associators. To render such manipulations computationally palatable, we need to know precisely how manipulating a simplex diagram can correspond to an algebraic operation. This is achieved by adopting the simplicial formalism that is outlined in Appendix \ref{sec:app-A} and by working with a \gusion category \D\ in the sense of Definition \ref{def:ourfusion}, so that in particular the class of simple objects is the finite set $\S \eq \S_\D$. For instance, for $\D \eq \SC$ the full and dense linear monoidal subcategory of a fusion category \C\ generated by a choice of representatives for the simple objects, as considered in Remark \ref{rem:whyndgusion}, the simplicial diagrams we draw exist in the single-object quasi-category $Q\SC$ constructed by enriching over $\SC$ (see Definition \ref{def:SCetc}(iii)).
 
The diagrammatic reasoning allows us to make statements that are independent of the choice of basis while at the same time they are easily translated to concrete (possibly basis-dependent) formulas once any specific choice of bases is made. In the sequel we work with a fixed \gusion category \D\ and simplify the notation by just writing $\Hom$ without a subscript for the morphism spaces of \D. 

 \medskip

We take 2-simplices to indicate both specific morphisms and \textit{spaces} of morphisms. This is achieved by using labeled and unlabeled faces for individual morphisms and for morphism spaces, respectively. Specifically, the statement that the morphisms $f$ and $g$ lie in the spaces $\Hom(i\oti j,k)$ and $\Hom(k,i\oti j)$, respectively, is expressed graphically as
  \begin{align}
  \qquad
  \raisebox{-1.6em}{\simptex{
  \ddraw (0,0) to["$k$"] (2,0);
  \ddraw (0,0) to[swap,"$i$"] (1,-1.4);
  \ddraw (1,-1.4) to[swap,"$j$"] (2,0);
  \node at (1,-0.5) {$f$};
  } }
  ~~\in~~
  \raisebox{-1.6em}{\simptex{
  \ddraw (3,0) to["$k$"] (5,0);
  \ddraw (3,0) to[swap,"$i$"] (4,-1.4);
  \ddraw (4,-1.4) to[swap,"$j$"] (5,0);
  } }
  ~=~ \Hom(i\oti j,k) \hspace*{0.52em}
  \\[3pt]
  \text{and} \qquad
  \raisebox{-2.8em}{\simptex{ \internaltwosimplex[][i,j,k][g] } }
  \in
  \raisebox{-2.8em}{\twosimplex[][i,j,k]}
  =~ \Hom(k,i\oti j) \,.
  \end{align}
This convention has the advantage that we can treat the composition of morphisms and of morphism spaces uniformly. On the other hand, as we will only be concerned with the 2-mor\-phisms derived from the associator, the tetrahedra we consider here do not have their \textit{volumes} labeled.

Let us display two examples: First, the composition of $g \,{\in}\, \Hom(i\oti j,p)$ and $f \,{\in}\, \Hom(p,k\oti l)$ is depicted as 
  \begin{align}
  f\circ g ~=
  \raisebox{-3.1em}{\simptex{
  \internaltwosimplex[][k,l,p][f]
  \ddraw (1,-1.4) to node[right,yshift=-2pt]{$j$} (2,0);
  \ddraw (0,0) to node[left,yshift=-4pt]{$i$} (1,-1.4);
  \draw (1.07,-0.7) node{$g$};
  } }
  \in	  
  \raisebox{-3.1em}{\simptex{
  \internaltwosimplex[][k,l,p]
  \ddraw (1,-1.4) to node[right,yshift=-2pt]{$j$} (2,0);
  \ddraw (0,0) to node[left,yshift=-4pt]{$i$} (1,-1.4);
  } }
  \subseteq
  \raisebox{-3.1em}{\simptex{
  \ddraw (0,0) to node[left,yshift=3pt]{$k$} (1,1.4);
  \ddraw (1,1.4) to node[right,yshift=5pt]{$l$}  (2,0);
  \ddraw (1,-1.4) to node[right,yshift=-2pt]{$j$} (2,0);
  \ddraw (0,0) to node[left,yshift=-4pt]{$i$} (1,-1.4);
  } }
  =~ \Hom(i \oti j,k \oti l) \,,
  \end{align}
Second, the isomorphism between the tensor products of morphism spaces that is furnished by the associativity constraint reads
  \begin{align}
  \raisebox{-2.9em}{\simptex{
  \ddraw (0.2,0) to["\ensuremath{l}"] (1.9,0);
  \ddraw (0.2,0) to["\ensuremath{i}" swap] (0.2,-1.4);
  \ddraw (0.2,-1.4) to (1.9,0);
  \ddraw (1.9,-1.4) to["\ensuremath{k}" swap] (1.9,0);
  \ddraw (0.2,-1.4) to["$j$" swap] (1.9,-1.4);
  \draw (1.2,-0.4) node{$p$};
  } }
  ~~ =
  \raisebox{-3.2em}{\simptex[scale=0.5]{\internalthreesimplex[][i,j,k,p,q,l]
  } }
  \hspace*{-0.8em} \circ
  \raisebox{-2.9em}{\simptex{
  \draw (6.6,-0.81) node{$q$};
  \ddraw (5.3,0) to["\ensuremath{i}" swap] (5.3,-1.4);
  \ddraw (5.3,-1.4) to["\ensuremath{j}" swap] (7,-1.4);
  \ddraw (5.3,0) to (7,-1.4);
  \ddraw (5.3,0) to["\ensuremath{l}"] (7,0);
  \ddraw (7,-1.4) to["\ensuremath{k}" swap] (7,0);
  } }
  \label{eq:UnframedF}
  \end{align}

For computational purposes it is often desirable to translate such a picture to coefficient matrices. To account for this we will work with \textit{framed} values of the simplices, for which faces are labeled by framing basis vectors. 
When drawing such framed simplices we employ the following conventions:
  \def\leftmargini{1.05em}~\\[-1.72em]\begin{itemize}
  \addtolength\itemsep{-7pt}
\item[--]  
We use one and the same symbol for a basis vector $\alpha$ and for its dual $\Bar\alpha$, leaving the orientation of the diagram to witness the domain and codomain of the morphism. 
\item[--]  
For the chosen basis of any one-dimensional morphism space we often use the special symbol ``$\oo$'' (compare e.g.\ formula \eqref{eq:FoGo} in Appendix \ref{sec:app6j}), and for each simple object $i \,{\in}\, \S$ we take $\oo \iN \Hom(i,i) \,{\cong}\, \ko$ to be the identity morphism $\id_i$. 
\item[--]  
We omit trivial faces or edges and multiplication symbols when they are already evident from context. 
\item[--]  
Sometimes we abbreviate $\alpha p\beta \equiv \alpha \cir (\id_i \oti \beta)$ when denoting the basis of $\Hom(i\oti (j \oti k),l)$ that is obtained from the basis of $\Hom(i \oti p,l) \otik \Hom(j \oti k,p)$.
\end{itemize}
It is also worth pointing out that we typically suppress summation symbols when working with unframed diagrams. For instance, whilst merely labeling all faces of the picture\eqref{eq:UnframedF} would not result in a sensible equality, it does become sensible once one in addition sums over the index set $\S$ for the edge with label $q$ as well as over the corresponding framings on its adjacent faces. 

Let us provide a few simple examples of framed diagrams. First, the completeness relation \eqref{eq:dominance} for the framings is expressed as
  \begin{align}         
  \sum_{k\in\S}\sum_\alpha \Bar{\alpha} \circ \alpha ~~=~~ \sum_{k,\alpha} \!
  \raisebox{-2.9em}{\begin{tikzpicture}[scale=0.9,line width=\simptexthickness,rounded corners,decoration={markings,mark=at position 0.46 with {\arrowreversed{Stealth[length=6pt,width=6pt]}}}]
  \internaltwosimplex[][i,j,\scriptsize{k}][\alpha]
  \ddraw (1,-1.4) to["\ensuremath{j}" swap] (2,0);
  \ddraw (0,0) to["$i$" swap] (1,-1.4);
  \draw (1,-0.7) node{$\alpha$};
  \end{tikzpicture}
  }
  = \quad
  \raisebox{-2.9em}{\begin{tikzpicture}[scale=0.9,line width=\simptexthickness,rounded corners,decoration={markings,mark=at position 0.46 with {\arrowreversed{Stealth[length=6pt,width=6pt]}}}]
  \ddraw (4,1.4) to["\ensuremath{j}"] (5,0);
  \ddraw (4,-1.4) to["\ensuremath{j}" swap] (5,0);
  \ddraw (3,0) to["$i$" swap] (4,-1.4);
  \ddraw (3,0) to["$i$"] (4,1.4);
  \ddraw (4,-1.4) to["$1$"] (4,1.4);
  \draw (3.5,0) node{$\oo$};
  \draw (4.5,0) node{$\oo$};
  \end{tikzpicture}
  }
  \label{eq:sum-alphabar-alpha}
  \end{align}
Second, the defining relation \eqref{eq:ab=F.cd-} satisfied by the \sjs s $\FF$ can be written as
  \begin{align} 
  \alpha \circ (\id_i \otimes \beta) ~=~
  \raisebox{-2.8em}{\simptex{
  \ddraw (0.2,0) to["\ensuremath{l}"] (1.9,0);
  \ddraw (0.2,0) to["\ensuremath{i}" swap] (0.2,-1.4);
  \ddraw (0.2,-1.4) to (1.9,0);
  \ddraw (1.9,-1.4) to["\ensuremath{k}" swap] (1.9,0);
  \ddraw (0.2,-1.4) to["$j$" swap] (1.9,-1.4);
  \draw (0.65,-0.5) node{$\alpha$};
  \draw (1.2,-0.4) node{$\scriptsize p$};
  \draw (1.5,-0.9) node{$\beta$};
  } }
  \!=~ \sum_{q,\gamma,\delta} \F ijklpq \alpha\beta\gamma\delta~
  \raisebox{-2.8em}{\simptex{
  \draw (5.7,-0.9) node{$\gamma$};
  \draw (6.6,-0.4) node{$\delta$};
  \draw (6.7,-0.95) node{$\scriptsize q$};
  \ddraw (5.3,0) to["\ensuremath{i}" swap] (5.3,-1.4);
  \ddraw (5.3,-1.4) to["\ensuremath{j}" swap] (7,-1.4);
  \ddraw (5.3,0) to (7,-1.4);
  \ddraw (5.3,0) to["\ensuremath{l}"] (7,0);
  \ddraw (7,-1.4) to["\ensuremath{k}" swap] (7,0);
  } }
  \label{eq:defF-simplicial}
  \end{align} 
while the one for the inverse symbols $\GG$ can analogously be expressed as
  \begin{align}
  \alpha \circ (\beta \otimes \id_i) ~=~
  \raisebox{-2.9em}{\simptex{
  \ddraw (0.2,0) to["\ensuremath{l}"] (1.9,0);
  \ddraw (0.2,0) to["\ensuremath{i}" swap] (0.2,-1.4);
  \ddraw (0.2,0) to (1.9,-1.4);
  \ddraw (1.9,-1.4) to["\ensuremath{k}" swap] (1.9,0);
  \ddraw (0.2,-1.4) to["$j$" swap] (1.9,-1.4);
  \draw (0.65,-0.9) node{$\alpha$};
  \draw (1.65,-0.95) node{$\scriptsize p$};
  \draw (1.5,-0.4) node{$\beta$};
  } }
  \!\! =~ \sum_{q,\gamma,\delta} \G ijklpq \alpha\beta\gamma\delta\,
  \raisebox{-2.9em}{\simptex{
  \draw (5.7,-0.5) node{$\gamma$};
  \draw (6.6,-0.9) node{$\delta$};
  \draw (6.3,-0.4) node{$\scriptsize q$};
  \ddraw (5.3,0) to["\ensuremath{i}" swap] (5.3,-1.4);
  \ddraw (5.3,-1.4) to["\ensuremath{j}" swap] (7,-1.4);
  \ddraw (5.3,-1.4) to (7,0);
  \ddraw (5.3,0) to["\ensuremath{l}"] (7,0);
  \ddraw (7,-1.4) to["\ensuremath{k}" swap] (7,0);
  } }
  \label{eq:defG-simplicial}
  \end{align}
Finally, note that we can interpret the equality \eqref{eq:defF-simplicial} as a framed version of the associativity constraint \eqref{eq:UnframedF}. Accordingly, the \sjs s can be thought of as the matrix coefficients of the 2-morphism that is realized by the tetrahedron in \eqref{eq:UnframedF}, and analogously for the inverse \sjs s. This amounts to the identifications\,%
 \footnote{~Put differently, the proper way to read the labeled tetrahedra in \eqref{eq:tetrahedra4FandG} is as the matrix coefficients of the abstract tetrahedra with spine $ijk$. Note that these can either be interpreted as a passive basis transformation of spaces or as an active transformation of framed spaces. Computationally, we can glue the in-faces of a tetrahedron to the appropriate input and use the completeness relation to erase them from a tetrahedron, leaving the out-faces as output, as in Equation $\eqref{eq:UnframedF}$. Alternatively, we can think of the tetrahedron as expanding the out-faces in terms of the in-faces.}
  \begin{align} 
  \raisebox{-4.0em}{
  \begin{tikzpicture}[scale=0.6,line width=\simptexthickness,rounded corners,decoration={markings,mark=at position 0.46 with {\arrowreversed{Stealth[length=6pt,width=6pt]}}}]
  \internalthreesimplex[][i,j,k,p,q,l][\gray{\alpha},\gray{\beta},\gamma,\gray{\delta}]
  \end{tikzpicture} }
  \hspace*{-0.6em} =~~ \F ijklpq \alpha\beta\gamma\delta \qquad \text{and} ~~
  \raisebox{-4.0em}{
  \begin{tikzpicture}[scale=0.6,line width=\simptexthickness,rounded corners,decoration={markings,mark=at position 0.46 with {\arrowreversed{Stealth[length=6pt,width=6pt]}}}]
  \internalGsimplex[][i,j,k,p,q,l][\alpha,\beta,\gray{\gamma},\delta]
  \end{tikzpicture} }
  \hspace*{-0.6em} =~~ \G ijklqp \gamma\delta\alpha\beta \,.
  \label{eq:tetrahedra4FandG}
  \end{align} 
This is the simplicial counterpart of the string diagram expressions \eqref{eq:F=o.abcd.o,} and \eqref{eq:G=o.abcd.o,} for the \sjs s.
Note that these coefficient values of $\FF$ and $\GG$ are independent of whether one uses basis or dual basis labels for the faces, since dualizing such a label at the same time reverses the orientation of the respective face.


\subsection{From covector duality to rigidity} \label{sec:covector2rigid}

The considerations in Section \ref{sec:semantics4gusion} apply to arbitrary \gusion categories. We now study issues for which we will have to restrict to \emph{non-degenerate} ones, i.e.\ to \ndgusion categories.

Let us introduce a special notation for the pair of dual basis vectors in the one-dimensional morphism spaces $\Hom(1,\Bar{i}\oti i)$ and $\Hom(\Bar{i}\oti i, 1)$ for $i \,{\in}\, \S$:
  \begin{align}
  \oob i \in \Hom(1,\Bar{i}\oti i) \qquad \text{and} \qquad \ood i \in \Hom(\Bar{i}\oti i,1) \,.
  \label{eq:notation:oob-ood}
  \end{align}
Besides the covector duality relation $\ood i \,{\circ}\, \oob i \,{=}\, 1$ these morphisms satisfy
  \begin{align}
  & (\oodd i \oti \id_i) \circ (\id_i \oti \oob i) = \Go i\, \id_i
  \\[4pt]
  \text{and} \quad & (\id_i \oti \ood i) \circ (\oobb i \oti \id_i) = \Fo i\, \id_i \,,
  \end{align}
with $\Go i$ the inverse \sjs\ analogous to $\Fo i$ (compare \eqref{eq:FoGo}), which according to \eqref{eq:pic:FkkkkGkkkk} satisfies $\Go i \eq \Fob i$.
In particular, in the special case that all coefficients $\Fo i$ are equal to 1, we simply have
  \begin{align}
  (\oodd i \oti \id_i) \circ (\id_i \oti \oob i) = \id_i
  = (\id_ \oti \ood i) \circ (\oobb i \oti \id_i) \,.
  \end{align}
When considered for both $i$ and $\ib$, these equalities can be recognized as both the right and left rigidity constraints, or \textit{snake identities}, for the simple object $i$, with the left and right dual object of $i$ being $\Bar i$. Moreover, we can actually extend this structure to left and right rigidity structures $\Wee(-)$ and $(-)^\wee$ on the category \D: On objects we define $(i \oti j)^\wee \,{:=}\, j^\wee \oti i^\wee \,{\equiv}\, \Bar j \oti \Bar i$ as well as $\Wee(i \oti j) \,{:=}\, \Bar j \oti \Bar i$, and analogously for multiple tensor products and for direct sums; thus the left and right duals of any object $x \iN \D$ coincide, and accordingly we just write both of them as $\Bar x$. We then set
  \begin{align}
  \ood{i\otimes j} := \ooD{j} \circ (\id_{\Bar j}\oti \ood i \oti \id_j)
  \label{eq:ood(i.j)}
  \end{align}
etc., and for any pair of composable morphisms $\!\!\begin{tikzcd} x \ar[r,"f"] & y \ar[r,"g"] & z\end{tikzcd}\!\!$ we set
  \begin{align}
  f^\wee &:= (\ooD y \oti \id_{\Bar{x}}) \circ (\id_{\Bar{y}} \oti f \oti \id_{\Bar{x}}) \circ (\id_{\Bar{y}} \oti \oobb x)
  \\[4pt]
  \text{and} \qquad  \Wee f &:= (\id_{\Bar{x}} \oti \ooDD y) \circ (\id_{\Bar{x}} \oti f \oti \id_{\Bar{y}}) \circ (\oob x \oti \id_{\Bar{y}}) \,.
  \label{eq:defdualmorphs-F=1}
  \end{align}
It is then immediately checked that the snake identities are fulfilled for every object $x$, as well as $(g{\circ}f)^\wee \,{=}\, f^\wee \,{\circ}\, g^\wee$ and $\Wee(g{\circ}f) \,{=}\, \WEe f \,{\circ}\, \Wee g$.
Also note that the double duals act on objects as the identity, $x^{\wee\wee} \eq x \eq {}^{\wee\wee\!}x$.

The assumption that all $\Fo i$ are equal to 1 is a strong restriction on the category \C, though. If this requirement is not satisfied, then the associator term breaks functoriality and hence also rigidity. Indeed, the definitions \eqref{eq:defdualmorphs-F=1} immediately lead to $\id^\wee_i \,{=}\, \Go i\,\id_{\Bar i}$, and e.g.\ for the basis element $\alpha p\beta$ on the \lhs\ of \eqref{eq:defF-simplicial} they give (using also \eqref{eq:pic:FkkkkGkkkk}) $(\id_i \oti \beta)^\wee \,{\circ}\, \alpha^\wee \,{=}\, \Fo p\, \Fo i\, (\alpha \,{\circ}\, (\id_i\oti \beta))^\wee$.

However, in case the \gusion category \D\ is non-degenerate, i.e.\ is a \ndgusion category in the sense of Definition \ref{def:ndgusion}, then owing to the invertibility of the numbers $\Fo i \iN \ko$ this problem can be resolved by rescaling: We start with

\begin{defi}\label{def:cupcap}
Let \D\ be a \ndgusion category with generating set $\S$, and let $\oobb i$ and $\oodd i$ be the basis morphisms \eqref{eq:notation:oob-ood}. We set
  \begin{align}
  \oocup i := \mu_i\, \oob i  
  \qquad \text{and} \qquad \oocapp i := \frac1{\mu_i\, \Fob i}\, {\oodd i}
  \label{eq:def:cup}
  \end{align}
for $i \,{\in}\, \S$, with arbitrary invertible scalars $\mu_i$. 
\end{defi}

These definitions can be extended to all objects of \D\ by the analogous prescription as in \eqref{eq:ood(i.j)}. We then use these rescaled morphisms to define candidate rigidity functors on morphisms as
  \begin{align}
  & f^\cupp := (\ooCap y \oti \id_{\Bar{x}}^{}) \circ (\id_{\bar{y}} \oti f \oti \id_{\Bar{x}}) \circ (\id_{\Bar{y}} \oti \oocupp x)
  \nonumber
  \\[3pt]
  \text{and} \qquad & \CUpp f := (\id_{\Bar{x}}^{} \oti \ooCapp y) \circ (\id_{\Bar{x}} \oti f \oti \id_{\Bar{y}}) \circ (\oocup x \oti \id_{\bar{y}})
  \label{eq:def:fcupp}
  \end{align}
for $f \,{\in}\, \Hom(x,y)$. We then immediately satisfy the snake identities for all objects and have $(g{\circ}f)^\cupp \,{=}\, f^\cupp \,{\circ}\, g^\cupp$ and $\Cupp (g{\circ}f) \,{=}\, \Cupp f \,{\circ}\, \Cupp g$ as well as $(\id_x)^\cupp \,{=}\, \id_{\Bar x} \,{=}\, \Cupp(\id_x)$ for all $x \,{\in}\, \D$.
Also note that 
  \be
  (\Cupp f)^\cupp = f \,{=}\, \Cupp(f^\cupp)
  \label{ufu=f}
  \ee
on the nose. We can also extend the rescaling factors to all objects of \D, by defining iteratively $\Fx{x\otimes y} \,{:=}\, \Fx x\,\Fx y$ with $\Fx i \,{\equiv}\, \Fo i$ and $\Fx{x\oplus y} \,{:=}\, \Fx x \,{\oplus}\, \Fx y$ (as diagonal matrices). We then have
  \begin{align}
  f^\cupp = {(\Fx y)}^{-1} f^\wee \qquad \text{and} \qquad 
  \CUpp f = {(\Fx {\Bar{y}})}^{-1}\, \WEe f 
  \end{align}
for all morphisms $f \,{\in}\, \Hom(x,y)$.

We summarize these observations as 

\begin{prop}\label{prop:gusion-fusion}
Let \D\ be a \ndgusion category. Then the functors $(-)^\cupp$ and $\Cupp(-)$ from \D\ to $\D\opp$ defined on objects by $x^\cupp \,{:=}\, \Bar x \,{=:}\, \Cupp x$ and on morphisms by \eqref{eq:def:fcupp} constitute left and right rigidities on \D, respectively. This endows \D\ with the structure of a fusion category.
\end{prop}

Recall that in any monoidal category the (left or right) dual of an object, if it exists, is unique up to unique isomorphism \Cite{Prop.\,2.10.5}{EGno}. In a \ndgusion category, the freedom given by these isomorphisms for all objects reduces to automorphisms of the simple objects in $\S$. The invertible numbers $\mu_i$ introduced in \ref{def:cupcap} precisely account for this freedom.
Also, we immediately have (compare also \Cite{Prop.\,5.3.13}{BAki})

\begin{cor}
A \gusion category is rigid, and is thus a fusion category, if and only if it is non-degenerate.
\end{cor}


\subsection{Traces and dimensions} \label{sec:dim}

{}From now on we assume that \D\ is a \ndgusion category, so that left and right rigidity functors as constructed above exist. It is then natural to ask whether \D\ admits a \emph{pivotal structure}, i.e.\ 
a monoidal natural transformation $\piv\colon (-)^{\cupp\cupp} \To \Id$ from the (right, say) double rigidity dual to the identity functor; this question will be studied in Section \ref{sec:piv} below. The datum of a pivotal structure is equivalent to the one of a \emph{sovereign} structure, i.e.\ a monoidal natural transformation between left and right rigidities. Note that the evaluation and coevaluation morphisms \eqref{eq:def:cup} appear in the definition \eqref{eq:def:fcupp} of both left and right rigidity. This makes it easy to handle composite morphisms that involve both rigidities, even in the absence of a pivotal structure.\,%
 \footnote{~Thus in case $\ko \eq \complex$ the left and right rigidities are automatically related according to the ``pairing convention'' that is explained in Section 3.2 of \cite{bart6}.}
In particular we can directly define two traces as follows.

\begin{defi} \label{def:traces}
For any endomorphism $f \,{\in}\, \Hom(x,x)$ in \D\ we set
  \be
  \Tr_L(f) ~:=~~
  \raisebox{-3.2em}{\begin{tikzpicture}[line width=\simptexthickness,rounded corners,decoration={markings,mark=at position 0.46 with {\arrowreversed{Stealth[length=6pt,width=6pt]}}}]
  \ddraw (-1,-1.4) to (0,0);
  \ddraw (-1,1.4) to (0,0);
  \draw (-0.5,0) node{$f$};
  \draw (0.5,0) node{$\oo$};
  \draw (0,.85) node{$\ooCapp x$};
  \draw (0,-.85) node{$\oocupp x$};
  \draw (-1,-1.4) to[dotted] (-1,1.4);
  \ddraw (0,0) to (1,1.4);
  \ddraw (0,0) to (1,-1.4);
  \draw (1,-1.4) to[dotted] (-1,-1.4);
  \draw (1,-1.4) to[dotted] (1,1.4);
  \draw (-1,1.4) to[dotted] (1,1.4);
  \end{tikzpicture}
  }
  \qquad~ \text{and} \qquad \Tr_R(f) ~:=~~
  \raisebox{-3.2em}{\begin{tikzpicture}[line width=\simptexthickness,rounded corners,decoration={markings,mark=at position 0.46 with {\arrowreversed{Stealth[length=6pt,width=6pt]}}}]
  \ddraw (-1,-1.4) to (0,0);
  \ddraw (-1,1.4) to (0,0);
  \draw (0.5,0) node{$f$};
  \draw (-0.5,0) node{$\oo$};
  \draw (0,.85) node{$\ooCap x$};
  \draw (0,-.85) node{$\oocup x$};
  \draw (-1,-1.4) to[dotted] (-1,1.4);
  \ddraw (0,0) to (1,1.4);
  \ddraw (0,0) to (1,-1.4);
  \draw (1,-1.4) to[dotted] (-1,-1.4);
  \draw (1,-1.4) to[dotted] (1,1.4);
  \draw (-1,1.4) to[dotted] (1,1.4);
  \end{tikzpicture}
  } \quad
  \ee
The numbers $\Tr_L(f)$ and $\Tr_R(f)$ are called the \textit{left trace} and \textit{right trace}, respectively, of the endomorphism $f$.
\end{defi}

In order that the traces $\Tr_L(f)$ and $\Tr_R(f)$ become cyclic, a pivotal structure on \D\ is needed. But note that owing to the fact that the double dual functors ${}^{\cupp\cupp\!}(-)$ and $(-)^{\cupp\cupp}$ act trivially on objects, the traces are indeed defined on \textit{endo}morphisms. As a consequence, even without assuming a pivotal structure 
we can introduce dimension functions by considering the traces of identity morphisms:

\begin{defi}
The left and right \textit{dimensions} of an object $x \,{\in}\, \D$ are the numbers
  \be
  \dim_L(x) := \Tr_L(\id_x) \qquad\text{and}\qquad \dim_R(x) := \Tr_R(\id_x) \,,
  \label{eq:def:dimLR}
  \ee
respectively.
\end{defi}

\begin{rem}
Since the traces do not involve a pivotal structure, they are not cyclic on the nose, but rather satisfy 
  \be
  \Tr_L(f \cir g) = \Tr_L(g \cir {}^{\cupp\cupp\!}f) \qquand 
  \Tr_R(f \cir g) = \Tr_R(g \cir f^{\cupp\cupp})
  \label{eq:not-so-cyclic}
  \ee
(compare \Cite{Prop.\,4.7.3(4)}{EGno}).
In contrast to the traces, for a general fusion category dimensions can only be defined in the pivotal case \Cite{Def.\,4.7.11}{EGno}. But since $i^{\cupp\cupp} \eq i$, for a \ndgusion category \D\ we can still define non-pivotal dimensions as in \eqref{eq:def:dimLR}. In case \D\ does have a pivotal structure $\piv$, these must be distinguished from the `pivotal' dimensions 
  \be 
  \dim_L^\piv(i) := \Tr_L(\pi_i^{-1}) 
  \qquand
  \dim_R^\piv(i) := \Tr_R(\pi_i^{})   
  \,.
  \label{eq:def:dimLRpiv}
  \ee 
The latter functions are also called `quantum dimensions' or sometimes `categorical dimensions', which hides their dependence on the pivotal structure.
\end{rem}
		 
Using \eqref{eq:def:cup} we directly get
  \begin{align}
  \dim_L(i) = \frac{\mu_{\Bar{i}}}{\mu_{i}\, \Fob i} \qquad \text{and}\qquad
  \dim_R(i) = \dim_L(\Bar i) \,.
  \label{eq:dimL,dimR}
  \end{align}
Note that by construction the dimensions of simple objects are invertible numbers,
and that $\dim_{L,R}(\OO) \eq 1$.
Also, while the left and right dimensions individually depend on the choice of the numbers $\mu_i$ (and are thus not canonical), their product does not: we have
  \be
  \dim_L(i)\,\dim_R(i) = \frac1{\Fo i\,\Fob i} \,.
  \label{eq:dimL.dimR}
  \ee
For fusion categories over \complex, such a product of left and right dimensions has first been studied in \Cite{Prop.\,2.4}{muge8}. It is also known as \textit{squared norm} \Cite{Sect.\,7.21}{EGno} or as \textit{paired dimension} \Cite{Def.\,3.2}{bart6} of the object $i$. We will use the latter term; the formula \eqref{eq:dimL.dimR} amounts to Corollary 4.7 of \cite{bart6}.
Besides being independent of the $\mu_i$, the paired dimension is also gauge-independent, and by a suitable rescaling one can achieve $\Fo i\eq\Fob i$ (see Remarks \ref{rem:gauge} and \ref{rem:rescaledFi} in Appendix \ref{sec:app6j}), whereby the paired dimension is written as a square.
We do not introduce a separate notation for the paired dimension, but it will be convenient to use a separate symbol for the quotient of the left and right dimensions of an object, which we also call the \emph{relative dimension}:
  \be
  \dd x := \dim_L(x) / \dim_R(x) \,.
  \label{eq:def:dd}
  \ee
In the sequel we also abbreviate
  \be
  \di x := \dim_L(x) \,.
  \label{def:di}
  \ee

The left and right traces are very similar. Indeed, using the relation \eqref{eq:def:cup} between covector and rigidity dualities, one checks that $\Tr_L(f)/\di i\di j \,{=}\, \Tr_R(f)/\di\ib \di\jb$ for $f \,{\in}\, \End(i\oti j,i\oti j)$, and analogously for general endomorphisms. 
In particular the relative dimensions obey
  \be
  \dd i \equiv \frac{\di i}{\di\ib} = \frac {\Fo i} {\Fob i}\, \Big(\frac{\mu_\ib}{\mu_i} \Big)^{\!2} .
  \label{eq:reldim}
  \ee
Thus if and only if the freedom present in the scalars $\mu_i$ is partially fixed in such a way that
  \be
  \frac {\mu_i^2}  {\mu_\ib^2} = \frac {\Fo i} {\Fob i} \,,
  \label{eq:mu-choice}
  \ee
then the two dimension functions as well as the two traces coincide.


\subsection{Pivotality} \label{sec:piv}

Next we ask under which conditions the category \D\ has a pivotal structure. To this end we first compute the expansion coefficients of the double duals of basis elements $\alpha \iN \Hm ijk$ and $\cc\alpha \iN \Hd ijk$ with respect to the chosen bases. We find

\begin{lem} \label{lem:1piv}
(i) The right and left double duals of the basis morphisms in the spaces $\Hm ijk$ and $\Hd ijk$ can be expressed as
  \begin{align}
  & \alpha^{\cupp\cupp} 
  = \frac{\di i\,\di j}{\di k} \sum_\beta \MM_{\alpha,\beta}\, \beta
  \qquad \text{and} \qquad
  \cc\alpha^{\cupp\cupp} 
  = \frac{\di\ib\,\di\jb}{\di\kb} \sum_\beta \MM_{\beta,\alpha}\, \cc\beta
  \label{eq:alphabarcuppcupp}
  \end{align}
and as
  \begin{align}
  & {}^{\cupp\cupp\!}\alpha
  = \frac{\di\ib\,\di\jb}{\di\kb} \sum_\beta \MM_{\alpha,\beta}\, \beta
  \qquad \text{and} \qquad
  {}^{\cupp\cupp}\cc\alpha 
  = \frac{\di i\,\di j}{\di k} \sum_\beta \MM_{\beta,\alpha}\, \cc\beta \,, 
  \label{eq:cuppcuppalphabar}
  \end{align}
respectively, where $\MM \,{\equiv}\, \MM^{(i\,j\,k)}$ is the $\N ijk{\times}\N ijk$-matrix with entries
  \be
  \MM^{(i\,j\,k)}_{\alpha,\beta}
  := \sum_\mu \F ij\jb i\OO k \oo\oo\beta\mu\, \G ij\jb ik\OO \alpha\mu\oo\oo
  = \sum_\mu \G \ib ijj\OO k \oo\oo\mu\beta\, \F \ib ijjk\OO \mu\alpha\oo\oo \,.
  \label{eq:MM2FGversions}
  \ee
(ii) $\cc\alpha^{\cupp\cupp}$ is covector dual to $\alpha^{\cupp\cupp}$, and  ${}^{\cupp\cupp}\cc\alpha$ is covector dual to ${}^{\cupp\cupp}\alpha$.
\\[5pt]
(iii) The quadruple dual satisfies
  \be
  \alpha^{\cupp\cupp\cupp\cupp}
  = \dd i\, \dd j \, \dd\kb \; \alpha
  \label{eq:quadrupledual-id}
  \ee
with $\dd x$ the relative dimension \eqref{eq:def:dd},
and analogously for ${}^{\cupp\cupp\cupp\cupp}\alpha$. In particular, if left and right dimensions coincide, then the (right or left) quadruple dual functor is the identity as a functor.
\end{lem}

\begin{proof}
(i) We first note that by combining the Definitions \ref{def:cupcap} and \ref{def:traces} we have
  \begin{align}
  & \alpha^{\cupp\cupp} 
  = \frac1{\di k} \sum_\beta \Tr_L (\Bar\beta \,{\circ}\, \alpha)\, \beta
  \label{eq:alphacuppcupp2}
  \\
  \text{and} \qquad
  & \cc\alpha^{\cupp\cupp} = \frac1{\di\kb} \sum_\beta \Tr_R (\beta \,{\circ}\, \cc\alpha^{\cupp\cupp})\, \cc\beta
  \equ{eq:not-so-cyclic} \frac1{\di\kb} \sum_\beta \Tr_R (\cc\alpha \,{\circ}\, \beta)\, \cc\beta
  \label{eq:alphabarcuppcupp2}
  \end{align}
as well as
  \begin{align}
  & {}^{\cupp\cupp\!}\alpha
  = \frac1{\di\kb} \sum_\beta \Tr_R (\Bar\beta \,{\circ}\, \alpha)\, \beta
  \label{eq:cuppcuppalpha2}
  \\
  \text{and} \qquad
  & {}^{\cupp\cupp}\cc\alpha = \frac1{\di k} \sum_\beta \Tr_L (\beta \,{\circ}\, {}^{\cupp\cupp}\cc\alpha)\, \cc\beta
  \equ{eq:not-so-cyclic} \frac1{\di k} \sum_\beta \Tr_L (\cc\alpha \,{\circ}\, \beta)\, \cc\beta \,.
  \label{eq:cuppcuppalphabar2}
  \end{align}
On the other hand, the matrix $\MM$ can be expressed in terms of string diagrams as in \eqref{eq:MM3versions}, from which it directly follows that 
  \be
  \MM_{\alpha,\beta} = \frac1{\di i\,\di j} \, \Tr_L (\Bar\beta \,{\circ}\, \alpha) 
  = \frac1{\di\ib\,\di\jb} \, \Tr_R (\Bar\beta \,{\circ}\, \alpha) \,.
  \ee
Note that the equalities of string diagrams in \eqref{eq:MM3versions} imply in particular shows the equality of the two expressions for $\MM$ in \eqref{eq:MM2FGversions}.
\\[3pt]
(ii) {}By (i) we have in particular
  \be 
  \frac{\di k}{\di i\,\di j}\, \alpha^{\cupp\cupp} = \frac{\di\kb}{\di\ib\,\di\jb}\, {}^{\cupp\cupp\!}\alpha
  \label{eq:cuppcupp:l-vs-r}
  \ee
(and thus $\alpha^{\cupp\cupp} \,{=}\, {}^{\cupp\cupp\!}\alpha$ in case the left and right dimensions coincide).
It follows further that, up to a prefactor $\di k\di\kb/\di i\di\ib\di j\di\jb$, the morphism ${}^{\cupp\cupp\!}\alpha^{\cupp\cupp}$ is obtained from $\alpha$ by applying the matrix $\MM^2$. However, we already know from \eqref{ufu=f} that ${}^{\cupp\cupp\!}\alpha^{\cupp\cupp} \,{=}\, \alpha$. Hence we learn that the matrix with entries
  \be
  \sqrt{\frac{\di i\, \di\ib\, \di j\, \di\jb}{\di k\, \di\kb}}\, \MM^{(i\,j\,k)}_{\alpha,\beta}
  \stackrel{\eqref{eq:dimL.dimR},\eqref{eq:MM2FGversions}}=
  \sqrt{ \frac{\Fo k\,\Fob k} {\Fo i\,\Fob i\, \Fo j\,\Fob j} }\,
  \sum_\mu \F ij\jb i\OO k \oo\oo\beta\mu\, \G ij\jb ik\OO \alpha\mu\oo\oo
  \label{eq:involmatrix}
  \ee
is involutive.
As a consequence, the relations \eqref{eq:alphabarcuppcupp} imply that
  \be
  \alpha^{\cupp\cupp} \,{\circ}\, \cc\beta^{\cupp\cupp} 
  = \frac{\di i\,\di j}{\di k}\, \frac{\di\ib\,\di\jb}{\di\kb} \sum_{\gamma,\delta}
  \MM_{\alpha,\gamma}\, \MM_{\delta,\beta}\, \gamma\,{\circ}\, \cc\delta
  = \frac{\di i\,\di\ib\,\di j\,\di\jb}{\di k\,\di\kb} \big(\MM^2\big)_{\alpha,\beta}
  = \delta_{\alpha,\beta} \,.
  \label{eq:doubledual-vs-covector}
  \ee
The statement for the double left duals is seen in the same way.
\\[3pt]
(iii) follows directly from \eqref{eq:alphabarcuppcupp} via the last equality in \eqref{eq:doubledual-vs-covector}.
\end{proof}

Since the matrix \eqref{eq:involmatrix} is involutive, it can be diagonalized with eigenvalues $\pm1$. We refer to the basis consisting of the corresponding eigenvectors briefly as an \emph{eigenbasis} of the morphism space $\Hm ijk$. In an eigenbasis the diagonal elements of $\MM$ read
  \be
  \MM_{\alpha,\alpha} = \sqrt{\frac{\di k\, \di\kb}{\di i\, \di\ib\, \di j\, \di\jb}} \, \eps_\alpha
  \qquad \text{with} \qquad \eps_\alpha \,{\equiv}\, \eps_{i,j;\alpha}^{~k} \in \{\pm1\}
  \,.
  \label{eq:MMdiag}
  \ee

\begin{rem}
(i) For fusion categories over $\ko \eq \complex$, the linear automorphism realized by the involutive matrix \eqref{eq:involmatrix} has already been used in the proof of Theorem 2.3 of \cite{etno}. Thus in this case the matrix coincides with what is called a \emph{pivotal operator} in Definition 3.7 of \cite{bart6}. Accordingly, Lemma \ref{lem:1piv}(iii) amounts to Theorem 3.10 of \cite{bart6} (compare also \Cite{Thm.\,3}{haHon}). The eigenvalues $\eps_\alpha$ appear in \cite{bart6} as \emph{pivotal symbols}, with the formula \eqref{eq:MMdiag} corresponding to Lemma 3.15 \cite{bart6}.
 \\[3pt]
(ii) In the square-root expressions in \eqref{eq:involmatrix} and \eqref{eq:MMdiag} (as well as below, e.g.\ in \eqref{eq:dd:sumrule}), the paired dimension \eqref{eq:dimL.dimR} appears. 
As already noted, the paired dimension can be written explicitly as a square if the basis choice described in Remark \ref{rem:rescaledFi} is made. In this case a natural choice of square root is to take $\sqrt{\di i \, \di\ib} \eq 1/\Fo i \eq 1/\Fob i$ for every $i\iN S$.
In case $\ko \eq \complex$, another possible prescription is to take the positive square root; this is done in \cite{bart6}.
If some of the numbers $\Fo i$ are negative, the two prescriptions are different. In particular, also the values of the sign factors $\eps_{i,j;\alpha}^{~k}$ differ, but only by an ($\alpha$-independent) coboundary, so that various results involving these sign factors, like e.g.\ Lemma \ref{lem:barpi} below, are not affected.
\end{rem}

The following observations are now immediate:

\begin{lem} \label{lem:epsalpha}
(i) The basis $\{\alpha\}$ of $\Hm ijk$ can be chosen in such a way that the left and right double duals of the basis vectors are given by
  \be
  \alpha^{\cupp\cupp} = \eps_\alpha\,
  \sqrt{ \dd i\, \dd j \, \dd\kb } \, \alpha
  \qquad\text{and}\qquad
  {}^{\cupp\cupp}\alpha  = \eps_\alpha\,
  \sqrt{ \dd\ib \, \dd\jb \, \dd k } \, \alpha \,,
  \label{eq:doubleduals-diag}
  \ee
respectively, with $\eps_\alpha \iN \{\pm1\}$.
 \\[3pt]
(ii) With this choice, the elements of the covector-dual basis of $\Hd ijk$ satisfy
  \be
  \cc\alpha^{\cupp\cupp} = \eps_\alpha\,
  \sqrt{ \dd\ib \, \dd\jb \, \dd k } \, \cc\alpha
  \qquad\text{and}\qquad
  {}^{\cupp\cupp}\cc\alpha  = \eps_\alpha\,
  \sqrt{ \dd i\, \dd j \, \dd\kb } \, \cc\alpha \,.
  \label{eq:doubleduals-diag-bar}
  \ee
In short, we can write $\eps_{\cc\alpha} \eq \eps_\alpha$.
 \\[3pt]
(iii) For the standard basis elements $\id_i \iN \Hm i0i$, $\id_i \iN \Hm 0ii$ and 
$\oodd i \iN \Hm i\ib 0$ we have
  \be
  \eps_{\id_i} = +1 = \eps_{\oodd i} \,.
  \label{eq:eps-oodd=1}
  \ee
(iv) For every quadruple $\alpha \iN \Hm ipl$, $\beta \iN \Hm jkp$, $\gamma \iN \Hm ijq$, $\delta \iN \Hm qkl$ the implication
  \be
  \F ijklpq \alpha\beta\gamma\delta \ne 0 \,~\Longrightarrow~
  \eps_\alpha\, \eps_\beta = \eps_\gamma\, \eps_\delta
  \label{eq:epseps=epseps}
  \ee
holds.
 \\[3pt]
(v) The paired dimensions satisfy the sum rule
  \be
  \sum_{k\in\S} \sqrt{\di k^{}\,\di\kb}\, \sumN\alpha ijk \eps_{i,j;\alpha}^{~k}
  = \sqrt{\di i^{}\, \di\ib\, \di j^{}\, \di\jb}
  \label{eq:dd:sumrule}
  \ee
for $i,j \iN \S$.
\end{lem}

\begin{proof}
When taking the basis of $\Hm ijk$ to be an eigenbasis, the equalities \eqref{eq:doubleduals-diag} and \eqref{eq:doubleduals-diag-bar} follow directly from Lemma \ref{lem:1piv}(i). The equalities \eqref{eq:eps-oodd=1} hold because we have, trivially, $(\id_i)^{\cupp\cupp} \eq \id_i$ as well as, as seen by direct calculation, $(\oodd i)^{\cupp\cupp} \eq \oodd i$ for every $i \iN \S$.
 \\[3pt]
The formula \eqref{eq:epseps=epseps} is obtained by specializing the equality
  \be
  \alpha^{\cupp\cupp} \circ (\id_{i} \oti \beta^{\cupp\cupp})  \circ a_{i,j,k}
  = \sumI q \sum_{\gamma,\delta} \F ijklpq\alpha\beta\gamma\delta \,
  \delta^{\cupp\cupp} \circ (\gamma^{\cupp\cupp} \oti \id_{k})
  \ee
to bases of eigenvectors in each of the spaces $\Hm ipl$, $\Hm jkp$, $\Hm ijq$ and $\Hm qkl$. The latter equality, in turn, follows directly from the definition \eqref{eq:ab=F.cd-} of the \sjs s by the fact that $(-)^{\cupp\cupp}$ is a linear functor. (For a different proof of \eqref{eq:epseps=epseps} see \Cite{Cor.\,3.26}{bart6}.)
 \\[3pt]
Finally, the sum rule \eqref{eq:dd:sumrule} is obtained by combining the formula \eqref{eq:MMdiag} with the identity
  \be
  \sum_{k\in S} \sumN\alpha ijk \MM_{\alpha,\alpha}
  \equ{eq:MM2FGversions} \sum_{k,\alpha,\beta} \G \ib ijj\OO k \oo\oo\alpha\beta\, \F \ib ijjk\OO \alpha\beta\oo\oo 
  = \big(\GM \ib ijj\, \FM \ib ijj \big)_{\!\oo\OO\oo,\oo\OO\oo} = 1 \,,
  \label{eq:traceMM2}
  \ee
which is the ``covector trace'' of the completeness relation \eqref{eq:dominance}.
\end{proof}
	
The formula \eqref{eq:eps-oodd=1} and the sum rule \eqref{eq:dd:sumrule} correspond to Lemma 3.16 and Proposition 3.17 of \cite{bart6}, respectively.

\begin{rem}
(i) The square root factors in \eqref{eq:doubleduals-diag} and \eqref{eq:doubleduals-diag-bar}
are equal to 1 iff 
  \be
  \dd i \, \dd j = \dd k \quad \text{for}~~ \N ijk \,{\ne}\, 0 \,.
  \ee
According to \eqref{eq:quadrupledual-id} this is the case iff the quadruple dual is the identity functor.
 \\[3pt]
(ii) If we choose the parameters $\mu_i$ as in \eqref{eq:mu-choice} so that left and right dimensions coincide, then in an eigenbasis of $\Hm ijk$ we just have $\alpha^{\cupp\cupp} \eq \eps_\alpha\, \alpha$. More generally, according to \eqref{eq:reldim} we have
  \be
  \alpha^{\cupp\cupp} = \eps_\alpha\, \frac{\mu_\ib\,\mu_\jb\,\mu_k}{\mu_i\,\mu_j\,\mu_\kb}
  \, \sqrt{ \frac {\Fo i\,\Fo j\,\Fob k} {\Fob i\,\Fob j\,\Fo k} } \, \alpha \;.
  \ee
Thus in particular for some triples $(i,j,k) \iN \S^{\times 3}$ involving non-self-conjugate objects we can compensate the signs $\eps_\alpha$ (uniformly for all members of an eigenbasis) by judiciously modifying the choice of parameters $\mu_i$.
 \\[3pt]
(iii) As an immediate consequence of the fact that $(\alpha^\cupp)^{\cupp\cupp} \eq (\alpha^{\cupp\cupp})^{\cupp}$, in the eigenbasis we also have
  \be
  (\alpha^\cupp)^{\cupp\cupp}
  = \eps_\alpha\, \sqrt{ \dd\ib\, \dd\jb\, \dd k }\, \alpha^\cupp
  \label{eq:doubleduals-diag-dual}
  \ee
for each basis vector $\alpha$, as well as analogous formulas involving left rigidities and/or dual basis vectors $\cc\alpha$.
 \\[3pt]
(iv) In the multiplicity-free case, i.e.\ when $\dimk(\Hm ijk) \iN \{0,1\}$ for all $i,j,k \iN \S$, every basis is automatically an eigenbasis. Accordingly, in applications in condensed matter physics that restrict attention to the multiplicity-free case, the diagonalization of the matrix $\MM$ is not an issue.
\end{rem}

Since $\Tr_L (\alpha\,{\circ}\,\Bar\beta) \,{=}\, \delta_{\alpha,\beta}\,\di k$ and $\Tr_R (\alpha\,{\circ}\,\Bar\beta) \,{=}\, \delta_{\alpha,\beta}\,\di\kb$, the formulas \eqref{eq:alphabarcuppcupp} and \eqref{eq:cuppcuppalphabar} imply that \D\ can be endowed with a strict pivotal structure if and only if the traces $\Tr_L$ and $\Tr_R$ are cyclic. 
More generally, any pivotal structure $\pi$ on \D\ is completely determined by its components $\pi_i$ on the simple objects $i \iN \S$: for any object $x$ we have
  \be
  \pi_x = \sum_{i \in \S} \sum_{\gamma} \Bar{\gamma} \circ \pi_i \circ \gamma^{\cupp\cupp} ,
  \label{eq:pi-candidate}
  \ee
where the $\gamma$-summation is over a basis of $\Hom(x,i)$. 
Moreover, since the double dual is the identity on objects, $\pi_i$ is just an invertible multiple of the identity morphism. It is convenient to express this as
  \be
  \pi_i = \sqrt{\dd i}\, \barpi_i \, \id_i
  \label{eq:def.barpi}
  \ee
with some invertible scalars $\barpi_i$, for $i \iN \S$. 
We then get (compare Theorem 5.4 and Lemma 5.7 of \cite{bart6})

\begin{lem} \label{lem:barpi}
(i) A \ndgusion category \D\ admits a pivotal structure only if for every triple $(i,j,k)\iN \S^{\times 3}$ the matrix $\MM^{(i\,j\,k)}$ is a multiple of the identity matrix, so that
  \be
  \eps_{i,j;\alpha}^{~k} = \eps_{ij}^{\,k}
  \label{eq:eps-ijk}
  \ee
does not depend on the specific element $\alpha$ of an eigenbasis of $\Hm ijk$, and hence every basis is an eigenbasis.
 \\[3pt]
(ii) A pivotal structure on a \ndgusion category \D\ is characterized by a collection 
$\{\barpi_i\mid i\iN\S\}$ of roots of unity satisfying the coboundary equation
  \be
  \barpi_i\, \barpi_j = \eps_{i,j}^{~k}\, \barpi_k
  \label{eq:barpi.ijk}
  \ee
for $\N ijk \,{\ne}\, 0$, as well as 
  \be
  \barpi_i\, \barpi_\ib = 1
  \label{eq:barpibarpib=1}
  \ee	
for all $i\iN\S$ and $\barpi_\OO \eq 1$.
\end{lem}

\begin{proof}
Since all morphisms can be expressed in terms of the ones in the basic morphism spaces, the requirement that $\pi$ is a monoidal natural transformation boils down to the condition that
  \be
  \pi_k \circ \alpha^{\cupp\cupp} = \alpha \circ (\pi_i \oti \pi_j)
  \label{eq:pi.cupcup=pipi}
  \ee
for all $\alpha \iN \Hm ijk$, for all $i,j,k\iN \S$. Once we restrict to bases of the spaces $\Hm ijk$ consisting of eigenvectors of the double dual, we can use \eqref{eq:doubleduals-diag} and \eqref{eq:def.barpi} to rewrite the condition \eqref{eq:pi.cupcup=pipi} as $\barpi_i\, \barpi_j \eq \eps_{i,j;\alpha}^{~k}\, \barpi_k$ whenever $\N ijk \,{\ne}\, 0$. This can only hold if $\eps_{i,j;\alpha}^{~k}$ actually does not depend on $\alpha$, thus proving (i).
Further, in this case the scalars $\barpi_i$ defined by \eqref{eq:def.barpi} satisfy \eqref{eq:barpi.ijk}. That $\barpi_\OO \eq 1$ follows from the fact that $\eps_{i\OO}^{\,i} \eq 1 \eq \eps_{\OO i}^{\,i}$ any $i \iN \S$, and then $\barpi_i\, \barpi_\ib \eq 1$ follows from $\eps_{i\ib}^{\,\OO} \eq 1$ which, in turn, is a consequence of \eqref{eq:eps-oodd=1}.
Finally observe that the squares of the so-obtained numbers $\barpi_i$ furnish a grading of the Grothendieck ring $\Gr(\D)$ and thus, by Corollary 3.7 of \cite{geNi}, a one-dimensional representation of the \textit{universal grading group} of $\Gr(\D)$; since the latter group is finite,
this implies that the numbers $\barpi_i$ are roots of unity.
\end{proof}

\begin{rem} \label{rem:Tfunctor}
(i) If the signs \eqref{eq:eps-ijk} are all equal to 1, then the coboundary equation \eqref{eq:barpi.ijk} is satisfied trivially with $\barpi_i \eq 1$ for all $i \iN \S$. Thus in this case \D\ does admit a pivotal structure.
 \\[3pt]
(ii) The collection of involutive matrices of which the signs $\eps_{i,j;\alpha}^{~k}$ are the eigenvalues can be extended uniquely to a natural transformation $T$ from the tensor product functor to itself, which endows the identity functor with the structure of a monoidal functor \Cite{Prop.\,3.25}{bart6}. Moreover, there is (canonically) a monoidal natural isomorphism between the so-obtained monoidal functor $\mathcal T$ and the double dual functor \Cite{Thm.\,3.29}{bart6}. Since the functor $\mathcal T$ is the identity as a monoidal functor iff all the $\eps_{i,j;\alpha}^{~k}$ are equal to 1, this reproduces in particular the observation of part (i).
\end{rem}
	    
Note that the assertions of Lemma \ref{lem:barpi}, while formulated for \ndgusion categories, are preserved under monoidal natural equivalence and therefore apply in fact to all fusion categories.

\medskip

Next recall that the \emph{Frobenius-Perron dimensions} $\dF i$ are the unique positive real numbers $\dF i$ obeying $\sum_{k\in\S}\N ijk\,\dF k \eq \dF i\,\dF j$. A pivotal fusion category over $\ko \eq \complex$ is \emph{pseudo-unitary} if and only if the paired pivotal dimension of every simple object coincides with the square of its Frobenius-Perron dimension \Cite{Sect.\,9.4}{EGno}, i.e.\
  \be
  \dim_L^\piv(i) \, \dim_R^\piv(i) = \big( \dF i \big)^2_{}
  \label{eq:def:pseudo}
  \ee
for every $i \iN \S$. We have

\begin{lem} \label{lem:pseudo-etc}
(i) The pivotal dimensions of a pivotal \ndgusion category obey
  \be
  \dim_R^\piv(i) = \dim_L^\piv(\ib) 
  \label{eq:dimRpiv-dimLpiv}
  \ee
and 
  \be
  \dim_L^\piv(i) \, \dim_R^\piv(i) = \di i\,\di \ib \,.
  \label{eq:dimRpiv.dimLpiv}
  \ee
(In particular, the paired pivotal dimension actually does not dependent on the pivotal structure.)
 \\[3pt]
(ii) In a pseudo-unitary \ndgusion category \D\ the matrix \eqref{eq:MMdiag} is the identity matrix, for every triple $(i,j,k)\iN \S^{\times 3}$, so that in particular
  \be
  \eps_{i,j;\alpha}^{~k} = 1
  \label{eq:eps=1}
  \ee
for all $i,j,k \iN \S$ and all basis elements $\alpha \iN \Hm ijk$.
\end{lem} 

\begin{proof}
(i) Combining \eqref{eq:barpi.ijk} with the expressions \eqref{eq:def:dimLRpiv} for 
the pivotal dimensions gives
  \be
  \dim_L^\piv(i) = \sqrt{\di i\,\di \ib}\; \barpi_i^{-1} \qquand
  \dim_R^\piv(i) = \sqrt{\di i\,\di \ib}\; \barpi_i^{} \,.
  \label{eq:def:dimLRpiv.D}
  \ee
Together with \eqref{eq:barpibarpib=1} this implies both \eqref{eq:dimRpiv-dimLpiv} 
and \eqref{eq:dimRpiv.dimLpiv}.
 \\[3pt]
(ii) For $\ko \eq \complex$ the paired dimensions are positive 
\Cite{Thm.\,7.21.12}{EGno}, hence the numbers $\sqrt{\di i\, \di\ib}$ are real, and
we may choose square roots such that they are positive. Further, 
the sum rule \eqref{eq:dd:sumrule} implies that the $\sqrt{\di i\, \di\ib}$ furnish a character of the Grothendieck ring $\Gr(\D)$ if and only if $\eps_{i,j;\alpha}^{~k} \eq 1$ for all $i,j,k$ and $\alpha$. Being
positive, they must then coincide with the Frobenius-Perron dimensions, whereby the 
equality \eqref{eq:dimRpiv.dimLpiv}
amounts to the characterization of \eqref{eq:def:pseudo} of pseudo-unitarity.
\end{proof}

Part (ii) of Lemma \ref{lem:pseudo-etc} corresponds to Corollary 3.22 of \cite{bart6}.

\begin{rem}
(i) The formulas \eqref{eq:def:dimLRpiv.D} also show that the pivotal structure $\piv$ is spherical if and only if $\barpi_i \iN \{ \pm1 \}$ for every $i \iN \S$, compare \Cite{Thm.\,5.4}{bart6}. Note that owing to \eqref{eq:barpibarpib=1} this condition is automatically satisfied if all objects $i \iN \S$ are self-conjugate.
 \\[3pt]
(ii) A pseudo-unitary fusion category \C\ admits a canonical spherical structure \Cite{Prop.\,8.13}{etno}. By \eqref{eq:eps=1} and \eqref{eq:barpi.ijk}, for that structure the signs $\barpi \iN \{\pm 1\}$ furnish a $\mathbb Z_2$-grading of the Grothen\-dieck ring of \C.
 \\[3pt]
(iii) A fusion category over \complex\ is called \emph{Hermitian} if every morphism space is endowed with a non-degenerate Hermitian form in such a way that some natural compatibility conditions are fulfilled. If all these forms are positive definite, then the category is called \emph{unitary}. It is known (see \Cite{Sect.\,4}{yama8} and \Cite{Rem.\,2.2(ii)}{gali2}) that a fusion category over \complex\ is unitary if and only if there is a choice of bases in the morphism spaces $\Hm ijk$ such that the matrices $\FFm ijklpq$ formed by the \sjs s are unitary. A unitary fusion category is in particular pseudo-unitary; while the unitarity of $\FF$-matrices can be technically convenient, in our context pseudo-unitarity is the more interesting property. (However, no example is known of a pseudo-unitary fusion category over \complex\ that does not admit a unitary structure.)
\end{rem}

\begin{rem}
For any pivotal fusion category \C, with pivotal structure $\piv$ and with chosen set $\S$ of representatives $X_i$ for the isomorphism classes of simple objects, one defines the \emph{Frobenius-Schur endomorphism} $\FS_i$ of $X_i$ by
  \be
  \FS_i := \big( \ev_{X_i^\vee} \oti \vapi_i^{-1} \big)
  \circ \big( \id_{X_i^{\vee\vee}} \oti \vapi_\ib \oti \id_{X_\ib^\vee} \big)
  \circ \big( \piv_{X_i^{}}^{-1} \oti \coev_{X_\ib} \big)
  ~ \in \Hom(X_i^{},X_i^{}) \,,
  \label{eq:def:FSi}
  \ee
where $\ib$ is the label such that $X_\ib$ is isomorphic to $X_i^\vee$ and, for each $i$, 
  \be
  \vapi_i \in \Hom(X_i^{},X_\ib^\vee)
  \label{eq:vapi}
  \ee
is a fixed selection of isomorphism. By definition, $\FS_i$ is an invertible multiple of the identity morphism. Accordingly we can write $\FS_i \,=:\, \fs_i\, \id_{X_i}^{}$ with invertible numbers $\fs_i$. 
It is straightforward to show that 
  \be
  \fs_i\, \fs_\ib\, \Tr_R(\piv_i^{}) = \tr_L(\piv_\ib^{-1})
  \ee
or, when expressed in terms of the pivotal dimensions \eqref{eq:def:dimLRpiv}, $\fs_i\, \fs_\ib\, \dim_R^\piv(X_i) \eq \dim_L^\piv(X_\ib)$. Using further that $\dim_L^\piv(X_\ib) \eq \dim_L^\piv(X_i^\vee)$ and that (see e.g.\ \Cite{Exc.\,4.7.9}{EGno}) $\piv_{X_i^\vee} \eq (\piv_{X_i})^{\vee^{\scriptstyle -1}}$, this simplifies to
  \be
  \fs_i\, \fs_\ib = 1
  \label{eq:fsi-fsib}
  \ee
for every $i \iN \S$. 
In case the simple object $X_i$ is self-conjugate, i.e.\ $\ib \eq i$, the number $\fs_i$ is independent of the choice of isomorphism $\vapi_i$ and takes values in $\{\pm1\}$; this number is called the \emph{Frobenius-Schur indicator} of $X_i$. If $X_i$ is not self-conjugate., then $\vapi_i$ and $\vapi_\ib$ can be chosen in such a way that $\fs_i \eq 1 \eq \fs_\ib$.
 \\[3pt]
In the special case of a pivotal \NDgusion category \D, we have $X_i^\vee \eq X_\ib$ and can identify $X_i \,{\equiv}\, i$, so that in particular $\piv_i$ as well as $\vapi_i$ are automorphisms and thus just multiples of the identity. Writing $\piv_i \eq \barpi_i\, \id_i$ as above, as well as $\vapi_i \eq \bar\vapi_i\, \id_i$, the number $\fs_i$ is expressed as
  \be
  \fs_i = \frac{\bar\vapi_\ib}{\bar\vapi_i}\, \frac1{\barpi_i} \,.
  \ee
This implies e.g.\ $\fs_i\, \fs_\ib \eq (\barpi_i\, \barpi_\ib)^{-1}$ in agreement with \eqref{eq:fsi-fsib}, and for self-conjugate $i$ the identification $\fs_i \eq \barpi_i^{-1} \eq \barpi_i$ of the Frobenius-Schur indicator with the scalar given by the pivotal structure.
In particular, in view of \eqref{eq:dimL.dimR}, and making the root choice $\sqrt{(\FO i)_{\phantom1}^2} \eq \FO i$, we can express the pivotal dimension of a self-conjugate simple object as
  \be
  \dim_L^\piv(i) = \dim_R^\piv(i) = \frac{\fs_i}{\FO i} \equ{eq:dimL,dimR} \fs_i\, \di i
  \ee	
in terms of its Frobenius-Schur indicator and its non-pivotal dimension. This illustrates in particular the dependence of the Frobenius-Schur indicator on the pivotal structure.
\end{rem}


\subsection{Deconstructing rigidity} \label{sec:decon}

It turns out to be instructive to disassemble the action of the rigidity functors on basis morphisms into simpler operations on the generating morphism spaces $\Hm ijk$. To this end we introduce two operations $L$ and $R$ which in the diagrammatic description amount to gluing unit-bounded faces:

\begin{defi}\label{def:LR-0}
(i) The operations $L$ and $R$ on the morphism spaces $\Hd ijk \eq \Hom(k,i \oti j)$ and $\Hm ijk \eq \Hom(i \oti j,k)$, for $(i,j,k) \,{\in}\, \S^{\times 3}$, are given by taking the tensor product with the one-dimensional spaces $\Hm \ib i \OO$ and $\Hm j\jb \OO$, respectively with $\Hd \ib i \OO$ and $\Hd j\jb \OO$, according to the diagrammatic prescription
  \begin{align}
  \raisebox{-2.9em}{ \simptex{
  \internaltwosimplex[][i,j,k]
  \ddraw (-1,1.4) to["\ensuremath{\bar{i}}" swap] (nx.center);
  \draw (-1,1.4) to["1"] (ny.center);
  } }
~\stackrel{\,\displaystyle L}\longmapsfrom~
  \raisebox{-2.9em}{ \twosimplex[][i,j,k] }
~\stackrel{\displaystyle R\,}\longmapsto~
  \raisebox{-2.9em}{ \simptex{
  \internaltwosimplex[][i,j,k]
  \ddraw (nz.center) to["\ensuremath{\bar{j}}" swap] (3,1.4);
  \draw (3,1.4) to["1" swap] (ny.center);
  } }
  \label{eq:def:LR-on-k-ij}
  \end{align}
and
  \begin{align}
  \raisebox{-2.9em}{ \simptex{
\ddraw (0,0) to["$k$"] (2,0);
\ddraw (0,0) to["$i$" swap] (1,-1.4);
\ddraw (1,-1.4) to["$j$" swap] (2,0);
\ddraw (-1,-1.4) to["\ensuremath{\bar{i}}"] (0,0);
\draw (-1,-1.4) to["1" swap] (1,-1.4);
  } }
~~~\stackrel{\,\displaystyle L}\longmapsfrom~~~
  \raisebox{-2.9em}{ \simptex[baseline=-20mm]{
\ddraw (0,0) to["$k$"] (2,0);
\ddraw (0,0) to["$i$" swap] (1,-1.4);
\ddraw (1,-1.4) to["$j$" swap] (2,0);
   } }
~~~\stackrel{\displaystyle R\,}\longmapsto~~~
  \raisebox{-2.9em}{ \simptex{
\ddraw (0,0) to["$k$"] (2,0);
\ddraw (0,0) to["$i$" swap] (1,-1.4);
\ddraw (1,-1.4) to["$j$" swap] (2,0);
\ddraw (2,0) to["\ensuremath{\bar{j}}"] (3,-1.4);
\draw (3,-1.4) to["1"] (1,-1.4);
  } }
  \label{eq:def:LR-on-ij-k}
  \end{align}
respectively.
We call the operation $L$ the \textit{partial left dual} and $R$ the \textit{partial right dual}. 
\\[2pt]
(ii) The \textit{framed partial left dual} and \textit{framed partial right dual}, also to be denoted by $L$ and $R$, respectively, are the linear maps that are obtained analogously when instead acting with chosen basis vectors in the one-dimensional spaces that are being glued in the prescriptions \eqref{eq:def:LR-on-k-ij} and \eqref{eq:def:LR-on-ij-k}.
\end{defi}
 
\begin{conv}
As the monoidal unit $1$ is taken to be strict, in the diagrammatic description we draw edges labeled by the object $1$ without orientation. Moreover, from now on the label of such unoriented edges will be omitted.
\end{conv}

Let us exhibit the matrix coefficients of the linear maps $L$ and $R$. Denote the component of $L$ that acts on the space $\Hm ijk$ by $L_{ij}^k$ and the one acting on $\Hd ijk$ by $L^{ij}_k$, and similarly for $R$. Then we have (omitting the labeling of the trivial faces by their unique basis morphisms)
  \begin{align}
  L_{ij}^k(\alpha) \,= \simptex[baseline=5mm,scale=0.8]{
    \ddraw (0,0) to["$\Bar{i}$"] (1,1.4);
    \ddraw (1,1.4) to["$i$" swap] (2,0);
    \draw (0,0) to (2,0);
    \ddraw (2,0) to["$j$" swap] (3,1.4);
    \ddraw (1,1.4) to["$k$"] (3,1.4);
    \draw (2,0.8) node{$\alpha$};
    } =\, \sum_\delta \F{\Bar i}ijjk1\delta\alpha\oo\oo \Bar\delta \,,
    &&
  L^{ij}_k(\Bar\alpha) \,= \simptex[baseline=5mm,scale=0.8]{\internaltwosimplex[][i,j,k][\alpha]
    \draw (1,1.4) to (-1,1.4);
    \ddraw (-1,1.4) to["$\Bar i$" swap] (0,0);
    } \!=\, \sum_\delta \G{\Bar i}ijj1k\oo\oo\delta\alpha \delta
  \label{eq:defPandQ-simplicesL}
  \end{align}
and
  \begin{align}
  R_{ij}^k(\alpha) \,= \simptex[baseline=5mm,scale=0.8]{
    \ddraw (3,1.4)  to["$\Bar{j}$"] (4,0);
    \ddraw (1,1.4) to["$i$" swap] (2,0);
    \draw (4,0) to (2,0);
    \ddraw (2,0) to["$j$" swap] (3,1.4);
    \ddraw (1,1.4) to["$k$"] (3,1.4);
    \draw (2,0.8) node{$\alpha$};
    } =\, \sum_\delta \G ij{\Bar j}ik1\delta\alpha\oo\oo \Bar\delta \,,
    &&
  R^{ij}_k(\Bar\alpha) \,= \!\!\simptex[baseline=5mm,scale=0.8]{\internaltwosimplex[][i,j,k][\alpha]
    \draw (1,1.4) to (3,1.4);
    \ddraw (2,0) to["$\Bar j$" swap] (3,1.4);
    } =\, \sum_\delta \F ij{\Bar j}i1k\oo\oo\delta\alpha \delta \,.
  \label{eq:defPandQ-simplicesR}
  \end{align}

The partial duals are obviously involutive as operations on the fundamental morphism spaces. In contrast, the \emph{framed} partial duals are, in general, \emph{not} involutive. However, $L$ and $R$ do become involutive as linear maps if we modify them by suitable scalar factors, namely if we replace the basis morphisms in the prescription in Definition \ref{def:LR-0}(ii) by rigidity morphisms,\,%
 \footnote{~Owing to $(\oocap i)^\cap \,{=}\, \oocup i$ these form a rigidity-dual pair rather than a covector-dual pair, so one has to be careful when using them in the diagrammatic semantics.} 
according to the following prescription:

\begin{defi} \label{def:LL,RR}
The \emph{modified left and right framed partial duals}, denoted by $\LL$ and $\RR$, are the maps obtained from $L$ and $R$ by using the morphism $\oocup i$ in place of $\oob i$ and $\oocapp i$ in place of $\oodd i$. Thus diagrammatically we have
  \begin{align}
  \LL_{ij}^k(\alpha) ~=~~ \simptex[baseline=5mm,scale=0.8]{
    \ddraw (0,0) to["$\Bar{i}$"] (1,1.4);
    \ddraw (1,1.4) to["$i$" swap] (2,0);
    \draw (0,0) to (2,0);
    \ddraw (2,0) to["$j$" swap] (3,1.4);
    \ddraw (1,1.4) to["$k$"] (3,1.4);
    \draw (2,0.85) node{$\alpha$};
  \draw (0.95,0.6) node{$\oocup{}$};
    }
  \qquad\text{and}\qquad
  \RR_{ij}^k(\alpha) ~=~~ \simptex[baseline=5mm,scale=0.8]{
    \ddraw (3,1.4)  to["$\Bar{j}$"] (4,0);
    \ddraw (1,1.4) to["$i$" swap] (2,0);
    \draw (4,0) to (2,0);
    \ddraw (2,0) to["$j$" swap] (3,1.4);
    \ddraw (1,1.4) to["$k$"] (3,1.4);
    \draw (2,0.85) node{$\alpha$};
  \draw (2.95,0.6) node{$\oocupp{}$};
    }
  \end{align}
and analogously for $\LL^{ij}_k(\Bar\alpha)$ and $\RR^{ij}_k(\Bar\alpha)$.
\end{defi}
  
The matrix coefficients of $\LL$ and $\RR$ are given by
  \be
  \begin{array}{llll}
  & \LL_{ij}^k = \mu_i\, L_{ij}^k & ~~\text{and}~~ & \dsty
  \LL^{ij}_k = \frac1{\mu_\ib\, \Fo i}\, \, L^{ij}_k
  \Nxl3
  \text{and by} \quad
  & \RR_{ij}^k = \mu_\jb\, R_{ij}^k & ~~\text{and}~~ & \dsty
  \RR^{ij}_k = \frac1{\mu_j\, \Fob j}\, \, R^{ij}_k \,,
  \label{LvsLL,RvsRR}
  \eear
  \ee
respectively. 

\begin{lem} \label{lem:Lq**2,R**2}
The modified framed partial duals square to the identity: we have
  \be
  \LL^2 = \id = \RR^2 
  \label{eq:LL2=1=RR2}
  \ee
as linear maps on each of the fundamental morphism spaces.
\end{lem}

\begin{proof}
It is instructive to exhibit how the result arises when starting out with the original linear maps $L$ and $R$. By definition, applying $R$ to $\Hd ijk$ twice, or $L$ to $\Hm jik$ twice, amounts to gluing with the diagram 
  \begin{align}
  \raisebox{-3.2em}{
  \begin{tikzpicture}[line width=\simptexthickness,rounded corners,decoration={markings,mark=at position 0.46 with {\arrowreversed{Stealth[length=6pt,width=6pt]}}}]
  \ddraw (1,1.4) to ["$j$" swap] (0,0);
  \draw (1,1.4) to["$1$"] (2,0);
  \ddraw (0,0) to ["$\bar{j}$"] (2,0);
  \ddraw (2,0) to["\ensuremath{j}"] (1,-1.4);
  \draw (0,0) to["1" swap] (1,-1.4);
  \end{tikzpicture}
  }
  \quad =~ \Go j ~~
  \raisebox{-3.2em}{
  \begin{tikzpicture}[line width=\simptexthickness,rounded corners,decoration={markings,mark=at position 0.46 with {\arrowreversed{Stealth[length=6pt,width=6pt]}}}]
  \ddraw (1,1.4) to ["$j$" swap] (0,0);
  \draw (1,1.4) to (2,0);
  \ddraw (1,1.4) to ["$j$"] (1,-1.4);
  \draw (1,-1.4) to (0,0);
  \ddraw (2,0) to["\ensuremath{j}"] (1,-1.4);
  \end{tikzpicture}
  }
  \end{align}
Using further that the monoidal unit is strict, this yields
  \begin{align}
  R^2\big|_{\Hd ijk}^{} = \Go j\, \id_{\Hd ijk}^{} \qquad \text{and} \qquad 
  L^2\big|_{\Hm jik}^{} = \Go j\, \id_{\Hm jik}^{} \,.
  \label{eq:R**2,L**2}
  \end{align}
Moreover, replacing $L$ and $R$ according to \eqref{LvsLL,RvsRR} by $\LL$ and $\RR$, respectively, precisely cancels the prefactor $\Go j$, and thus shows \eqref{eq:LL2=1=RR2} for $\RR^2$ acting on $\Hd ijk$ and for $\LL^2$ acting on $\Hm jik$.
 \\[2pt]
Analogously we have
  \begin{align}
  \raisebox{-3.2em}{
  \begin{tikzpicture}[line width=\simptexthickness,rounded corners,decoration={markings,mark=at position 0.46 with {\arrowreversed{Stealth[length=6pt,width=6pt]}}}]
  \draw (1,1.4) to ["$ $" swap] (0,0);
  \ddraw (2,0) to["$i$" swap] (1,1.4);
  \draw (2,0) to[""] (1,-1.4);
  \ddraw (1,-1.4) to["\ensuremath{i}"] (0,0);
  \ddraw (0,0) to ["$\bar{i}$"] (2,0);
  \end{tikzpicture}
  }
  \quad =~ \Fo i ~~
  \raisebox{-3.2em}{
  \begin{tikzpicture}[line width=\simptexthickness,rounded corners,decoration={markings,mark=at position 0.46 with {\arrowreversed{Stealth[length=6pt,width=6pt]}}}]
  \draw (1,1.4) to ["$ $" swap] (0,0);
  \ddraw (2,0) to["$i$" swap] (1,1.4);
  \draw (2,0) to[""] (1,-1.4);
  \ddraw (1,-1.4) to["\ensuremath{i}"] (0,0);
  \ddraw (1,-1.4) to ["$i$"] (1,1.4);
  \end{tikzpicture}
  }
  \end{align}
and thus
  \begin{align}
  L^2\big|_{\Hd ijk}^{} = \Fo i\, \id_{\Hd ijk}^{} \qquad \text{and} \qquad 
  R^2\big|_{\Hm jik}^{} = \Fo i\, \id_{\Hm jik}^{} \,.
  \label{eq:L**2,R**2}
  \end{align}
Again upon replacing $L$ and $R$ by $\LL$ and $\RR$, the prefactor $\Fo i$ is
canceled, showing that \eqref{eq:LL2=1=RR2} also holds for $\LL^2$ acting on $\Hd ijk$ and for $\RR^2$ acting on $\Hm jik$.
\end{proof}

\begin{rem}
(i) 
The proof of Lemma \ref{lem:Lq**2,R**2} shows that the statement boils down to the snake
identities for the rigidities $(-)^\cupp$ and $\Cupp(-)$. Diagrammatically it reads
  \begin{align}
  \raisebox{-1.6em}{
  \begin{tikzpicture}[line width=\simptexthickness,rounded corners,decoration={markings,mark=at position 0.46 with {\arrowreversed{Stealth[length=6pt,width=6pt]}}}]
  \node (nx) at (0,0) {};
  \node (ny) at (1,1.4) {};
  \node (nz) at (2,0) {};
  \draw[postaction=decorate] (nx.center)  to[edge label={\ensuremath{i}}] (ny.center);
  \draw[postaction=decorate] (ny.center)  to[edge label={\ensuremath{\bar{i}}}] (nz.center);
  \draw (nx.center)  to (nz.center);
    \draw (1,1.4) to (3,1.4);
    \ddraw (2,0) to["$i$" swap] (3,1.4);
    \draw (2,0.8) node{$\oocap i$};
    \draw (1,0.6) node{$\oocapp i$};
    \end{tikzpicture}
  }
  =~ \id_i \qquad \text{and} \qquad
  \raisebox{-1.6em}{
  \begin{tikzpicture}[line width=\simptexthickness,rounded corners,decoration={markings,mark=at position 0.46 with {\arrowreversed{Stealth[length=6pt,width=6pt]}}}]
  \node (nx) at (0,0) {};
  \node (ny) at (1,1.4) {};
  \node (nz) at (2,0) {};
  \draw[postaction=decorate] (nx.center)  to[edge label={\ensuremath{i}}] (ny.center);
  \draw[postaction=decorate] (ny.center)  to[edge label={\ensuremath{\bar{i}}}] (nz.center);
  \draw (nx.center)  to (nz.center);
    \draw (-1,1.4) to (1,1.4);
    \ddraw (-1,1.4) to["$\bar{i}$" swap] (0,0);
    \draw (0,0.8) node{$\oocap i$};
    \draw (1,0.6) node{$\oocapp i$};
    \end{tikzpicture}
  }
  =~ \id_\ib \,.
  \end{align}
(ii) The operations $L$ and $R$ make sense for any \gusion category \D, whereas $\LL$ and $\RR$ are only defined for non-degenerate \D. However, as illustrated by the presence of the factors $\Go j$ in \eqref{eq:R**2,L**2} and $\Fo i$ in \eqref{eq:L**2,R**2} in the degenerate case $L$ and $R$ are uninteresting.
\end{rem}

We now describe the relation between $\LL$ and $\RR$ and the rigidity functors. By direct calculation one checks:

\begin{lem} \label{lem:cupp-va-LLRR}
The left and right rigidity duals can be expressed as combinations of the modified framed partial duals $\LL$ and $\RR$ according to
  \be
  \begin{array}{lll}
  & \alpha^\cupp = \RR\LL\RR (\alpha) \,, \quad
  & \cc\alpha^\cupp = \LL\RR\LL (\cc\alpha)
  \Nxl3
  \text{and as} \quad & {}^{\cupp\!}\alpha  = \LL\RR\LL (\alpha) \,,
  & {}^\cupp\cc\alpha = \RR\LL\RR (\cc\alpha) \,,
  \eear
  \label{eq:duals-as-LRLorRLR}
  \ee
respectively, for $\alpha \iN \Hm ijk$. 
In particular, the two composites $\RR\LL\RR$ and $\LL\RR\LL$ result in the same morphism spaces, and they are equal as linear maps if the left and right dimensions coincide. 
\end{lem}

Diagrammatically the action of these combinations of $\LL$ and $\RR$ on $\Hd ijk$ is given by
  \begin{align}
  \raisebox{-2.7em}{ \twosimplex[][i,j,k] }
  \quad \xmapsto{~ \RR\LL\RR ~} \quad
  \raisebox{-3.2em}{
  \begin{tikzpicture}[line width=\simptexthickness,rounded corners,decoration={markings,mark=at position 0.46 with {\arrowreversed{Stealth[length=6pt,width=6pt]}}}]
  \internaltwosimplex[][i,j,k]
  \ddraw (nz.center) to["\ensuremath{\bar{j}}" swap] (3,1.4);
  \draw (3,1.4) to (ny.center);
  \draw (2,0) to[""] (1,-1.4);
  \ddraw (1,1.4) to["$\bar i$" swap] (-1,1.4);
  \ddraw (1,-1.4) to["$\bar k$"] (0,0);
  \draw (-1,1.4) to[""] (0,0);
  \end{tikzpicture}
  }
  ~~ \cong ~
  \raisebox{-2.7em}{
  \twosimplex[][\bar{i},\bar{j},\bar{k}][][line width=\simptexthickness,rounded corners,decoration={markings,mark=at position 0.58 with {\arrow{Stealth[length=6pt,width=6pt]}}}]
  }
  \label{simp:QPQ}
  \\
  \raisebox{-2.7em}{ \twosimplex[][i,j,k] }
  \quad \xmapsto{~ \LL\RR\LL ~} \quad
  \raisebox{-3.2em}{
  \begin{tikzpicture}[line width=\simptexthickness,rounded corners,decoration={markings,mark=at position 0.46 with {\arrowreversed{Stealth[length=6pt,width=6pt]}}}]
  \internaltwosimplex[][i,j,k]
  \draw (nz.center) to[" " swap] (3,1.4);
  \ddraw (3,1.4) to["\ensuremath{\bar{j}}" swap] (ny.center);
  \draw (0,0) to[" " swap] (1,-1.4);
  \draw (1,1.4) to[" " swap] (-1,1.4);
  \ddraw (2,0) to["$\bar k$"] (1,-1.4);
  \ddraw (-1,1.4) to["$\bar i$" swap] (0,0);
  \end{tikzpicture}
  }
  ~~ \cong ~
  \raisebox{-2.7em}{
  \twosimplex[][\bar{i},\bar{j},\bar{k}][][line width=\simptexthickness,rounded corners,decoration={markings,mark=at position 0.58 with {\arrow{Stealth[length=6pt,width=6pt]}}}]
  \label{simp:PQP}
  }
  \end{align}
and analogously for the action on $\Hm ijk$.

For the double dual, the equalities \eqref{eq:duals-as-LRLorRLR} mean
  \be
  \alpha^{\cupp\cupp} = (\LL\RR)^3 (\alpha) \qquand
  {}^{\cupp\cupp\!}\alpha = (\RR\LL)^3 (\alpha) \,.
  \label{eq:dbleduals-as-LR3orRL3}
  \ee
Thus explicitly we have
  \be
  (\LL\RR)^3(\alpha) 
  \equ{eq:alphabarcuppcupp} \frac{\di i\, \di j} {\di k} \, \MM\, \alpha
  \equ{eq:MM2FGversions} \frac{\di i\, \di j} {\di k} \sum_{\beta,\mu}
  \F ij\jb i\OO k \oo\oo\beta\mu\, \G ij\jb ik\OO \alpha\mu\oo\oo \, \beta
  \label{eq:(RL)3-form1rescaled}
  \ee
for $\alpha \iN \Hm ijk$, and similarly for $(\RR\LL)^3$.

Combining \eqref{eq:dbleduals-as-LR3orRL3} with Lemma \ref{lem:Lq**2,R**2} we arrive at

\begin{prop}\label{prop:genuineS3}
The linear maps $\LL$ and $\RR$ generate a genuine action of the symmetric group \SSS\ on the collection of morphism spaces $\Hm ijk$ and $\Hd ijk$ if and only if the double dual $(-)^{\cupp\cupp}$ is the identity as a functor.
\end{prop}

Let us also provide alternative expressions for the linear maps $(\RR\LL)^3$ and $(\LL\RR)^3$, which are obtained by iterating the maps $\RR\LL$ and $\LL\RR$, respectively.  
It turns out to be notationally convenient to choose object labels such that $\RR\LL$ and $\LL\RR$ are linear maps $\Hm ij{\kb} \,{\to}\, \Hm jk{\ib}$ and $\Hm ij{\kb} \,{\to}\, \Hm ki{\jb}$, respectively. Doing so we obtain
  \begin{align}
  RL(\alpha) ~=~  
  \simptex[baseline=9mm,scale=0.8]{
  \ddraw (0,0) to["$\ib$"] (1,1.4);
  \ddraw (1,1.4) to["$i$"] (2,0);
  \draw (0,0) to (2,0);
  \ddraw (2,0) to["$j$" swap] (3,1.4);
  \ddraw (1,1.4) to["$\kb$"] (3,1.4);
  \draw (1,1.4) to (2,2.8);
  \ddraw (3,1.4) to["$k$" swap] (2,2.8);
  \draw (2,0.7) node{$\alpha$};
  }
  ~=~ \sum_\delta \G ijk\OO\kb\ib \alpha\oo\oo\delta~ 
  \simptex[baseline=9mm,scale=0.8]{
  \ddraw (0,0) to["$\ib$"] (1,1.4);
  \ddraw (1,1.4) to["$i$"] (2,0);
  \draw (0,0) to (2,0);
  \ddraw (2,0) to["$j$" swap] (3,1.4);
  \ddraw (2,0) to["$\bar{i}$" swap] (2,2.8);
  \draw (1,1.4) to (2,2.8);
  \ddraw (3,1.4) to["$k$" swap] (2,2.8);
  \draw (2.5,1.4) node{$\delta$};
  } 
  ~=~ \Fob i \sum_\delta \G ijk\OO\kb\ib \alpha\oo\oo\delta 
  \, \delta
  \label{eq:RLnonrescaled}
  \end{align}
for $\alpha \iN \Hm ij\kb$, and similarly $LR(\alpha) \,{=}\, \Fo j 
\sum_\delta \F kij\OO\kb\jb\oo\alpha\delta\oo \,\delta$.
With the chosen convention of object labels, iteration just amounts to permuting those labels cyclically. We thus arrive at
  \begin{align}
  (\RR\LL)^3(\alpha)
  & = \di\ib\, \di\jb\, \di\kb\, \Fob i\,\Fob j\,\Fob k \sum_{\delta,\delta',\delta''} 
  \G ijk\OO\kb\ib\alpha\oo\oo\delta\, \G jki\OO\ib\jb\delta\oo\oo{\delta'}\,
  \G kij\OO\jb\kb{\delta'}\oo\oo{\delta''}\, \delta''
  \label{eq:RL**3rescaled}
  \end{align}
and
  \begin{align}
  (\LL\RR)^3(\alpha) & = \di i\, \di j\, \di k\, \Fo i\,\Fo j\,\Fo k \sum_{\delta,\delta',\delta''}
  \F kij\OO\kb\jb\oo\alpha\delta\oo\, \F jki\OO\jb\ib\oo\delta{\delta'}\oo\,
  \F ijk\OO\ib\kb\oo{\delta'}{\delta''}\oo\, \delta''
  \label{eq:LR**3rescaled}
  \end{align}
for $\alpha \iN \Hm ij\kb$.
Note that, using the fact that the numbers $\F ijk\OO\ib\kb\oo\alpha\beta\OO$ form an invertible matrix in the labels $\alpha$ and $\beta$ alone, with inverse matrix given by $\G ijk\OO\ib\kb\oo\beta\gamma\OO$, it directly follows from these identities that $(LR)^3\cir(RL)^3 \eq \id$, in agreement with Lemma \ref{lem:Lq**2,R**2}.

 \medskip

It is also worth recalling the diagonalized form of the double dual functors, as given in formulas \eqref{eq:doubleduals-diag} and \eqref{eq:doubleduals-diag-bar}. It is natural to study the compatibility of the operations $\LL$ and $\RR$ with the diagonalization of the double dual. We find

\begin{lem}
The linear maps $\LL$ and $\RR$ are interchanged by the rigidity dualities: we have
  \be
  \begin{array}{lll}
  & (\LL(\alpha))^{\cupp} = \RR(\alpha^\cupp) \,, \quad~
  & (\LL(\cc\alpha))^{\cupp} = \RR(\cc\alpha^{\cupp}) 
  \Nxl3
  \text{and} \qquad & (\RR(\alpha))^{\cupp} = \LL(\alpha^\cupp) \,, 
  & (\RR(\cc\alpha))^{\cupp} = \LL(\cc\alpha^{\cupp})
  \end{array}
  \label{eq:LL.RR.cap}
  \ee
and similarly for left duals.
As a consequence we have
  \be
  \eps_{\LL(\alpha)} = \eps_\alpha = \eps_{\RR(\alpha)} \,,
  \label{eq:eps:LR}
  \ee
and hence the entire orbit of an eigenvector $\alpha$ under iterated actions of $\LL$ and $\RR$ consists of eigenvectors with the same eigenvalue as $\alpha$ under the double dual. The same conclusion is obtained for $\cc\alpha$.
\end{lem}
     
\begin{proof}
We first note that by direct calculation one has 
  \be
  \LL(\cc\alpha) \circ \LL(\alpha) = \di\ib\, \MM_{\alpha,\alpha}\, \id_j \qquand
  \RR(\cc\alpha) \circ \RR(\alpha) = \di j\, \MM_{\alpha,\alpha}\, \id_i \,.
  \ee
It follows that
  \be
  \LL\alpha = \di\ib\, M_{\alpha,\alpha}\, \overline{\LL\cc\alpha}
  \label{eq:L-vs-barLbar}
  \ee
and analogously for $\RR$. Invoking \eqref{eq:duals-as-LRLorRLR} this, in turn, implies the equalities \eqref{eq:LL.RR.cap}.
Given these equalities, it follows in particular that $(\LL(\alpha))^{\cupp\cupp} \eq \LL(\alpha^{\cupp\cupp})$ etc. 
When expressing the double dual through the signs obtained from diagonalization, this amounts to \eqref{eq:eps:LR}.
\end{proof}

\begin{rem} \label{rem:FiFipositive}
Let us, as in the proof of Theorem 2.3 of \cite{etno}, define $T_i$ to be the square matrix with entries $(T_i)_{j,k} \,{:=}\, \sum_\alpha \eps_{i,j;\alpha}^{~k}$, for $i,j,k \iN \S$. Then the sum rule \eqref{eq:dd:sumrule} states that the length-$|\S|$ vector $D$ with entries $D_i \,{:=}\, \sqrt{\di i\,\di\ib}$ given by the square roots of the paired dimensions of simple objects is a simultaneous eigenvector to all matrices $T_j$, with respective eigenvalues $D_j$.
Moreover, by combining \eqref{eq:eps:LR} with Lemma \ref{lem:epsalpha}(ii) one sees that the matrix $T_\ib$ is the transpose of the matrix $T_i$, and thus $D$ is also an eigenvector to the matrix $T_i^{} T_i^{\mathrm t} \eq T_i^{} T_\ib$, with eigenvalue $\di i\,\di\ib$.
Let now $\ko \eq \complex$. Then $T_i^{} T_i^{\mathrm t}$ is non-negative definite, and hence its eigenvalue $\di i\,\di\ib$ must be non-negative, and thus in fact positive. By the formula \eqref{eq:dimL.dimR} for $\di i\,\di\ib$ this implies that the product $\Fo i\,\Fob i$ of \sjs s is positive and thus, once we make the basis choice described in Remark \ref{rem:rescaledFi}, $\Fo i$ is a real number for every $i \iN \S$.
\end{rem}


\section{Tetrahedral symmetry} \label{sec:SSSS}

\subsection{An \SSSS-action on morphism spaces}

We are now ready to address the main theme of this paper: tetrahedral symmetries of \sjs s. We first establish an action of the symmetric group \SSSS\ on the four-fold tensor products of basic morphism spaces which appear in the definition of \sjs s. To this end we make use of the action of \SSS\ by linear maps on the spaces $\Hm ijk$ and $\Hd ijk$ that is generated by the framed partial duals $\LL$ and $\RR$, as seen in Proposition \ref{prop:genuineS3}. Concretely, we define linear endomorphisms $\tau_{12}$, $\tau_{23}$ and $\tau_{34}$ of the direct sum vector space
  \be
  \HtotF := \!\bigoplus_{i,j,k,p,q,l \in \S}\!  \big( \HFm ijklpq \oplus \HGm ijklpq \big) \,,
  \label{eq:HtotF}
  \ee
where
  \be
  \HFm ijklpq := \Hm jkp \oti \Hd qkl \oti \Hm ipl \oti \Hd ijq
  \qquad \text{and} \qquad
  \HGm ijklpq := \Hd jkp \oti \Hm qkl \oti \Hd ipl \oti \Hm ijq \,.
  \label{eq:HF,HG}
  \ee
These maps will eventually play the role of the standard transpositions that generate the group \SSSS. The order of tensor product factors in $\HFm ijkpql$ and $\HGm ijkpql$ has been chosen so as to match that interpretation.

\begin{defi} \label{def:tauij}
Let \D\ be a \ndgusion category with distinguished set $S$ of simple objects. The maps $\tau_{ij}\colon \HtotF \To \HtotF$, for $(i,j) \iN \{ (1,2),(2,3),(3,4) \}$, are defined by linearly extending the following assignments for the basis of $\HtotF$ that consists of tensor products of basis elements of the basic morphism spaces:
  \be
  \bearll
  & \tau_{12}(\beta {\otimes} \Bar\delta {\otimes} \alpha {\otimes} \Bar\gamma)
  := \Bar\delta {\otimes} \beta {\otimes} \LL\alpha {\otimes} \LL\Bar\gamma
  ~ \iN \HGm {\ib}qkplj \,,
  \Nxl3
  & \tau_{23}(\beta {\otimes} \Bar\delta {\otimes} \alpha {\otimes} \Bar\gamma)
  := \LL\beta {\otimes} \alpha {\otimes} \Bar\delta {\otimes} \RR\Bar\gamma
  ~ \iN \HGm q{\jb}plki \,,
  \Nxl3
  \text{and} \quad & \tau_{34}(\beta {\otimes} \Bar\delta {\otimes} \alpha {\otimes} \Bar\gamma)
  := \RR\beta {\otimes} \RR\Bar{\delta} {\otimes} \Bar{\gamma} {\otimes} \alpha
  ~ \iN \HGm ip{\kb}qjl 
  \eear
  \label{eq:def:tauij-on-HF}
  \ee
for $(i,j,k,p,q,l) \iN S^{\times 6}$ and
$\alpha \iN \Hm ipl$, $\beta \iN \Hm jkp$, $\gamma \iN \Hm ijq$, $\delta \iN \Hm qkl$,
as well as 
  \be
  \hspace*{3.3em} \tau_{12}(\Bar\beta {\otimes} \delta {\otimes} \Bar\alpha {\otimes} \gamma)
  := \delta {\otimes} \Bar\beta {\otimes} \LL\Bar\alpha {\otimes} \LL\gamma
  ~ \in \HFm {\ib}qkplj
  \label{eq:def:tauij-on-HG}
  \ee
etc.\ for $\Bar\beta {\otimes} \delta {\otimes} \Bar\alpha {\otimes} \gamma
\iN \HGm ijklpq$.
\end{defi}

To see under which conditions the so defined maps generate an action of the group \SSSS\ on the vector space $\HtotF$, we note the following identities, which are obtained by direct calculation:

\begin{lem} \label{lem:S4relationsN}
The maps \eqref{eq:def:tauij-on-HF} and \eqref{eq:def:tauij-on-HG} satisfy 
  \be
  \begin{array}{rl}
  \tau_{12}^2 (\beta {\otimes} \Bar\delta {\otimes} \alpha {\otimes} \Bar\gamma) \!\!&
  = \beta {\otimes} \Bar\delta {\otimes} \LL^2\alpha {\otimes} \LL^2\Bar\gamma \,,
  \Nxl4
  \tau_{23}^2 (\beta {\otimes} \Bar\delta {\otimes} \alpha {\otimes} \Bar\gamma) \!\!&
  = \LL^2\beta {\otimes} \Bar\delta {\otimes} \alpha {\otimes} \RR^2\Bar\gamma \,,
  \Nxl4
  \tau_{34}^2 (\beta {\otimes} \Bar\delta {\otimes} \alpha {\otimes} \Bar\gamma) \!\!&
  = \RR^2\beta {\otimes} \RR^2\Bar\delta {\otimes} \alpha {\otimes} \Bar\gamma \,,
  \Nxl4
  \multicolumn2c{\hspace*{-3.7em}
  \tau_{12}^{}\,\tau_{34}^{}(\beta {\otimes} \Bar\delta {\otimes} \alpha {\otimes} \Bar\gamma)
  = \RR\Bar\delta \oti \RR\beta \oti \LL\Bar\gamma \oti \LL\alpha
  = \tau_{34}^{} \, \tau_{12}^{}(\beta {\otimes} \Bar\delta {\otimes} \alpha {\otimes} \Bar\gamma) \,, }
  \Nxl4
  \tau_{12}^{}\,\tau_{23}^{}\,\tau_{12}^{}
  (\beta {\otimes} \Bar\delta {\otimes} \alpha {\otimes} \Bar\gamma) \!\!&
  = \LL\alpha {\otimes} \LL\Bar\delta {\otimes} \LL\beta {\otimes} \LL\RR\LL\Bar\gamma \,,
  \Nxl4
  \tau_{23}^{}\,\tau_{12}^{}\,\tau_{23}^{}
  (\beta {\otimes} \Bar\delta {\otimes} \alpha {\otimes} \Bar\gamma) \!\!&
  = \LL\alpha {\otimes} \LL\Bar\delta {\otimes} \LL\beta {\otimes} \RR\LL\RR\Bar\gamma \,,
  \Nxl4
  \tau_{23}^{}\,\tau_{34}^{}\,\tau_{23}^{}
  (\beta {\otimes} \Bar\delta {\otimes} \alpha {\otimes} \Bar\gamma) \!\!&
  = \LL\RR\LL\beta {\otimes} \RR\Bar\gamma {\otimes} \RR\alpha {\otimes} \RR\Bar\delta \,,
  \Nxl4
  \tau_{23}^{}\,\tau_{34}^{}\,\tau_{23}^{}
  (\beta {\otimes} \Bar\delta {\otimes} \alpha {\otimes} \Bar\gamma) \!\!&
  = \RR\LL\RR\beta {\otimes} \RR\Bar\gamma {\otimes} \RR\alpha {\otimes} \RR\Bar\delta 
  \eear
  \label{eq:tauij-actionN}
  \ee
for all $\beta {\otimes} \Bar\delta {\otimes} \alpha {\otimes} \Bar\gamma \iN
\HF_{i,j,k,p,q,l} \,{\subset}\, \HtotF$. 
 \\
Analogous identities hold when acting on $\Bar\delta \oti \beta \oti \Bar\gamma \oti \alpha \iN \HG_{i,j,k,p,q,l} \,{\subset}\, \HtotF$.
\end{lem}

When combined with Proposition \ref{prop:genuineS3}, this gives

\begin{prop}\label{prop:genuineS4N}
Let \D\ be a \ndgusion category such that $(-)^{\cupp\cupp} \,{=}\, \Id_\D$. Then the maps $\tau_{12}$, $\tau_{23}$ and $\tau_{34}$ generate an \SSSS-action on the space $\HtotF$.
\end{prop}

\begin{proof}
According to Proposition \ref{prop:genuineS3} we have $\LL^2 \,{=}\, \id \,{=}\, \RR^2$ and $\LL\RR\LL \,{=}\, \RR\LL\RR$ if $(-)^{\cupp\cupp} \,{=}\, \Id_\D$. The equalities \eqref{eq:tauij-actionN} then imply that $\tau_{12}$, $\tau_{23}$ and $\tau_{34}$ satisfy the relations for the standard transpositions that generate the group \SSSS.
\end{proof}

\begin{rem} 
The precise form of the maps $\LL$ and $\RR$, and thus also the \SSS-action on the fundamental morphism spaces, depends on the choice of the parameters $\mu_i$. In the \SSSS-action obtained from Definition \ref{def:tauij}, this dependence on the $\mu_i$ drops out.
\end{rem}


\subsection{Simplicial origin of the \SSSS-action} \label{sec:simplicialSSSS}

The group \SSSS\ is the symmetry group of the tetrahedron (including orientation reversing transformations). That this symmetry group plays a role in our context is not a coincidence. Indeed, recall that in diagrammatic terms the action of the group \SSS\ generated by $\LL$ and $\RR$ on the basic morphism spaces amounts to an \SSS-action on framed 2-simplices. In a similar vein, the maps $\tau_{ij}$ introduced in Definition \ref{def:tauij} can be interpreted diagrammatically as mappings of the framed 3-simplices that realize the basis elements of the spaces $\HFm ijkpql$ and $\HGm ijkpql$. Concretely, these maps amount to transpositions of the vertices of a framed 3-simplex, which are effected on its framed faces by a transposition combined with suitably applying the operations $\LL$ and $\RR$.\,%
 \footnote{~This diagrammatic interpretation is well known. In the setting of spherical fusion categories it is described in Definition 3.10 and Lemma 3.11 of \cite{bawe2}.}
For instance, the map $\tau_{12}$ can be described as
  \be
  \tau_{12} : \quad
  \raisebox{-4.3em} {	  
  \simptex[scale=0.6]{
    \internalthreesimplex[1-2- - ][i,j,k,p,q,l][\gray\alpha,\gray\beta,\gamma,\gray\delta]
  } }
  \hspace*{-0.2em} \longmapsto \hspace*{-0.3em}
  \raisebox{-4.3em} {	  
  \simptex[scale=0.6]{
    \internalGsimplex[2-1- - ][\bar{i},q,k,l,j,p][\LL\alpha,\delta,\gray{{\LL\gamma}},\beta]
  } }
  \label{eq:tau12firstpic}
  \ee
(Here and below the labels of the faces on the back of a tetrahedron are drawn mirrored and in smaller font.)
Thus $\tau_{12}$ turns a tetrahedron of type $\FFm ijklpq$ into one of type $\GGm \ib qkpjl$ and vice versa, exchanging the morphisms $\beta$ and $\delta$ and in addition carrying with it the $\LL$-transformations on the morphisms $\alpha$ and $\gamma$. Similarly descriptions are obtained for $\tau_{23}$ and $\tau_{34}$.

Furthermore, the so-obtained diagrammatic prescription can conveniently be encoded in a suitable 4-simplex. Let us illustrate this again for the case of $\tau_{12}$. 
Consider the following two decompositions of a labeled 4-simplex into two and three, respectively, of the five 3-simplices that make up its boundary:
  \be
  \hspace*{-0.2em} \begin{array}{ll}
  ~ \\[-2.5em]
  & \raisebox{-3.1em}{\begin{turn}{35}\mbox{\Large$\rightsquigarrow$}\end{turn}}
  \raisebox{-3.3em} {
  \simptex[scale=0.6]{
    \threesimplexvertices - - - - 
    \ddraw[red] (0,3) to["\scriptsize\ensuremath{\bar{i}}" swap] (0,0);
    \ddraw[gray] (va.east) to["\scriptsize\ensuremath{l}" swap] (vd.center);
    \ddraw (va.east) to["\scriptsize\ensuremath{q}" swap] (vc.west);
    \ddraw[gray] (vc.west) to["\scriptsize\ensuremath{k}"] (vd.center);
    \draw[dashed,gray] (0,3) to (vd.center);
    \draw[dashed] (0,3) to (vc.west);
  }}
  \hspace*{-0.8em}\cup
  \raisebox{-3.5em} {
  \simptex[scale=0.6]{
    \threesimplexvertices - - - - 
    \draw[red] (3,4.5) to["\ensuremath{}" swap] (0,3);
    \ddraw[gray] (vb.south) to[xshift=-39pt,"\scriptsize\ensuremath{p}"] (vd.center);
    \ddraw (vb.south) to[xshift=-15pt,"\scriptsize\ensuremath{j}"] (vc.west);
    \ddraw[gray] (vc.west) to[xshift=-14pt,yshift=57pt,"\scriptsize\ensuremath{k}"] (vd.center);
    \draw[dashed,gray] (0,3) to (vd.center);
    \draw[dashed] (0,3) to (vc.west);
  }}
  \\[-2.6em]
  \simptex[scale=0.6]{
    \internalthreesimplex[][i,j,k,p,q,l] 
    \draw[red] (3,4.5) to["\ensuremath{}" swap] (0,3);
    \ddraw[red] (0,3) to["\scriptsize\ensuremath{\bar{i}}" swap] (0,0);
    \draw[dashed] (0,3) to (vc.west);
    \draw[dashed,gray] (0,3) to (vd.center);
  }
  \\[-3.8em]
  & \raisebox{1.3em}{\begin{turn}{-30}\mbox{\Large$\rightsquigarrow$}\end{turn}}
  \hspace*{-0.9em}
  \raisebox{-4.3em} {
  \simptex[scale=0.6]{
    \internalthreesimplex[][i,j,k,p,q,l] 
  }}
  \hspace*{-0.4em} \cup \hspace*{-0.2em}
  \raisebox{-4.3em} {
  \simptex[scale=0.6]{
    \threesimplexvertices - - - - 
    \draw[red] (3,4.5) to["\ensuremath{}" swap] (0,3);
    \ddraw[red] (0,3) to["\scriptsize\ensuremath{\bar{i}}" swap] (0,0);
    \ddraw (va.east) to["\scriptsize\ensuremath{i}"] (vb.south);
    \ddraw (vb.south) to["\scriptsize\ensuremath{j}"] (vc.west);
    \ddraw (va.east) to["\scriptsize\ensuremath{q}" swap] (vc.west);
    \draw[dashed] (0,3) to (vc.west);
  }}
  \hspace*{-.8em}\cup
  \raisebox{-3.2em} {
  \simptex[scale=0.6]{
    \threesimplexvertices - - - - 
    \draw[red] (3,4.5) to["\ensuremath{}" swap] (0,3);
    \ddraw[red] (0,3) to["\scriptsize\ensuremath{\Bar{i}}" swap] (0,0);
    \ddraw[gray] (vb.south) to[xshift=-5pt,yshift=-1pt,"\scriptsize\ensuremath{p}"] (vd.center);
    \ddraw[gray] (va.east) to[xshift=9pt,yshift=-51pt,"\scriptsize\ensuremath{l}"] (vd.center);
    \ddraw (va.east) to[xshift=-4pt,yshift=-32pt,"\scriptsize\ensuremath{i}"] (vb.south);
    \draw[dashed,gray] (0,3) to (vd.center);
  }}
  \\[-2.4em]~
  \eear
  \label{eq:4simplex12-decompo}
  \ee
 ~\\[-0.4em]
Accounting for the orientations of the tetrahedra, we can think of the 4-simplex as describing a mapping of its in-boundary, which is given by the three tetrahedra in the second row, to its out-boundary, given by the two tetrahedra in the first row.
 
Moreover, via the identification \eqref{eq:tetrahedra4FandG}, the two tetrahedra that do not contain the unoriented edge (which is labeled by the monoidal unit of \D) correspond to generic \sjs s, while the other three are of a special type:
the second and third tetrahedron in the second row can be recognized as describing the action of $\LL$ on the faces of the first tetrahedron in the second row that are labeled by $\Hd ijq$ and on $\Hm ipl$, respectively, while the second tetrahedron in the first row amounts to $\F \OO jkpjp \oo\beta\beta\oo \eq 1$.
Taken together, this means that the 4-simplex in \eqref{eq:4simplex12-decompo} precisely describes the mapping \eqref{eq:tau12firstpic}.

Analogous considerations apply to the maps $\tau_{23}$ and $\tau_{34}$. In summary, we arrive at the following interpretation of the maps $\tau_{ij}$ (for better readability we suppress some of the structure of the 4-simplices, which can unambiguously be restored):
  \begin{align}
  \hspace*{-0.4em} \tau_{12} ~\hat=
  \raisebox{-4.3em} {
  \simptex[scale=0.6]{
    \internalthreesimplex[1-2- - ][i,j,k,p,q,l][\gray\alpha,\gray\beta,\gamma,\gray\delta]
    \draw (3,4.5) to["\ensuremath{}" swap] (0,3);
    \ddraw (0,3) to["\scriptsize\ensuremath{\bar{i}}" swap] (0,0);
  } }
  \hspace*{-1.1em} :\hspace*{-0.2em}
  \raisebox{-4.3em} {	  
  \simptex[scale=0.6]{
    \internalthreesimplex[1-2- - ][i,j,k,p,q,l][\gray\alpha,\gray\beta,\gamma,\gray\delta]
  } }
  \hspace*{-1.1em} \longmapsto \hspace*{-1.2em}
  \raisebox{-4.3em} {	  
  \simptex[scale=0.6]{
    \internalGsimplex[2-1- - ][\bar{i},q,k,l,j,p][\LL\alpha,\delta,\gray{{\LL\gamma}},\beta]
  } }
  \label{eq:pic:tau12action}
  \end{align}
~\\[-27pt]
  \begin{align}
  \hspace*{-0.4em} \tau_{23} ~\hat= \hspace*{-0.7em}
  \raisebox{-4.3em} {
  \simptex[scale=0.6]{
    \internalthreesimplex[ -2-3- ]%
    [i,j,k,p,q,l]%
    [\gray\alpha,\gray\beta,\gamma,\gray\delta]
    \ddraw (6,3) to["\scriptsize\ensuremath{\bar{j}}" swap] (3,4.5);
    \draw (6,0) to["\ensuremath{}"] (6,3);
  } }
  \hspace*{-0.7em} :~ \hspace*{-0.5em} 
  \raisebox{-4.3em} {
  \simptex[scale=0.6]{
    \internalthreesimplex[ -2-3- ]%
    [i,j,k,p,q,l]%
    [\gray\alpha,\gray\beta,\gamma,\gray\delta]
  } }
  \hspace*{-1.1em} \longmapsto \hspace*{-1.2em}
  \raisebox{-4.3em} {
  \simptex[scale=0.6]{
    \internalGsimplex[ -3-2- ][q,\bar{j},{p},{k},i,{l}]%
    [\delta,\LL\beta,\gray{{\RR\gamma}},\alpha]
  } }
  \label{eq:pic:tau23action}
  \end{align}
~\\[-27pt]
  \begin{align}
  \hspace*{-0.4em} \tau_{34} ~\hat= \hspace*{-0.9em}
  \raisebox{-4.3em} {
  \simptex[scale=0.6]{
    \internalthreesimplex[ - -3-4]%
    [i,j,k,p,q,l]%
    [\gray\alpha,\gray\beta,\gamma,\gray\delta]
    \node (ve) at (4.5,1.4){};
    \ddraw[gray] (vd.center) to["\scriptsize\ensuremath{\bar{k}}"] (ve.center);
    \draw[gray] (6,0) to["\ensuremath{}"] (ve.center);
  } }
  \hspace*{-0.8em} :~ \hspace*{-0.5em} 
  \raisebox{-4.3em} {
  \simptex[scale=0.6]{
    \internalthreesimplex[ - -3-4]%
    [i,j,k,p,q,l]%
    [\gray\alpha,\gray\beta,\gamma,\gray\delta]
  } }
  \hspace*{-1.1em} \longmapsto \hspace*{-1.2em}
  \raisebox{-4.3em} {
  \simptex[scale=0.6]{
    \internalGsimplex[ - -4-3][i,p,\bar{k},j,l,q]%
    [\gamma,\RR\beta,\gray{\alpha},\RR\delta]
  } }
  \label{eq:pic:tau34action}
  \end{align}
for $\alpha \iN \Hm ipl$, $\beta \iN \Hm jkp$, $\gamma \iN \Hm ijq$, $\delta \iN \Hm qkl$.


\subsection{Symmetries of \sjs s}

Next we turn the action of \SSSS\ obtained above into a source of statements about \sjs s. To this end we consider the induced action $(\tau(\Ftot))(-) \,{=}\, \Ftot(\tau^{-1}(-))$ on a suitable function $\Ftot$ on the space $\HtotF$ defined in \eqref{eq:HtotF} with values in \ko. In the definition of this function $\Ftot$ we introduce some factors involving dimensions which are chosen with some hindsight:

\begin{defi} \label{def:Ftot}
The function $\Ftot\colon \HtotF \To \ko$ is the direct sum over $(i,j,k,p,q,l) \iN \S^{\times6}$ of linear maps 
  \be
  \Fm ijklpq : \quad \HFm ijklpq \xrightarrow{\phantom{xx}} \ko
  \qquad \text{and} \qquad
  \Gm ijklqp : \quad \HGm ijklpq \xrightarrow{\phantom{xx}} \ko
  \label{eq:FGmaps2kN}
  \ee
on the direct summands $\HFm ijklpq$ and $\HGm ijklpq$ of $\HtotF$. These maps are defined on basis elements by
  \be
  \bearll &\dsty
  \Fm ijklpq (\beta {\otimes} \Bar\delta {\otimes} \alpha {\otimes} \Bar\gamma)
  := \sqrt{\di l^{}\,\di\lb}\, \F ijklpq \alpha\beta\gamma\delta
  \Nxl3
  \text{and} \qquad & \dsty
  \Gm ijklpq (\Bar\delta {\otimes} \beta {\otimes} \Bar\gamma {\otimes} \alpha)
  := \sqrt{\di l^{}\,\di\lb}\, \G ijklpq \alpha\beta\gamma\delta \,,
  \eear
  \label{eq:def:FmGmN}
  \ee
respectively, and extended by linearity to all of $\HFm ijklpq$ and $\HGm ijklpq$.
\end{defi}

With the chosen ordering of the tensor factors in $\HFm ijklpq$ and $\HGm ijklpq$ (and with our convention that the faces of a tetrahedron are ordered by face-vertex duality), each of the operations $\tau_{ij}$ in particular exchanges the $i$th and $j$th argument of the maps $\Ftot$. Accordingly we have
  \begin{align}
  & \big(\tau_{12}(\Ftot)\big)(\beta {\otimes} \Bar\delta {\otimes} \alpha {\otimes} \Bar\gamma)
  = \Gm \ib qkpjl (\bar\delta {\otimes} \beta {\otimes} \LL\alpha {\otimes} \LL\bar\gamma) \,,
  \label{eq:tau12actionN}
  \\[3pt]
  & \big(\tau_{23}(\Ftot)\big)(\beta {\otimes} \Bar\delta {\otimes} \alpha {\otimes} \Bar\gamma)
  = \Gm q{\jb}plik (\LL\beta {\otimes} \alpha {\otimes} \Bar\delta {\otimes} \RR\Bar\gamma)
  \label{eq:tau23actionN}
  \\[3pt]
  \text{and} \qquad
  & \big(\tau_{34}(\Ftot)\big)(\beta {\otimes} \Bar\delta {\otimes} \alpha {\otimes} \Bar\gamma)
  = \Gm ip{\kb}qlj (\RR\beta {\otimes} \RR\Bar{\delta} {\otimes} \Bar{\gamma} {\otimes} \alpha) 
  \label{eq:tau34actionN}
  \end{align}
for $\beta {\otimes} \Bar\delta {\otimes} \alpha {\otimes} \Bar\gamma \iN \HFm ijklpq$, and analogous formulas are obtained for the values of $\tau_{ij}(\Ftot)$ on the summands $\HGm ijklpq$.
These maps can actually be expressed as simple diagonal matrices acting on the summands of $\HtotF$:

\begin{prop} \label{prop:tauij-explN}
Let \D\ be a \ndgusion category. For $\beta {\otimes} \Bar\delta {\otimes} \alpha {\otimes} \Bar\gamma \iN \HFm ijklpq$ we have
  \be
  ~~~~ \big(\tau_{23}(\Ftot)\big)(\beta {\otimes} \Bar\delta {\otimes} \alpha {\otimes} \Bar\gamma)
  = \Ftot(\beta {\otimes} \Bar\delta {\otimes} \alpha {\otimes} \Bar\gamma) \,,
  \label{tau23=F2N}
  \ee
where $\dd x$ is the relative dimension \eqref{eq:def:dd}.
Moreover, when working with eigenbases for the basic morphism spaces, we have in addition
  \begin{align}
  & \big(\tau_{12}(\Ftot)\big)(\beta {\otimes} \Bar\delta {\otimes} \alpha {\otimes} \Bar\gamma)
  = \sqrt{\dd\ib}\,
  \eps_{i,p;\alpha}^{~l} \, \Ftot(\beta {\otimes} \Bar\delta {\otimes} \alpha {\otimes} \Bar\gamma) 
  \nonumber
  \nxl3
  \text{and} \quad &
  \big(\tau_{34}(\Ftot)\big)(\beta {\otimes} \Bar\delta {\otimes} \alpha {\otimes} \Bar\gamma)
  = \sqrt{\dd k}\, \eps_{q,k;\delta}^{~l} \,
  \Ftot(\beta {\otimes} \Bar\delta {\otimes} \alpha {\otimes} \Bar\gamma) \,,
  \label{tau34=dFMdiagN}
  \end{align}
with $\eps$ the signs defined in \eqref{eq:MMdiag}.
Analogous formulas hold when acting with the maps $\tau_{ij}(\Ftot)$ on the summands $\HGm ijklpq$ of $\HtotF$.
\end{prop}

\begin{proof}
We first rewrite the expressions on the right hand sides of \eqref{eq:tau12actionN}\,--\,\eqref{eq:tau34actionN} in terms of \sjs s. By straightforward manipulations of string diagrams, details of which are presented in Example \ref{exa:g12} in Appendix \ref{sec:app6j}, one finds
  \begin{align}
  & \big(\tau_{12}(\Ftot)\big)(\beta {\otimes} \Bar\delta {\otimes} \alpha {\otimes} \Bar\gamma)
  = \di\ib \, \sqrt{\di p^{}\, \di\pb}\,
  \sumN\mu ipl \MM^{(i\,p\,l)}_{\alpha,\mu} \, \F ijklpq \mu\beta\gamma\delta \,,
  \label{tau12=dMFN}
  \\[2pt]
  & \big(\tau_{23}(\Ftot)\big)(\beta {\otimes} \Bar\delta {\otimes} \alpha {\otimes} \Bar\gamma)
  = \sqrt{\di l^{}\, \di\lb}\, \F ijklpq \alpha\beta\gamma\delta
  \label{tau23=FN}
  \\[-2pt]
  \text{and} \qquad
  & \big(\tau_{34}(\Ftot)\big)(\beta {\otimes} \Bar\delta {\otimes} \alpha {\otimes} \Bar\gamma)
  = \di k \, \sqrt{\di q^{}\, \di\qb}\,
  \sumN\mu qkl \F ijklpq \alpha\beta\gamma\mu \, \, \MM^{(q\,k\,l)}_{\mu,\delta} 
  \label{tau34=dFMN}
  \end{align}
for $\beta {\otimes} \Bar\delta {\otimes} \alpha {\otimes} \Bar\gamma \iN \HFm ijklpq$, with $\MM^{(i\,j\,k)}$ the matrix defined in \eqref{eq:MM2FGversions}. The equality \eqref{tau23=FN} can directly be rewritten as \eqref{tau23=F2N}, while upon invoking \eqref{eq:MMdiag} the equalities \eqref{tau12=dMFN} and \eqref{tau34=dFMN} yield \eqref{tau34=dFMdiagN}.
\end{proof}
 
\begin{rem} 
(i) According to the proof, the rewriting of the right hand sides of \eqref{eq:tau12actionN}\,--\,\eqref{eq:tau34actionN} given in Proposition \ref{prop:tauij-explN} is achieved without any reference to the definition of the maps $\tau_{ij}$. The interpretation \eqref{eq:pic:tau12action}\,--\,\eqref{eq:pic:tau34action} of those maps shows that the so-obtained identities are not accidental, but have a definite geometric origin.
 \\[2pt]
(ii) Inserting the explicit form of the maps $\LL$ and $\RR$ from \eqref{eq:defPandQ-simplicesL}, \eqref{eq:defPandQ-simplicesR} and \eqref{eq:def:cup}, Proposition \ref{prop:tauij-explN} gives the equalities
  \be
  \bearll
  \F ijklpq \alpha\beta\gamma\delta \!\! &\dsty
  = \frac1{\sqrt{\di i^{}\,\di\ib}}\; \eps_{i,p;\alpha}^{~l} \, \sum_{\alpha',\gamma'}
  \F {\ib}ippl\OO{\alpha'}\alpha\oo\oo\, \G {\ib}ijj\OO q\oo\oo{\gamma'}\gamma\,
  \G {\ib}qkpjl{\gamma'}\beta\delta{\alpha'}
  \Nxl3 &\dsty
  = \Fob j \sum_{\beta',\gamma'}
  \F {\jb}jkkp\OO{\beta'}\beta\oo\oo\, \F ij{\jb}i\OO q\oo\oo\gamma{\gamma'}\,
  \G q{\jb}plik{\gamma'}\alpha{\beta'}\delta
  \Nxl1 &\dsty
  = \frac1{\sqrt{\di k^{}\,\di\kb}}\; \eps_{q,k;\delta}^{~l} \, \sum_{\beta',\delta'}
  \G jk{\kb}jp\OO \beta{\beta'}\oo\oo\, \F qk{\kb}q\OO l\oo\oo\delta{\delta'}\,
  \G ip{\kb}qlj\alpha{\delta'}{\beta'}\gamma \,.
  \eear
  \label{eq:F=withLLRR}
  \ee
\end{rem}

When combined with Proposition \ref{prop:genuineS4N}, the previous result implies in particular

\begin{thm} \label{thm:sym}
Let \D\ be a \ndgusion category. Then the linear maps $\tau_{12}$, $\tau_{23}$ and $\tau_{34}$ from Definition \ref{def:tauij} generate an \SSSS-action that leaves the function $\Ftot$ from Definition \ref{def:Ftot} invariant if and only if left and right dimensions coincide and the signs $\eps_\alpha$ defined by the relations \eqref{eq:doubleduals-diag} are all equal to 1.
\end{thm}

Now recall from Remark \ref{rem:whyndgusion} that from any fusion category \C\ we can construct a monoidally equivalent \ndgusion category \D, namely the full subcategory $\D \eq \SC$. Monoidally equivalent fusion categories have the same \sjs s, and thus, upon a coherent choice of the square root in the expression \eqref{eq:MMdiag},
in particular the same eigenvalues $\eps_{i,j;\alpha}^{~k}$ of the involutive matrices \eqref{eq:involmatrix}. Moreover, by making use of the freedom in the specification of the rigidity of \C\ we can achieve that the left and right dimension functions for the induced rigidity of \D\ coincide (i.e.\ that the scalars $\mu_i$ satisfy \eqref{eq:mu-choice}). Thereby we arrive at

\begin{cor} \label{cor:sym}
Let \C\ be a fusion category. Then with a judicious choice of the rigidity functor the linear maps $\tau_{12}$, $\tau_{23}$ and $\tau_{34}$ generate an \SSSS-action that leaves the function $\Ftot$ invariant if and only if the signs $\eps_\alpha$ are all equal to 1.
\end{cor}

\medskip

When expressing the so-obtained tetrahedral relations as identities between numerical values of \sjs s, it is vital to keep in mind that -- as indicated by the extra factors of \sjs s with matrix label $\oo\OO\oo$ in \eqref{eq:F=withLLRR} -- for a given choice of (eigen)bases of the basic morphism spaces, the morphisms $\LL\alpha$ and $\RR\alpha$ need not be basis morphisms again, not even in the multiplicity-free case. It is also worth noting that the identities \eqref{eq:tau12actionN}\,--\,\eqref{eq:tau34actionN}, including $\LL$ and $\RR$, are gauge independent, whereas generically the numbers $\F ijklpq\alpha\beta\gamma\delta$ and $\G ijklpq\alpha\beta\gamma\delta$ do depend on a choice of gauge, compare Remark \ref{rem:gauge}. On the other hand, even if the function $\Ftot$ is not \SSSS-invariant, Proposition \ref{prop:tauij-explN} does provide non-trivial numerical identities between \sjs s.

Still, it would certainly be convenient if, with a suitable gauge choice, one could impose the requirement that for each basis element $\alpha$ of any of the spaces $\Hm ijk$ also $\LL\alpha$ and $\RR\alpha$ are again basis elements. However, this is in general not possible. For instance, our conventions that the basis elements of $\Hm i\OO i$ and $\Hm \OO ii$ are taken to be $\id_i$ while those of $\Hm \ib i\OO$ and  $\Hd i\ib\OO$ are $\oob i$ and $\oodd i$, respectively, are clearly incompatible with this requirement unless $\Fob i \eq 1$.
A partial remedy for this problem is to modify the fundamental convention that the basis of $\Hd ijk$ is covector dual to the one of $\Hm ijk$ by a suitable scalar. Concretely, let us demand that in place of \eqref{eq:dualbasis} we have
  \be
  \alpha' \circ \cc\alpha
  = \delta_{\alpha,\alpha'}^{}\, \sqrt{\frac{\di i\,\di j}{\di k}}\, \id_{k}
  \label{eq:newcodual}
  \ee
for $\alpha,\alpha' \iN \Hm ijk$. (This is indeed a standard convention in the quantum and condensed matter physics literature, see e.g.\ Appendix A.2 of \cite{szbv} for a recent exposition.) 
Switching to the convention \eqref{eq:newcodual} leads to modifications in various identities. These boil down to making the replacement $\F ijklpq\alpha\beta\gamma\delta \,{\xmapsto{\,\,\,}} \sqrt{\di i\,\di j\,\di k\,\di\lb}\, \F ijklpq\alpha\beta\gamma\delta$ in formulas like \eqref{eq:F=o.abcd.o,}.

The prescription \eqref{eq:newcodual} indeed eliminates the incompatibility problem with the basis choices in the spaces $\Hm i\OO i$ etc. However, in general neither \eqref{eq:newcodual} nor any other choice can lead to a basis such that any basis vector $\alpha$ is mapped to basis covectors by $\LL$ and $\RR$. This is demonstrated by the following example.

\begin{exa}
Consider the fusion categories over \complex\ with three isoclasses of simple objects, $\S \eq \{ \OO, x, y \}$, and with tensor products
  \be
  y \oti y \cong \OO \,, \qquad x \oti y \cong x \cong y \oti x \,, \qquad
  x \oti x \cong \OO \oplus x \oplus x \oplus y \,.
  \ee
Up to monoidal equivalence there are four such fusion categories; their \sjs s were computed in \cite{haHon}. There exists a gauge choice for which the numerical values of all \sjs s are elements of the number field $\mathbb Q[\sqrt3,\sqrt{-1}]$, and these are given in \Cite{App.\,A}{haHon} for one of the four equivalence classes of categories; the other three cases are obtained as Galois conjugates. With these values one finds e.g.\ $\di y \eq 1$, $\di x \eq 1  + \sqrt3$ and $\eps_{x,x;\alpha}^{~x} \eq 1$ independently of $\alpha$.
 \\[2pt]
As is shown by direct calculation in \Cite{Sect.\,4.2}{hong}, there does not exist any basis choice in the spaces $\Hm xxx$ and $\Hd xxx$ for which the two basis elements $\alpha\iN\Hm xxx$ both map to basis elements via $\LL$ and $\RR$ at the same time.
\end{exa}

On the other hand, there are many fusion categories for which a suitable choice of gauge and covector duality reduces the tetrahedral symmetries to equalities between \sjs s on the nose. As a demonstration, let us restrict our attention to the case that $\N ijk \iN \{0,1\}$ and $\eps_{i,j}^{~k} \eq 1$ for all $i,j,k \iN \S^{\times 3}$, assume that the left and right dimensions coincide, and impose the convention \eqref{eq:newcodual}. In this case the morphisms $\LL\alpha$ and $\RR\alpha$ are basis elements for each basis element $\alpha$ of any of the one-dimensional spaces $\Hm ijk$ if and only if there is a gauge choice for which the equalities
  \be
  \Fm {\ib}ikkp\OO = \Fm ik{\kb}i\OO p = \Gm ik{\kb}ip\OO = \Gm {\ib}ikk\OO p
  = \sqrt{\frac{\di p}{\di i\,\di k}}
  \label{eq:F=G=sqrt(d/dd)}
  \ee
hold for all $i,k,p \iN \S$. And if this is the case, then Proposition \ref{prop:tauij-explN} reduces to
  \be
  \FFm ijklpq = \sqrt{\frac{\di p\,\di q}{\di j\,\di l}}\, \GGm \ib qkpjl
  = \sqrt{\frac{\di p\,\di q}{\di i\,\di k}}\, \GGm q{\jb}plik
  = \sqrt{\frac{\di p\,\di q}{\di j\,\di l}}\, \GGm ip{\kb}qlj .
  \label{eq:S4onF}
  \ee

Well known categories for which one can achieve \eqref{eq:F=G=sqrt(d/dd)} are the following:

\begin{exa}
Consider the fusion categories over \complex\ with two isoclasses of simple objects, $\S \eq \{ \OO, \ta \}$, and with tensor product $\ta \oti \ta \cong \OO \oplus \ta$. Up to monoidal equivalence there are two such fusion categories \Cite{Sect.\,2.5}{ostr3}. They can be characterized in terms of the general solution to the pentagon equations: With the convention \eqref{eq:Fi0k=1}, the non-trivial \sjs s are \Cite{Lemma\,5}{boDa}
  \be
  \Fte = 1 \qquad\text{and}\qquad
  \Fta = \left( \begin{array}{cc} -\,a & -\,a\,b^{-1} \Nxl2 b & a \eear \right) ,
  \label{eq:Fte,Fta}
  \ee
with $b \iN \complex$ an arbitrary invertible number and $a \iN \complex$ a solution of the quadratic equation $a^2 \eq a + 1$ (and with the first row and column of $\Fta$ referring to the monoidal unit). Note that $\Fta$ is involutive.
 \\
The two equivalence classes of fusion categories correspond to the two solutions $a_\pm \eq (1 \,{\pm}\, \sqrt5)/2$ for $a$. The one with $a \eq a_-$, often called the \emph{Fibonacci category}, admits a unitary structure, while the one with $a \eq a_+$, known as the \emph{Yang-Lee category}, does not. The gauge choice for which the matrix $\Fta$ of the Fibonacci category is unitary is given by $b \eq \beta\,\sqrt{-a_-}$ with $|\beta| \eq 1$.
Inserting the \sjs s \eqref{eq:Fte,Fta} into \eqref{eq:dimL,dimR} and \eqref{eq:MM2FGversions} we have e.g.\ $\di\ta \eq -a^{-1}$ and $\MM^{(\ta\,\ta\,\ta)} \eq {-}a$ and thus, by \eqref{eq:MMdiag}, $\eps_{\ta,\ta}^{~\ta} \eq 1$.
It is then readily checked that some of the tetrahedral relations \eqref{eq:S4onF} are fulfilled automatically, while those involving the \sjs s $\Fta_{\one,\ta}$ or $\Fta_{\ta,\one}$ are satisfied iff
  \be
  \Fta_{\one,\ta} = \sqrt{-a} = \Fta_{\ta,\one} ,
  \ee
i.e.\ iff the free parameter $b$ is fixed to $b \eq \sqrt{-a}$. Note that in the unitary case this means that the gauge choice is fixed further in such a way that $\beta \eq 1$.
\end{exa}

 \vskip 3em
  
 \noindent
{\sc Acknowledgements:}\\[.3em]
We thank Bruce Bartlett and Laurens Lootens for helpful comments on an earlier version of the manuscript.
JF is supported by VR under project no.\ 2017-03836.

\newpage
\appendix

\section{\sjs s} \label{sec:app6j}

Here we collect a few useful facts about the \sjs s $\F ijklpq\alpha\beta\gamma\delta$ of a \gusion category \C. Some of these are still valid beyond the class of \gusion categories (compare Remark \ref{rem:gusion}), but for brevity we will not spell this out. For further specific information about \sjs s see e.g.\ Chapter VI of \cite{TUra} (in the context of unimodal modular tensor categories), Section 4.9 of \cite{EGno} (in the context of semisimple tensor categories), and Section 2.1 of \cite{ushe} (in the context of fusion categories).

Let us first present the defining equality \eqref{eq:ab=F.cd-} in terms of string diagrams: we have
    \def\locpa  {0.82}   
    \def\locpb  {0.58}   
  \be
  \raisebox{-3.8em} {\begin{tikzpicture}
  \coordinate (cooalpha) at (0.5*\locpb,2*\locpa);
  \coordinate (coobeta)  at (1.5*\locpb,\locpa);
  \draw[\colorObject,line width=\widthObject]
              (0,0) node[below=-1pt]{$i$} -- (cooalpha) node[left]{$\alpha$}
              (\locpb,0) node[below=-1pt]{$j$} -- (coobeta) node[right]{$\beta$}
              (2*\locpb,0) node[below=-1pt]{$k$} -- (coobeta)
              (coobeta) -- node[right]{$p$} (cooalpha)
              (cooalpha) -- (\locpb,3.1*\locpa) node[above=-1pt]{$l$};
  \basmorph {cooalpha}; \basmorph {coobeta};
  \node at (3*\locpb+1.5,\locpa) {$=~\dsty\sum_{q,\gamma,\delta}\F ijklpq\alpha\beta\gamma\delta $};
  \begin{scope}[shift={(3*\locpb+3.3,0)}]
  \coordinate (coogamma) at (0.5*\locpb,\locpa);
  \coordinate (coodelta) at (1.5*\locpb,2*\locpa);
  \draw[\colorObject,line width=\widthObject]
              (0,0) node[below=-1pt]{$i$} -- (coogamma) node[left]{$\gamma$}
              (\locpb,0) node[below=-1pt]{$j$} -- (coogamma)
              (2*\locpb,0) node[below=-1pt]{$k$} -- (coodelta) node[right]{$\delta$}
              (coogamma) -- node[left]{$q$} (coodelta)
              (coodelta) -- (\locpb,3.1*\locpa) node[above=-1pt]{$l$};
  \basmorph {coogamma}; \basmorph {coodelta};
  \end{scope}
  \end{tikzpicture}}
  \label{eq:pic:ab=Fcd}
  \ee
Recall that when using string diagrams, \C\ is tacitly assumed to be strict. Thus any inherent non-triviality of the \sjs s has to arise from the choice of composition maps. Indeed, since the two binary operations of morphism composition and tensor product are strictly associative by themselves, the \sjs s can be thought of as a measure of the mutual associativity of all the binary operations involved in diagram \eqref{eq:pic:ab=Fcd}.

The numbers $\F ijklpq\alpha\beta\gamma\delta$ are naturally combined into matrices, with the rows and columns labeled by the multilabels $\alpha p\beta$ and $\gamma q\delta$, respectively. This way we get square matrices $\FM ijkl$ of size $\N{ij}kl \,{:=}\,\sumI p \N{j}{k}{p} \N{i}{p}{l} \eq \sumI q \N{i}{j}{q}\N{q}{k}{l}$. Henceforth we only consider the case that $\N{ij}kl$ is non-zero. Then $\FM ijkl$ is invertible; we denote the inverse by\,%
 \footnote{~This should not be mixed up with the use of the symbol $\mathrm G$ for a rescaled version of $\mathrm F$, as e.g.\ in \Cite{App.\,B}{leWe} and  \Cite{App.\,B}{haWo}.}
$\GM ijkl$ and its matrix elements by $\G ijklpq\alpha\beta\gamma\delta$. In terms of string diagrams, the latter obey
    \def\locpa  {0.82}   
    \def\locpb  {0.58}   
  \be
  \raisebox{-3.8em} {\begin{tikzpicture}
  \coordinate (cooalpha) at (0.5*\locpb,\locpa);
  \coordinate (coobeta) at (1.5*\locpb,2*\locpa);
  \draw[\colorObject,line width=\widthObject]
              (0,0) node[below=-1pt]{$i$} -- (cooalpha) node[left]{$\alpha$}
              (\locpb,0) node[below=-1pt]{$j$} -- (cooalpha)
              (2*\locpb,0) node[below=-1pt]{$k$} -- (coobeta) node[right]{$\beta$}
              (cooalpha) -- node[left]{$p$} (coobeta)
              (coobeta) -- (\locpb,3.1*\locpa) node[above=-1pt]{$l$};
  \basmorph {cooalpha}; \basmorph {coobeta};
  \node at (3*\locpb+1.5,\locpa) {$=~\dsty\sum_{q,\gamma,\delta}\G ijklpq\alpha\beta\gamma\delta $};
  \begin{scope}[shift={(3*\locpb+3.3,0)}]
  \coordinate (coogamma) at (0.5*\locpb,2*\locpa);
  \coordinate (coodelta)  at (1.5*\locpb,\locpa);
  \draw[\colorObject,line width=\widthObject]
              (0,0) node[below=-1pt]{$i$} -- (coogamma) node[left]{$\gamma$}
              (\locpb,0) node[below=-1pt]{$j$} -- (coodelta) node[right]{$\delta$}
              (2*\locpb,0) node[below=-1pt]{$k$} -- (coodelta)
              (coodelta) -- node[right]{$q$} (coogamma)
              (coogamma) -- (\locpb,3.1*\locpa) node[above=-1pt]{$l$};
  \basmorph {coogamma}; \basmorph {coodelta};
  \end{scope}
  \end{tikzpicture}}
  \label{eq:pic:ab=Gcd}
  \ee

The same coefficients appear when comparing compositions of the spaces $\Hd pqr$ instead of $\Hm pqr$. With the convention \eqref{eq:dualbasis} on dual bases we have
    \def\locpa  {0.82}   
    \def\locpb  {0.58}   
  \be
  \raisebox{-4.0em} {\begin{tikzpicture}
  \coordinate (coogamma) at (0.5*\locpb,-\locpa);
  \coordinate (coodelta) at (1.5*\locpb,-2*\locpa);
  \draw[\colorObject,line width=\widthObject]
              (0,0) node[above=-1pt]{$i$} -- (coogamma) node[left]{$\cc\gamma$}
              (\locpb,0) node[above=-1pt]{$j$} -- (coogamma)
              (2*\locpb,0) node[above=-1pt]{$k$} -- (coodelta) node[right]{$\cc\delta$}
              (coogamma) -- node[left]{$q$} (coodelta)
              (coodelta) -- (\locpb,-3.1*\locpa) node[below=-1pt]{$l$};
  \basmorph {coogamma}; \basmorph {coodelta};
  \node at (3*\locpb+1.5,-1.7*\locpa) {$=~\dsty\sum_{p,\alpha,\beta}\F ijklpq\alpha\beta\gamma\delta $};
  \begin{scope}[shift={(3*\locpb+3.6,0)}]
  \coordinate (cooalpha) at (0.5*\locpb,-2*\locpa);
  \coordinate (coobeta)  at (1.5*\locpb,-\locpa);
  \draw[\colorObject,line width=\widthObject]
              (0,0) node[above=-1pt]{$i$} -- (cooalpha) node[left]{$\cc\alpha$}
              (\locpb,0) node[above=-1pt]{$j$} -- (coobeta) node[right]{$\cc\beta$}
              (2*\locpb,0) node[above=-1pt]{$k$} -- (coobeta)
              (coobeta) -- node[right]{$p$} (cooalpha)
              (cooalpha) -- (\locpb,-3.1*\locpa) node[below=-1pt]{$l$};
  \basmorph {cooalpha}; \basmorph {coobeta};
  \end{scope}
  \end{tikzpicture}}
  \label{eq:pic:ab=Fcd-dual}
  \ee
and
    \def\locpa  {0.82}   
    \def\locpb  {0.58}   
  \be
  \raisebox{-4.1em} {\begin{tikzpicture}
  \coordinate (coogamma) at (0.5*\locpb,-2*\locpa);
  \coordinate (coodelta) at (1.5*\locpb,-\locpa);
  \draw[\colorObject,line width=\widthObject]
              (0,0) node[above=-1pt]{$i$} -- (coogamma) node[left]{$\cc\gamma$}
              (\locpb,0) node[above=-1pt]{$j$} -- (coodelta) node[right]{$\cc\delta$}
              (2*\locpb,0) node[above=-1pt]{$k$} -- (coodelta)
              (coodelta) -- node[right]{$q$} (coogamma)
              (coogamma) -- (\locpb,-3.1*\locpa) node[below=-1pt]{$l$};
  \basmorph {coogamma}; \basmorph {coodelta};
  \node at (3*\locpb+1.3,-1.8*\locpa) {$=~\dsty\sum_{p,\alpha,\beta}\G ijklpq\alpha\beta\gamma\delta $};
  \begin{scope}[shift={(3*\locpb+3.3,0)}]
  \coordinate (cooalpha) at (0.5*\locpb,-\locpa);
  \coordinate (coobeta)  at (1.5*\locpb,-2*\locpa);
  \draw[\colorObject,line width=\widthObject]
              (0,0) node[above=-1pt]{$i$} -- (cooalpha) node[left]{$\cc\alpha$}
              (\locpb,0) node[above=-1pt]{$j$} -- (cooalpha)
              (2*\locpb,0) node[above=-1pt]{$k$} -- (coobeta) node[right]{$\cc\beta$}
              (cooalpha) -- node[left]{$p$} (coobeta)
              (coobeta) -- (\locpb,-3.1*\locpa) node[below=-1pt]{$l$};
  \basmorph {cooalpha}; \basmorph {coobeta};
  \end{scope}
  \end{tikzpicture}}
  \ee

Composing the defining formula \eqref{eq:ab=F.cd-} with the morphism
$(\cc\gamma' \oti \id_{X_k}) \cir \cc\delta'$ leads to
  \be
  \F ijklpq\alpha\beta\gamma\delta \, \id_{l}
  = \alpha \circ (\id_{i} \oti \beta) \circ (\cc\gamma \oti \id_{k}) \cir \cc\delta \,.
  \label{eq:F=a.b.c.d}
  \ee
Tensoring this equality with $\id_{\Bar l}$ and pre- and post-composing further with the dual basis elements of the one-dimensional spaces $\Hd l{\Bar l}1$ and $\Hm l{\Bar l}1$, or alternatively with those of $\Hd l{\Bar l}1$ and $\Hm l{\Bar l}1$, yields the two expressions
    \def\locpa  {0.78}   
    \def\locpb  {0.88}   
  \be
  \raisebox{-5.3em} {\begin{tikzpicture}
  \node at (-4.7*\locpb,2.5*\locpa) {$\F ijklpq\alpha\beta\gamma\delta ~=$};
  \coordinate (cooalpha) at (-\locpb,4*\locpa);
  \coordinate (coobeta)  at (0,3*\locpa);
  \coordinate (coogamma) at (-\locpb,2*\locpa);
  \coordinate (coodelta) at (0,\locpa);
  \coordinate (cooct1)   at (-2*\locpb,3*\locpa);  
  \coordinate (cooct2)   at (\locpb,2*\locpa);     
  \coordinate (cooct3)   at (-3*\locpb,2*\locpa);  
  \coordinate (coobot)   at (-1.5*\locpb,0);
  \coordinate (cootop)   at (-2*\locpb,5*\locpa);
  \draw[\colorObject,line width=\widthObject]
          (coobot) -- node[right,yshift=-2pt]{$l$} (coodelta) node[right]{$\cc\delta$}
          -- node[left,yshift=-2pt]{$q$} (coogamma) node[left]{$\cc\gamma$}
          -- node[left,yshift=2pt]{$j$} (coobeta) node[right]{$\beta$}
          -- node[right,yshift=2pt]{$p$} (cooalpha) node[right]{$\alpha$}
          -- node[right,yshift=2pt]{$l$} (cootop)
          (coogamma) .. controls (cooct1) .. node[left=-1pt]{$i$} (cooalpha)
          (coodelta) .. controls (cooct2) .. node[right=-2pt]{$k$} (coobeta)
          (coobot) node[below]{$\oo$} .. controls (cooct3) .. node[left=-1pt]{$\Bar l$}
          (cootop) node[above]{$\oo$};
  \basmorph {cooalpha}; \basmorph {coobeta}; \basmorph {coogamma}; \basmorph {coodelta};
  \basmorph {cootop}; \basmorph {coobot};
  \begin{scope}[shift={(5.2*\locpb,0)}]
  \node at (-3.0*\locpb,2.5*\locpa) {$=$};
  \coordinate (cooalpha) at (-\locpb,4*\locpa);
  \coordinate (coobeta)  at (0,3*\locpa);
  \coordinate (coogamma) at (-\locpb,2*\locpa);
  \coordinate (coodelta) at (0,\locpa);
  \coordinate (cooct1)   at (-2*\locpb,3*\locpa);  
  \coordinate (cooct2)   at (\locpb,2*\locpa);     
  \coordinate (cooct3)   at (3*\locpb,2*\locpa);  
  \coordinate (coobot)   at (1.2*\locpb,0);
  \coordinate (cootop)   at (0.5*\locpb,5*\locpa);
  \draw[\colorObject,line width=\widthObject]
          (coobot) -- node[below,xshift=-2pt]{$l$} (coodelta) node[left]{$\cc\delta$}
          -- node[left,yshift=-2pt]{$q$} (coogamma) node[left]{$\cc\gamma$}
          -- node[left,yshift=2pt]{$j$} (coobeta) node[right]{$\beta$}
          -- node[right,yshift=2pt]{$p$} (cooalpha) node[left]{$\alpha$}
          -- node[left,yshift=3pt]{$l$} (cootop)
          (coogamma) .. controls (cooct1) .. node[left=-1pt]{$i$} (cooalpha)
          (coodelta) .. controls (cooct2) .. node[right=-2pt]{$k$} (coobeta)
          (coobot) node[below]{$\oo$} .. controls (cooct3) .. node[right=-1pt]{$\Bar l$}
          (cootop) node[above]{$\oo$};
  \basmorph {cooalpha}; \basmorph {coobeta}; \basmorph {coogamma}; \basmorph {coodelta};
  \basmorph {cootop}; \basmorph {coobot};
  \end{scope}
  \end{tikzpicture}}
  \label{eq:F=o.abcd.o,}
  \ee
for $\F ijklpq\alpha\beta\gamma\delta$. Analogously we have
    \def\locpa  {0.78}   
    \def\locpb  {0.88}   
  \be
  \raisebox{-5.2em} {\begin{tikzpicture}
  \node at (-4.1*\locpb,2.5*\locpa) {$\G ijklpq\alpha\beta\gamma\delta ~=$};
  \coordinate (cooalpha) at (0,3*\locpa);
  \coordinate (coobeta)  at (\locpb,4*\locpa);
  \coordinate (coogamma) at (0,\locpa);
  \coordinate (coodelta) at (\locpb,2*\locpa);
  \coordinate (cooct1)   at (-\locpb,2*\locpa);    
  \coordinate (cooct2)   at (2*\locpb,3*\locpa);   
  \coordinate (cooct3)   at (-2.3*\locpb,2*\locpa);  
  \coordinate (coobot)   at (-1.2*\locpb,0);
  \coordinate (cootop)   at (-0.5*\locpb,5*\locpa);
  \draw[\colorObject,line width=\widthObject]
          (coobot) -- node[right,yshift=-4pt]{$l$} (coogamma) node[right,yshift=-2pt]{$\cc\gamma$}
          -- node[right,yshift=-3pt]{$q$} (coodelta) node[right,yshift=-1pt]{$\cc\delta$}
          -- node[left,yshift=-3pt]{$j$} (cooalpha) node[left,yshift=1pt]{$\alpha$}
          -- node[left=-2pt,yshift=3pt]{$p$} (coobeta) node[right,yshift=1pt]{$\beta$}
          -- node[right,yshift=4pt]{$l$} (cootop)
          (coogamma) .. controls (cooct1) .. node[left=-1pt]{$i$} (cooalpha)
          (coodelta) .. controls (cooct2) .. node[right=-2pt]{$k$} (coobeta)
          (coobot) node[below]{$\oo$} .. controls (cooct3) .. node[left=-1pt]{$\Bar l$}
          (cootop) node[above]{$\oo$};
  \basmorph {cooalpha}; \basmorph {coobeta}; \basmorph {coogamma}; \basmorph {coodelta};
  \basmorph {cootop}; \basmorph {coobot};
  \begin{scope}[shift={(5.2*\locpb,0)}]
  \node at (-2.0*\locpb,2.5*\locpa) {$=$};
  \coordinate (cooalpha) at (0,3*\locpa);
  \coordinate (coobeta)  at (\locpb,4*\locpa);
  \coordinate (coogamma) at (0,\locpa);
  \coordinate (coodelta) at (\locpb,2*\locpa);
  \coordinate (cooct1)   at (-\locpb,2*\locpa);    
  \coordinate (cooct2)   at (2*\locpb,3*\locpa);   
  \coordinate (cooct3)   at (3*\locpb,2*\locpa);  
  \coordinate (coobot)   at (1.5*\locpb,0);
  \coordinate (cootop)   at (2.0*\locpb,5*\locpa);
  \draw[\colorObject,line width=\widthObject]
          (coobot) -- node[left,yshift=-1pt]{$l$} (coogamma) node[left,yshift=-2pt]{$\cc\gamma$}
          -- node[right,yshift=-3pt]{$q$} (coodelta) node[right,yshift=-1pt]{$\cc\delta$}
          -- node[left,yshift=-3pt]{$j$} (cooalpha) node[left,yshift=1pt]{$\alpha$}
          -- node[left=-2pt,yshift=3pt]{$p$} (coobeta) node[left,yshift=1pt]{$\beta$}
          -- node[left,yshift=3pt]{$l$} (cootop)
          (coogamma) .. controls (cooct1) .. node[left=-1pt]{$i$} (cooalpha)
          (coodelta) .. controls (cooct2) .. node[right=-2pt]{$k$} (coobeta)
          (coobot) node[below]{$\oo$} .. controls (cooct3) .. node[right=-1pt]{$\Bar l$}
          (cootop) node[above]{$\oo$};
  \basmorph {cooalpha}; \basmorph {coobeta}; \basmorph {coogamma}; \basmorph {coodelta};
  \basmorph {cootop}; \basmorph {coobot};
  \end{scope}
  \end{tikzpicture}}
  \label{eq:G=o.abcd.o,}
  \ee

Here, and henceforth, the special symbol ``$\oo$'' stands for the chosen basis of a one-dimensional morphism space. One particular situation in which such one-dimensional spaces occur is the case that $k \eq l \eq i$, $j \eq \Bar i$ and $p \eq q \eq 1$; as already done in the Introduction, we abbreviate the corresponding \sjs s, and likewise the corresponding $\GG$-coefficients, as
  \be
  \Fo i \equiv \Foo i \qquand \Go i \equiv \Goo i \,.
  \label{eq:FoGo}
  \ee
As a special case of \eqref{eq:F=a.b.c.d}, these numbers satisfy
    \def\locpa  {1.18}   
    \def\locpb  {0.68}   
    \def\locpc  {30}     
    \def\locpC  {50}     
  \be
  \raisebox{-4.9em} {\begin{tikzpicture}
  \node at (-4.4*\locpb,0.3*\locpa) {$\Fo k$};
  \node at (-1.7*\locpb,0.3*\locpa) {$=$};
  \coordinate (coomid)  at (0,0);             
  \coordinate (cootopl) at (-\locpb,1.7*\locpa);  
  \coordinate (cootopr) at (\locpb,\locpa);   
  \coordinate (coobotr) at (1.7*\locpb,-\locpa); 
  \begin{scope}[shift={(-2.8*\locpb,0)}]
  \draw[\colorObject,line width=\widthObject] (0,-\locpa) node[below=-0.5pt] {$k$}
              -- (0,1.7*\locpa) node[above=0.6pt] {$k$};
  \end{scope}
  \draw[\colorObject,line width=\widthObject] (coomid) [out=90-\locpc,in=270-\locpc] to
              node[above,midway,xshift=-2.1pt] {$\kb$} (cootopr);
  \draw[\colorObject,line width=\widthObject] (coomid) [out=90+\locpc,in=270] to (cootopl)
              node[above=0.6pt] {$k$};
  \draw[\colorObject,line width=\widthObject] (cootopr) [out=270+\locpc,in=90] to (coobotr)
              node[below=-0.5pt,xshift=-1pt] {$k$};
  \basmorph {coomid}; \basmorph {cootopr};
  \node[below=1pt] at (coomid) {$\oo$};
  \node[above=1pt] at (cootopr) {$\oo$};
  \node at (6.4*\locpb-1.2,0.3*\locpa) {and};
  \begin{scope}[shift={(11.4*\locpb+1.2,0)}]
  \node at (-5.2*\locpb,0.3*\locpa) {$\Go k$};
  \node at (-2.5*\locpb,0.3*\locpa) {$=$};
  \coordinate (coomid)  at (0,0);             
  \coordinate (cootopl) at (-\locpb,\locpa);  
  \coordinate (cootopr) at (\locpb,1.7*\locpa);   
  \coordinate (coobotl) at (-1.7*\locpb,-\locpa); 
  \begin{scope}[shift={(-3.6*\locpb,0)}]
  \draw[\colorObject,line width=\widthObject] (0,-\locpa) node[below=-0.5pt] {$k$}
              -- (0,1.7*\locpa) node[above=0.6pt] {$k$};
  \end{scope}
  \draw[\colorObject,line width=\widthObject] (coomid) [out=90+\locpc,in=270+\locpc] to
              node[above,midway,xshift=2.4pt] {$\kb$} (cootopl);
  \draw[\colorObject,line width=\widthObject] (coomid) [out=90-\locpc,in=270] to (cootopr)
              node[above=0.6pt] {$k$};
  \draw[\colorObject,line width=\widthObject] (cootopl) [out=270-\locpc,in=90] to (coobotl)
              node[below=-0.5pt,xshift=-1pt] {$k$};
  \basmorph {coomid}; \basmorph {cootopl};
  \node[below=1pt] at (coomid) {$\oo$};
  \node[above=1pt] at (cootopl) {$\oo$};
  \end{scope}
  \end{tikzpicture}}
  \label{eq:pic:Fo,Go}
  \ee
from which we also get the special cases
    \def\locpa  {1.18}   
    \def\locpb  {0.68}   
    \def\locpc  {30}     
    \def\locpC  {50}     
  \be
  \Fo k ~=~ \Gob k ~=~~
  \raisebox{-3.5em}{\begin{tikzpicture}
  \coordinate (coomid)  at (0,0);             
  \coordinate (cootopl) at (-\locpb,\locpa);  
  \coordinate (cootopr) at (\locpb,\locpa);   
  \coordinate (coobot)  at (0,-\locpa);       
  \draw[\colorObject,line width=\widthObject] (coomid) [out=90+\locpc,in=270+\locpc] to
              node[above,midway,xshift=1.3pt] {$k$} (cootopl);
  \draw[\colorObject,line width=\widthObject] (coomid) [out=90-\locpc,in=270-\locpc] to
              node[above,midway,xshift=-1.1pt] {$\Bar k$} (cootopr);
  \draw[\colorObject,line width=\widthObject] (coobot) [out=90+\locpc,in=270-\locpC] to
              node[below,midway,yshift=-4pt] {$\Bar k$} (cootopl);
  \draw[\colorObject,line width=\widthObject] (coobot) [out=90-\locpc,in=270+\locpC] to
              node[below,midway,yshift=-3pt] {$k$} (cootopr);
  \basmorpho {coomid}; \basmorphO {cootopl}; \basmorphO {cootopr}; \basmorpho {coobot};
  \end{tikzpicture}
  }
  ~~=~~ \raisebox{-3.5em}{\begin{tikzpicture}
  \coordinate (coomid)  at (0,0);             
  \coordinate (cootop)  at (0,\locpa);        
  \coordinate (coobotl) at (-\locpb,-\locpa); 
  \coordinate (coobotr) at (\locpb,-\locpa);  
  \draw[\colorObject,line width=\widthObject] (coomid) [out=270-\locpc,in=90-\locpc] to
              node[above,midway,xshift=-1.3pt] {$\Bar k$} (coobotl);
  \draw[\colorObject,line width=\widthObject] (coomid) [out=270+\locpc,in=90+\locpc] to
              node[above,midway,xshift=1.3pt] {$k$} (coobotr);
  \draw[\colorObject,line width=\widthObject] (cootop) [out=270-\locpC,in=90+\locpc] to
              node[above,midway,yshift=4pt] {$k$} (coobotl);
  \draw[\colorObject,line width=\widthObject] (cootop) [out=270+\locpC,in=90-\locpc] to
              node[above,midway,yshift=5pt] {$\Bar k$} (coobotr);
  \basmorphO {coomid}; \basmorphO {cootop}; \basmorpho {coobotl}; \basmorpho {coobotr};
  \end{tikzpicture}
  }
  \label{eq:pic:FkkkkGkkkk}
  \ee
of \eqref{eq:F=o.abcd.o,} and \eqref{eq:G=o.abcd.o,}.

The pentagon identity for the associator $a$ of \C\ translates into a collection of polynomial equations for the \sjs s. While we do not need to refer to these equations explicitly in this paper, for completeness we present them anyway:
  \be
  \sumI w \sumN\kappa wsy \sumN\lambda pwx \sumN\eta qrw \F pqrxwu \lambda\eta\alpha\delta\,
  \F pwstyx\mu\kappa\lambda\nu\, \F qrsyvw\beta\gamma\eta\kappa
  = \sumN\sigma uvt \F pqvtyu\mu\beta\alpha\sigma\, \F urstvx\sigma\gamma\delta\nu
  \label{eq:pentagonF,}
  \ee
for all $p,q,r,s,t,u,v,x,y \iN \S$.

\begin{rem} \label{rem:gauge}
The choice of bases in the morphism spaces $\Hm ijk$ and $\Hd ijk$ is arbitrary. Clearly, the numerical values of \sjs s generically depend on these choices. Following physics parlance, altering the choice is often called a \emph{gauge transformation}.
We reserve this term for those transformations which preserve the covector duality, in the sense that the transformed basis vector is sent to the transformed covector under the covector duality map. (In case $\ko \eq \complex$, this means that the basis transformation is unitary.) 
As an example, if $\Bar i \,{\ne}\, i$, then by a gauge transformation that treats the bases $\ood i$ of $\Hm{\Bar i}i\OO$ and $\oodd i$ of $\Hm i{\Bar i}\OO$ differently, we can change the value of $\Fo i$; the value of $\Go i \eq \Fob i$ will then change as well, in such a way that the product $\Fo i\,\Go i$ is unchanged, i.e.\ is gauge independent.
If, on the other hand, $\Bar i \eq i$, then already $\Fo i \eq \Go i$ is gauge independent.
\end{rem}

\begin{rem} \label{rem:rescaledFi}
Trivially, $\Fob k \eq \Fo k$ if $\kb \eq k$. Provided that $\Fob k \,{\ne}\,0$, this equality can also be achieved for $\kb \,{\ne}\, k$ by a suitable rescaling (which generically, unlike a gauge transformation, modifies the convention for the covector duality).
Indeed, we may then multiply the basis morphism $\oodd k$ by $\xi_k \,{:=}\, (\Fo k\!/\Fob k)^{1/4}$ (with some choice of root) and correspondingly rescale $\oobb k$ by $\xi_k^{-1}$. The resulting redefined values of the \sjs s \eqref{eq:pic:FkkkkGkkkk}, which we indicate by a tilde, obey 
  \be
  \widetilde\FF^{(\kb\,k\,\kb)\,\kb}_{\OO,\OO}
  = \sqrt{\Fo k\,\Fob k} = \widetilde\FF^{(k\,\kb\,k)\,k}_{\OO,\OO} \,.
  \ee
\end{rem}

One can utilize the gauge freedom so as to get convenient values for specific \sjs s and thereby simplify some relations. In particular, the spaces $\HomC(i \oti 1,i)$ etc.\ are canonically isomorphic to $\EndC(i) \eq \ko\,\id_i$, and as we take the monoidal unit to be strict we can identify
  \be
  \HomC(i \oti 1,i) = \HomC(1 \oti i,i) = \HomC(i,i \oti 1) = \HomC(i,1 \oti i) 
  = \EndC(i) \,.
  \label{Homi1,i=Endi}
  \ee
This allows us to make the following

\begin{conv}
We take the identity morphism $\id_i$ as the basis of each of the one-di\-men\-sional spaces
\eqref{Homi1,i=Endi}.
\end{conv}

\noindent
With this convention we have
  \be
  \F ij\OO kpq\alpha\oo\beta\oo = \F k\OO jipq\alpha\oo\oo\beta
  = \F \OO kijpq\oo\alpha\oo\beta = \delta_{p,j}\,\delta_{q,k}\,\delta_{\alpha,\beta} \,.
  \label{eq:Fi0k=1}
  \ee

We finally mention further specific examples of the use of \sjs s. 

\begin{exa}
The numbers $\MM_{\alpha,\beta}^{(i\,j\,k)}$ defined in \eqref{eq:MM2FGversions} for $\alpha,\beta \iN \Hm ijk$ can be expressed in terms of string diagrams as follows:
    \def\locpa  {1.18}   
    \def\locpb  {0.68}   
    \def\locpc  {30}     
  \be
  \MM_{\alpha,\beta}^{(i\,j\,k)} ~:=~~
  \raisebox{-7.3em}{ \begin{tikzpicture}
  \coordinate (cooalpha) at (2*\locpb,-0.6*\locpa);   
  \coordinate (coobeta) at (2*\locpb,0.6*\locpa);     
  \coordinate (cootop)  at (0.5*\locpb,1.4*\locpa);   
  \coordinate (cootopp) at (0,2.3*\locpa);            
  \coordinate (coobot)  at (0.5*\locpb,-1.4*\locpa);  
  \coordinate (coobott) at (0,-2.3*\locpa);           
  \draw[\colorObject,line width=\widthObject]
              (cooalpha) -- node[midway,right=-2pt] {$k$} (coobeta)
              [out=90+\locpc,in=270+\locpc] to node[left=-3pt,yshift=-5pt] {$i$} (cootop)
              [out=270-\locpc,in=90+\locpc] to node[midway,left=-1pt] {$\ib$} (coobot)
              [out=90-\locpc,in=270-\locpc] to node[left=-3pt,yshift=5pt] {$i$} (cooalpha);
  \draw[\colorObject,line width=\widthObject]
              (coobeta) [out=90-1.8*\locpc,in=270+1.3*\locpc] to node[midway,yshift=8pt] {$j$} (cootopp)
              [out=270-\locpc,in=90+\locpc] to node[midway,left=-1pt] {$\jb$} (coobott)
              [out=90-1.3*\locpc,in=270+1.8*\locpc] to node[midway,yshift=-8pt] {$j$} (cooalpha);
  \basmorph {cooalpha}; \node[right=0.7pt] at (cooalpha) {$\scriptstyle\alpha$};
  \basmorph {coobeta};  \node[right=0.7pt] at (coobeta) {$\scriptstyle\cc\beta$};
  \basmorphO {cootop}; \basmorphO {cootopp}; \basmorpho {coobot}; \basmorpho {coobott};
  \end{tikzpicture}
  }
  =~~ \raisebox{-4.8em}{ \begin{tikzpicture}
  \coordinate (cooalpha) at (0,-0.6*\locpa);           
  \coordinate (coobeta) at (0,0.6*\locpa);             
  \coordinate (cootop)  at (-1.4*\locpb,1.5*\locpa);   
  \coordinate (cootopp) at (1.4*\locpb,1.5*\locpa);    
  \coordinate (coobot)  at (-1.4*\locpb,-1.5*\locpa);  
  \coordinate (coobott) at (1.4*\locpb,-1.5*\locpa);   
  \draw[\colorObject,line width=\widthObject]
              (cooalpha) -- node[midway,right=-2pt] {$k$} (coobeta)
              [out=90+\locpc,in=270+\locpc] to node[right=-3pt,yshift=5pt] {$i$} (cootop)
              [out=270-1.2*\locpc,in=90+1.2*\locpc] to node[midway,right=-1pt] {$\ib$} (coobot)
              [out=90-\locpc,in=270-\locpc] to node[right=-3pt,yshift=-5pt] {$i$} (cooalpha);
  \draw[\colorObject,line width=\widthObject]
              (coobeta) [out=90-\locpc,in=270-\locpc] to node[midway,yshift=8pt] {$j$} (cootopp)
              [out=270+1.2*\locpc,in=90-1.2*\locpc] to node[midway,left=-1pt] {$\jb$} (coobott)
              [out=90+\locpc,in=270+\locpc] to node[midway,yshift=-9pt] {$j$} (cooalpha);
  \basmorph {cooalpha}; \node[right=0.7pt] at (cooalpha) {$\scriptstyle\alpha$};
  \basmorph {coobeta};  \node[right=0.7pt] at (coobeta) {$\scriptstyle\cc\beta$};
  \basmorphO {cootop}; \basmorphO {cootopp}; \basmorpho {coobot}; \basmorpho {coobott};
  \end{tikzpicture}
  }
  ~~= \raisebox{-7.3em}{ \begin{tikzpicture}
  \coordinate (cooalpha) at (-2*\locpb,-0.6*\locpa);   
  \coordinate (coobeta) at (-2*\locpb,0.6*\locpa);     
  \coordinate (cootop)  at (-0.5*\locpb,1.4*\locpa);   
  \coordinate (cootopp) at (0,2.3*\locpa);             
  \coordinate (coobot)  at (-0.5*\locpb,-1.4*\locpa);  
  \coordinate (coobott) at (0,-2.3*\locpa);            
  \draw[\colorObject,line width=\widthObject]
              (cooalpha) -- node[midway,right=-2pt] {$k$} (coobeta)
              [out=90-\locpc,in=270-\locpc] to node[right=-3pt,yshift=-5pt] {$j$} (cootop)
              [out=270+\locpc,in=90-\locpc] to node[midway,right=-1pt] {$\jb$} (coobot)
              [out=90+\locpc,in=270+\locpc] to node[right=-3pt,yshift=5pt] {$j$} (cooalpha);
  \draw[\colorObject,line width=\widthObject]
              (coobeta) [out=90+1.8*\locpc,in=270-1.3*\locpc] to node[midway,yshift=6pt] {$i$} (cootopp)
              [out=270+\locpc,in=90-\locpc] to node[midway,right=-1pt] {$\ib$} (coobott)
              [out=90+1.3*\locpc,in=270-1.8*\locpc] to node[midway,yshift=-8pt] {$i$} (cooalpha);
  \basmorph {cooalpha}; \node[left=0.7pt] at (cooalpha) {$\scriptstyle\alpha$};
  \basmorph {coobeta};  \node[left=0.7pt] at (coobeta) {$\scriptstyle\cc\beta$};
  \basmorphO {cootop}; \basmorphO {cootopp}; \basmorpho {coobot}; \basmorpho {coobott};
  \end{tikzpicture}
  }
  \label{eq:MM3versions}
  \ee
Equality of the three morphisms described by these string diagrams is an immediate consequence of covector duality.
(In the setting of \cite{bart6}, the equality amounts to \Cite{Lemma\,3.9}{bart6}.
In terms of simplicial diagrams, the morphisms \eqref{eq:MM3versions} are realized by two tetrahedra that share a face.)
\end{exa}

\begin{exa} \label{exa:g12}
We may use the formulas just obtained to arrive at convenient expressions for the numbers
  \be
  \begin{array}{rl}
  & g_{12} := \sqrt{ 1 / \di p^{}\, \di\pb}\;
  \Gm \ib qkpjl (\Bar\delta {\otimes} \beta {\otimes} \LL\alpha {\otimes} \LL\Bar\gamma) \,,
  \Nxl3
  & g_{23} := \sqrt{ 1 / \di l^{}\, \di\lb}\;
  \Gm q{\jb}plik (\LL\beta {\otimes} \alpha {\otimes} \Bar\delta {\otimes} \RR\Bar\gamma)
  \Nxl3
  \text{and} \quad 
  & g_{34} := \sqrt{ 1 / \di q^{}\, \di\qb}\;
  \Gm ip{\kb}qlj (\RR\beta {\otimes} \RR\Bar\delta {\otimes} \Bar\gamma {\otimes} \alpha)
  \eear
  \label{eq:def:g12}
  \ee
with $\mathbb G$ the functions defined by \eqref{eq:def:FmGmN}, and with the square root factors compensating the ones in $\mathbb G$, which are not needed for the discussion here.
In terms of string diagrams we can express $g_{12}$ as
    \def\locpa  {0.78}   
    \def\locpb  {0.88}   
  \be
  g_{12} ~= \quad
  \raisebox{-5.8em} {\begin{tikzpicture}
  \coordinate (cooalpha) at (0,3*\locpa);
  \coordinate (coobeta)  at (\locpb,4*\locpa);
  \coordinate (coogamma) at (0,\locpa);
  \coordinate (coodelta) at (\locpb,2*\locpa);
  \coordinate (cooct1)   at (-\locpb,2*\locpa);    
  \coordinate (cooct2)   at (2*\locpb,3*\locpa);   
  \coordinate (cooct3)   at (-2.3*\locpb,2*\locpa);  
  \coordinate (coobot)   at (-1.2*\locpb,0);
  \coordinate (cootop)   at (-0.5*\locpb,5*\locpa);
  \draw[\colorObject,line width=\widthObject]
          (coobot) -- node[right,yshift=-4pt]{$p$} (coogamma) node[right,yshift=-2pt]{$\LL\alpha$}
          -- node[right,yshift=-3pt]{$l$} (coodelta) node[right,yshift=-1pt]{$\cc\delta$}
          -- node[left,yshift=-3pt]{$q$} (cooalpha) node[left,yshift=2pt]{$\LL\cc\gamma$}
          -- node[left=-2pt,yshift=3pt]{$j$} (coobeta) node[right,yshift=1pt]{$\beta$}
          -- node[right,yshift=4pt]{$p$} (cootop)
          (coogamma) .. controls (cooct1) .. node[left=-1pt]{$\ib$} (cooalpha)
          (coodelta) .. controls (cooct2) .. node[right=-2pt]{$k$} (coobeta)
          (coobot) .. controls (cooct3) .. node[left=-1pt]{$\cc p$} (cootop);
  \basmorph {cooalpha}; \basmorph {coobeta}; \basmorph {coogamma}; \basmorph {coodelta};
  \basmorphO {cootop}; \basmorpho {coobot};
  \end{tikzpicture}
  }
  \quad \stackrel{\eqref{eq:dimL,dimR},\,\eqref{eq:defPandQ-simplicesL}} = ~~ \di\ib ~
  \raisebox{-7.0em} {\begin{tikzpicture}
  \coordinate (cooalpha) at (-\locpb,-4*\locpa);   
  \coordinate (coobeta)  at (0,-3*\locpa);
  \coordinate (coogamma) at (-\locpb,-2*\locpa);
  \coordinate (coodelta) at (0,-\locpa);
  \coordinate (cooct1)   at (-3*\locpb,-3*\locpa);  
  \coordinate (cooct2)   at (\locpb,-2*\locpa);       
  \coordinate (cooct3)   at (-4.1*\locpb,-\locpa);    
  \coordinate (cooct4)   at (-4.1*\locpb,-5.7*\locpa);  
  \coordinate (coobot)   at (-1.7*\locpb,0.7*\locpa);
  \coordinate (coobotm)  at (-1.8*\locpb,-0.7*\locpa);
  \coordinate (cootop)   at (-0.6*\locpb,-6.1*\locpa);
  \coordinate (cootopm)  at (-1.8*\locpb,-4.8*\locpa);
  \draw[\colorObject,line width=\widthObject]
          (coobot) -- node[right,yshift=3pt]{$p$} (coodelta) node[right,yshift=2pt]{$\beta$}
          -- node[right=-4pt,yshift=-4pt]{$j$} (coogamma) node[left,yshift=-2pt]{$\cc\gamma$}
          -- node[left,yshift=-2pt]{$q$} (coobeta) node[right,yshift=-2pt]{$\cc\delta$}
          -- node[right,yshift=-2pt]{$l$} (cooalpha) node[right,yshift=-2pt]{$\alpha$}
          -- node[right=-2pt,yshift=1pt]{$p$} (cootop)
          (coogamma) -- node [right=-1pt,yshift=3pt]{$i$} (coobotm)
          (cooalpha) -- node [left=-2pt,yshift=4pt]{$i$} (cootopm)
          (coobotm) .. controls (cooct1) .. node[left=-2pt,yshift=9pt]{$\ib$} (cootopm)
          (coodelta) .. controls (cooct2) .. node[right=-3pt]{$k$} (coobeta)
          (coobot) .. controls (cooct3) and (cooct4) .. node[left=-1pt]{$\pb$} (cootop);
  \basmorph {cooalpha}; \basmorph {coobeta}; \basmorph {coogamma}; \basmorph {coodelta};
  \basmorpho {cootop}; \basmorphO {coobot}; \basmorpho {cootopm}; \basmorphO {coobotm};
  \end{tikzpicture}
  }
  \ee
Noting that
    \def\locpa  {0.78}   
    \def\locpb  {0.88}   
  \be
  \raisebox{-5.4em} {\begin{tikzpicture}
  \coordinate (cooalpha) at (0,-4.4*\locpa);          
  \coordinate (coobeta)  at (0,-3*\locpa);
  \coordinate (coogamma) at (-\locpb,-2*\locpa);
  \coordinate (coodelta) at (0,-\locpa);
  \coordinate (cooct2)   at (\locpb,-2*\locpa);       
  \coordinate (coobot)   at (0,0.4*\locpa);
  \coordinate (coobotm)  at (-1.4*\locpb,0.4*\locpa);
  \draw[\colorObject,line width=\widthObject]
          (coobot) node[above=-1pt]{$p$} -- (coodelta) node[right,yshift=1pt]{$\beta$}
          -- node[right=-4pt,yshift=-5pt]{$j$} (coogamma) node[left,yshift=-2pt]{$\cc\gamma$}
          -- node[left=-2pt,yshift=-5pt]{$q$} (coobeta) node[right,yshift=-2pt]{$\cc\delta$}
          -- (cooalpha) node[below=-1pt]{$l$}
          (coogamma) [out=-240,in=-90] to (coobotm) node [above=-1pt]{$i$}
          (coodelta) .. controls (cooct2) .. node[right=-2pt]{$k$} (coobeta);
  \basmorph {coobeta}; \basmorph {coogamma}; \basmorph {coodelta};
  \end{tikzpicture}
  }
  \quad = \quad \sumN\nu ipl \F ijklpq \nu\beta\gamma\delta\, \cc\nu \,, 
  \label{eq:aG2Fidentity}
  \ee
this can be rewritten as
  \be
  g_{12} = \di\ib \, \sumN\mu ipl \F ijklpq \mu\beta\gamma\delta\, \MM^{(i\,p\,l)}_{\alpha,\mu}
  \label{g12=dFM}
  \ee
with $\MM$ as in \eqref{eq:MM2FGversions} and \eqref{eq:MM3versions}.
A very similar calculation can be performed for $g_{34}$. We have
    \def\locpa  {0.78}   
    \def\locpb  {0.88}   
  \be
  g_{34} ~= \quad
  \raisebox{-5.2em} {\begin{tikzpicture}
  \coordinate (cooalpha) at (-\locpb,-4*\locpa);
  \coordinate (coobeta)  at (0,-3*\locpa);
  \coordinate (coogamma) at (-\locpb,-2*\locpa);
  \coordinate (coodelta) at (0,-\locpa);
  \coordinate (cooct1)   at (-2*\locpb,-3*\locpa);  
  \coordinate (cooct2)   at (\locpb,-2*\locpa);     
  \coordinate (cooct3)   at (-3*\locpb,-2*\locpa);  
  \coordinate (coobot)   at (-1.5*\locpb,0);
  \coordinate (cootop)   at (-2*\locpb,-5*\locpa);
  \draw[\colorObject,line width=\widthObject]
          (coobot) -- node[right,yshift=3pt]{$q$} (coodelta) node[right,yshift=3pt]{$\RR\cc\delta$}
          -- node[left,yshift=3pt]{$l$} (coogamma) node[left,yshift=2pt]{$\alpha$}
          -- node[left=-2pt,yshift=-4pt]{$p$} (coobeta) node[right=-2pt,yshift=-4pt]{$\RR\beta$}
          -- node[right=-1pt,yshift=-3pt]{$j$} (cooalpha) node[right,yshift=-3pt]{$\cc\gamma$}
          -- node[right=-2pt,yshift=-3pt]{$q$} (cootop)
          (coogamma) .. controls (cooct1) .. node[left=-2pt]{$i$} (cooalpha)
          (coodelta) .. controls (cooct2) .. node[right=-2pt]{$\Bar k$} (coobeta)
          (coobot) .. controls (cooct3) .. node[left=-2pt]{$\Bar q$} (cootop) ;
  \basmorph {cooalpha}; \basmorph {coobeta}; \basmorph {coogamma}; \basmorph {coodelta};
  \basmorpho {cootop}; \basmorphO {coobot};
  \end{tikzpicture}
  }
  \quad \stackrel{
	  \eqref{eq:dimL,dimR},\,\eqref{eq:defPandQ-simplicesR}} = ~~ \di k~
  \raisebox{-7.0em} {\begin{tikzpicture}
  \coordinate (cooalpha) at (-\locpb,-4*\locpa);
  \coordinate (coobeta)  at (0,-3*\locpa);
  \coordinate (coogamma) at (-\locpb,-2*\locpa);
  \coordinate (coodelta) at (-0.3*\locpb,-\locpa);
  \coordinate (cooct1)   at (-2*\locpb,-3*\locpa);      
  \coordinate (cooct2)   at (1.8*\locpb,-2.1*\locpa);   
  \coordinate (cooct3)   at (-3.6*\locpb,-0.6*\locpa);  
  \coordinate (cooct4)   at (-3.6*\locpb,-5.2*\locpa);  
  \coordinate (coobot)   at (-1.5*\locpb,0.9*\locpa);
  \coordinate (cootop)   at (-0.3*\locpb,-5.7*\locpa);
  \coordinate (coobotr)  at (0.6*\locpb,0.1*\locpa);
  \coordinate (cootopr)  at (1.0*\locpb,-4.1*\locpa);
  \draw[\colorObject,line width=\widthObject]
          (coobot) -- node[right=-3pt,yshift=4pt]{$q$} (coodelta) node[right,yshift=-3pt]{$\cc\delta$}
          -- node[left=-3pt,yshift=5pt]{$l$} (coogamma) node[left,yshift=2pt]{$\alpha$}
          -- node[left=-2pt,yshift=-5pt]{$p$} (coobeta) node[right,yshift=2pt]{$\beta$}
          -- node[left=-2pt,yshift=3pt]{$j$} (cooalpha) node[right=-1pt,yshift=-2pt]{$\cc\gamma$}
          -- node[right=-2pt,yshift=3pt]{$q$} (cootop)
          (coogamma) .. controls (cooct1) .. node[left=-2pt]{$i$} (cooalpha)
          (coodelta) -- node[left=-3pt,yshift=3pt]{$k$} (coobotr)
          .. controls (cooct2) .. node[right=-1pt]{$\Bar k$} (cootopr)
          -- node[left=-3pt,yshift=-3pt]{$k$} (coobeta)
          (coobot) .. controls (cooct3) and (cooct4) .. node[left=-2pt,yshift=-6pt]{$\Bar q$} (cootop) ;
  \basmorph {cooalpha}; \basmorph {coobeta}; \basmorph {coogamma}; \basmorph {coodelta};
  \basmorpho {cootop}; \basmorphO {coobot}; \basmorphO {coobotr}; \basmorpho {cootopr};
  \end{tikzpicture}
  }
  \ee
By a similar identity as \eqref{eq:aG2Fidentity}, this is rewritten as
  \be
  g_{34} = \di k \, \sumN\mu qkl \F ijklpq \alpha\beta\gamma\mu \, \MM^{(q\,k\,l)}_{\mu,\delta} .
  \ee
Finally for $g_{23}$ we have
  \be
  g_{23} ~= \quad
  \raisebox{-5.2em} {\begin{tikzpicture}
  \coordinate (cooalpha) at (-\locpb,-4*\locpa);
  \coordinate (coobeta)  at (0,-3*\locpa);
  \coordinate (coogamma) at (-\locpb,-2*\locpa);
  \coordinate (coodelta) at (0,-\locpa);
  \coordinate (cooct1)   at (-2*\locpb,-3*\locpa);  
  \coordinate (cooct2)   at (\locpb,-2*\locpa);     
  \coordinate (cooct3)   at (-3*\locpb,-2*\locpa);  
  \coordinate (coobot)   at (-1.5*\locpb,0);
  \coordinate (cootop)   at (-2*\locpb,-5*\locpa);
  \draw[\colorObject,line width=\widthObject]
          (coobot) -- node[right,yshift=4pt]{$l$} (coodelta) node[right,yshift=1pt]{$\alpha$}
          -- node[left,yshift=3pt]{$i$} (coogamma) node[left,yshift=2pt]{$\RR\cc\gamma$}
          -- node[left,yshift=-4pt]{$\Bar j$} (coobeta) node[right,yshift=-3pt]{$\LL\beta$}
          -- node[right,yshift=-2pt]{$k$} (cooalpha) node[right,yshift=-3pt]{$\cc\delta$}
          -- node[right,yshift=-2pt]{$l$} (cootop)
          (coogamma) .. controls (cooct1) .. node[left=-2pt]{$q$} (cooalpha)
          (coodelta) .. controls (cooct2) .. node[right=-2pt]{$p$} (coobeta)
          (coobot) .. controls (cooct3) .. node[left=-1pt]{$\Bar l$} (cootop) ;
  \basmorph {cooalpha}; \basmorph {coobeta}; \basmorph {coogamma}; \basmorph {coodelta};
  \basmorpho {cootop}; \basmorphO {coobot};
  \end{tikzpicture}
  }
  \quad \stackrel{\eqref{eq:defPandQ-simplicesL},\,\eqref{eq:defPandQ-simplicesR}} = ~~ \frac1{\Fob j} ~~
  \raisebox{-7.5em} {\begin{tikzpicture}
  \coordinate (cooalpha) at (0.5*\locpb,4.9*\locpa);
  \coordinate (coobeta)  at (0.5*\locpa,3.2*\locpa);
  \coordinate (coogamma) at (-1.3*\locpb,1.8*\locpa);
  \coordinate (coodelta) at (-0.5*\locpa,0.1*\locpa);
  \coordinate (cooct1)   at (-1.8*\locpb,4.1*\locpa);  
  \coordinate (cooct2)   at (1.2*\locpb,0.7*\locpa);   
  \coordinate (cooct3)   at (-3.5*\locpb,2*\locpa);    
  \coordinate (cooct4)   at (-1.9*\locpb,0.8*\locpa);  
  \coordinate (cooct5)   at (1.4*\locpb,4.1*\locpa);   
  \coordinate (coobot)   at (-1.7*\locpb,-0.9*\locpa);
  \coordinate (cootop)   at (-1.6*\locpb,5.9*\locpa);
  \coordinate (coomidl)  at (-0.7*\locpb,3.6*\locpa);
  \coordinate (coomidr)  at (-0.1*\locpb,1.4*\locpa);
  \draw[\colorObject,line width=\widthObject]
          (coobot) -- node[right=-2pt,yshift=-4pt]{$l$} (coodelta) node[right,yshift=-4pt]{$\cc\delta$}
          .. controls (cooct4) .. node[left=-2pt,yshift=-4pt]{$q$}
             (coogamma) node[left,yshift=-1pt]{$\cc\gamma$}
          -- node[left=-5pt,yshift=4pt]{$j$} (coomidl)
          -- node[left=-2pt,yshift=-2pt]{$\jb$} (coomidr)
          -- node[right=-3pt,yshift=-3pt]{$j$} (coobeta) node[right,yshift=-2pt]{$\beta$}
          .. controls (cooct5) .. node[right=-2pt,yshift=2pt]{$p$}
             (cooalpha) node[right,yshift=3pt]{$\alpha$}
          -- node[right,yshift=5pt]{$l$} (cootop)
          (coogamma) .. controls (cooct1) .. node[left=-1pt]{$i$} (cooalpha)
          (coodelta) .. controls (cooct2) .. node[right=-2pt]{$k$} (coobeta)
          (coobot) .. controls (cooct3) .. node[left=-1pt]{$\lb$} (cootop) ;
  \basmorph {cooalpha}; \basmorph {coobeta}; \basmorph {coogamma}; \basmorph {coodelta};
  \basmorphO {cootop}; \basmorpho {coobot}; \basmorphO {coomidl}; \basmorpho {coomidr};
  \end{tikzpicture}
  }
  \ee
Invoking \eqref{eq:pic:Fo,Go} this gives
  \be
  g_{23} = \F ijklpq \alpha\beta\gamma\delta \,.
  \ee
\end{exa}


\section{From simplicial sets to fusion categories} \label{sec:app-A}

One of the ingredients in our study of \sjs s is the graphical calculus based on simplicial diagrams. This calculus is somewhat less standard than the string diagram calculus that is described in Appendix \ref{sec:app6j}. Both diagrammatics have their own advantages and disadvantages. The present appendix provides the proper categorical setting for the simplicial semantic.
In more fancy terms, what we are doing in this appendix is to formulate facts from simplicial homology of the geometric realization of linear monoidal categories in the light of the $(\infty,2)$-ca\-te\-gorical structure of a category enriched over a \Vectk-enriched category.  Comprehensive information on these topics is given in \cite{JOya} and \cite{RIve},

\subsection{Simplicial sets and quasi-categories}

The natural domain of simplicial semantic is the theory of quasi-categories. As this is not commonly used in the fusion categorical context that we consider in the main text, we offer a brief general introduction to this topic, before specializing to linear monoidal categories.

\begin{defi} \label{def:Delta}
(i) The \textit{simplex category} $\Delta$ is a skeleton of the category of non-empty finite ordered sets. The objects of $\Delta$ are denoted by $\mathbf{n} \eq \{0,1,2,...\,,n\}$. The morphisms between them are order-preserving functions.
\\[2pt]
(ii) The \textit{face} maps $\delta_i$ and \textit{degeneracy} maps $\sigma_j$ are those morphisms
  \be
  \raisebox{-1.3em}{\begin{tikzpicture}
  \draw (0,0) node{$\mathbf{n}$};
  \draw (2.2,0) node{$\mathbf{n{+}1}$};
  \draw[right hook->] (0.27,0.12) to ["$\delta_i$"] +(1.35,0);
  \draw[<<-] (0.27,-0.12) to [swap,"$\sigma_j$"] +(1.35,0);
  \end{tikzpicture}}
  \ee
in $\Delta$ which \textit{skip} the $i$th element of an ordered set and \textit{identify} $j$ with $j{+}1$, respectively.
\end{defi}

The face and degeneracy maps satisfy $\sigma_i \cir \delta_i \eq \id$, and they generate the category $\Delta$. Note that here and below we slightly abuse notation by not specifying the domains  $\mathbf{n}$ and and {\bf n+1}, respectively, of these morphisms. That is, each of the symbols $\delta_i$ and $\sigma_j$ can refer to a multitude of morphisms which coincide up to inclusion maps.

\begin{defi}
A \textit{simplicial set} is a contravariant functor from $\Delta$ to the category of sets.
\end{defi}

We may think of a simplicial set $C$ as a countable family $\{C_n\}$ of sets $C_n \,{\equiv}\, C(\textbf{n})$ together with the data of what the generators of $\Delta$ are mapped to. By the Yoneda lemma, the elements of the set $C_n$ are determined by natural transformations from $\Hom_{\Delta}(-,\mathbf n)$ to $C$. We may visualize $\Hom_{\Delta}(-,\mathbf n)$ as follows. We draw an ordered $n$-simplex in such a way that its $m$-subsimplices represent the iterated face maps in $\Hom_{\Delta}(\mathbf m,\mathbf n)$. We omit labels for the non-injective maps, since those factor through the subsimplices by restriction. Specifically, for $n \eq 3$ this prescription looks as follows.
  \begin{align}
  \begin{split}
 \Hom_\Delta(0,3) ~= \hspace*{-0.5em}
  \raisebox{-3.3em} {\simptex[scale=0.6]{
  \internalthreesimplex[0-1-2-\gray{3}]}}
 \qquad \Hom_\Delta(1,3) ~= \hspace*{-0.5em}
  \raisebox{-3.7em} {\simptex[scale=0.6]{
  \internalthreesimplex[][01,12,23,13,02,03]}}
\\
 \Hom_\Delta(2,3) ~= \hspace*{-0.5em}
  \raisebox{-3.7em} {\simptex[scale=0.6]{
  \internalthreesimplex[][][\gray{013},\gray{123},012,\gray{023}]}}
 \qquad~~ \Hom_\Delta(3,3) ~= \hspace*{-0.5em}
  \raisebox{-3.7em} {\simptex[scale=0.6]{
  \internalthreesimplex[][][][] }}
  \end{split}
  \label{eq:Hom03..Hom33}
  \end{align}

To facilitate the appreciation of the three-dimensional structure of the simplices, here and below we emphasize that structure by drawing the ``back" of a simplex in greyscale and displaying their labels laterally reversed.

Next, to visualize each of the sets $C_n$, we replace the labels by the values they take under a natural transformation $\phi \colon \Hom_{\Delta}(-,\mathbf n) \,{\rightarrow}\, C$. Following the standard Yoneda Lemma argument, we use the naturality condition $\phi_m(f) \eq C(f)(\phi_n(\id_{\mathbf n}))$ 
arising from the diagram
  \begin{align}
  \begin{tikzcd}[ampersand replacement=\&]
  \Hom_\Delta(\textbf{n},\textbf{n}) \ar[rr,"{\Hom_\Delta(f,\textbf{n})}"] \ar[d,swap,"\phi_n"] 
  \& \& 
  \Hom_\Delta(\textbf{m},\textbf{n}) \ar[d,"\phi_m"]
  \\[5pt]
  C(\textbf{n}) \ar[rr,"C(f)"] 
  \& \& 
  C(\textbf{m})
  \end{tikzcd}
  \end{align}
to conclude that each such natural transformation $\phi$ is specified by a unique element $\phi_n(\id_{\mathbf{n}}) \iN C_n$.

For $n \iN \{1,2,3\}$, two alternative ways of executing this visualization are shown in the following diagrams:
  \begin{align}
  C_1 \ni~ f ~=~ 
  \raisebox{-0.3em}{\onesimplex[f_0-f_1][f_{01}]}
  & ~=~ \raisebox{-0.3em}{\onesimplex[f^1-f^0][f]}
  \\
  C_2 \ni~ \gamma ~=~
  \raisebox{-2.5em}{\simptex[scale=0.9]{
  \internaltwosimplex[\gamma_0-\gamma_1-\gamma_2][\gamma_{01},\gamma_{12},\gamma_{02}][\gamma_{012}]
  }}
  & ~=~ \raisebox{-2.5em}{\simptex[scale=0.9]{
  \internaltwosimplex[\gamma_0-\gamma_1-\gamma_2][\gamma^{2},\gamma^{0},\gamma^{1}][\gamma]
  }}
  \\
  C_3 \ni~ F ~= \hspace*{-0.9em}
  \raisebox{-5.0em}{\simptex[scale=0.8]{
  \internalthreesimplex[F_0-F_1-F_2-\gray{F_3}][F_{01},F_{12},F_{23},F_{13},F_{02},F_{03}][\gray{F_{013}},\gray{F_{123}},F_{012},\gray{F_{023}}]
  \draw(3,2) node{$\mathbf{F_{0123}}$};
  }}
  \hspace*{-0.3em} & = \hspace*{-0.3em}
  \raisebox{-5.0em}{\simptex[scale=0.8]{
  \internalthreesimplex[F_0-F_1-F_2-\gray{F_3}][F_{01},F_{12},F_{23},F_{13},F_{02},F_{03}][\gray{F^{2}},\gray{F^{0}},F^{3},\gray{F^{1}}]
  \draw(3,2) node{$\mathbf{F}$};
  }}
  \end{align}
(Here the use of superscripts versus subscripts indicates the face-vertex duality ordering when applicable.)

As is apparent from these pictures, we quickly run out of dimensions with which to visualize the geometries involved, and out of space needed to project the labels. It is therefore advisable to omit some of the labels when they are unambiguous from the context and we only wish to access some of the data related to a given natural transformation. In case more precision is needed, we may alternatively draw a diagram multiple times, each time labeling subsimplices of different dimension, as we have indeed done for $\Hom(-,\textbf{n})$ in the pictures above. In practice, this amounts to giving the simplex a cell decomposition as a topological object.

Further terminology that is useful when discussing simplicial sets is as follows:

\begin{defi}
(i) A \textit{face} of a simplex is a subsimplex of codimension one. We use the bijection furnished by face-vertex duality to define an ordering of faces according to which 0-simplex they do \textit{not} contain; faces are then given by face maps as $C(\delta_i)(F) \eq F^i$.
\\[2pt]
(ii)
A \textit{degenerate} simplex is an image of a degeneracy map, which we visualize by extending the simplex with equality symbols as labels. As an illustration, we have
  \begin{align}
  C(\sigma_2):~~ 
  \raisebox{-0.3em}{\onesimplex[x-y][f]}
  ~~\longmapsto~ \raisebox{-0.7em}{\simptex[baseline=3mm]{\internaltwosimplex[x-y-y][f,=,f][=]}}
  \end{align}
\\[2pt]
(iii) The \textit{boundary} $\partial X$ of a simplex $X$ is the union of all its faces. The only additional information needed to recover a simplex from its boundary is its \textit{filling}.
\\[2pt]
(iv) An ordered sequence $\{f^n\}$ of $1$-simplices is \textit{compatible} if the target of any one of them matches the source of the next, i.e.\ if $f^n_1 \eq f^{n+1}_0$.
\\[2pt]
(v) The \textit{spine} of a simplex is the longest compatible ordered sequence of 1-subsimplices. The spine of an $n$-simplex has length $n$.
\end{defi}

We are now ready to introduce the notion of a quasi-category.

\begin{defi}
(i) A \textit{k-horn} of a simplex is a union of all faces except the $k$th face. A $k$-horn is called \textit{inner} if $k$ is neither 0 nor $n$.
\\[2pt]
(ii) A \textit{quasi-category} is a simplicial set for which every inner horn has at least one filling.
\\[2pt]
(iii) In a quasi-category, a compatible ordered sequence of $1$-simplices is called \textit{composable},
as it can be recovered as the spine of some $n$-simplex for which the $0n$-edge is considered a \textit{composite}.
\end{defi}

In the rest of this appendix we assume the simplicial sets that we consider to be quasi-categories. The justification for the terminology ``horn" and ``quasi-category" comes from the ability to fill -- albeit non-uniquely -- horn-shaped composition-like diagrams.
For instance, 
  \begin{align}
  \text{for any diagram of type}~~
  \raisebox{-0.2em}{\simptex[scale=0.7,baseline=4mm]{
  \twosimplexvertices x-y-z- 
  \ddraw (nx.center) to["$f$"] (ny.center);
  \ddraw (ny.center) to["$g$"] (nz.center);
  }}
  ~~\text{there exists a filling}~~
  \raisebox{-0.2em}{\simptex[scale=0.7,baseline=4mm]{
  \internaltwosimplex[x-y-z][f,g,h][\gamma]
  }}
  \end{align}
Here $h$ can be considered a composite of $f$ and $g$. Note, however, that it is not \textit{the} composite, as there may be many different ones. Thus this ability to fill horns does not quite make the quasi-category into a category. The connection with ordinary category theory is made more precise by the \textit{homotopy-nerve adjunction} between quasi-categories and categories, which is furnished by the following two definitions.

\begin{defi} \Cite{Def.\,1.1.7}{RIve} Two $1$-simplices
\raisebox{-8pt}{\ensuremath{
    \begin{tikzpicture}[scale=0.7,baseline=-4mm]
    \draw (-0.05,0) node{\textit{x}};
    \draw (2.05,0) node{\textit{y}};
    \draw[thick, ->] (1.75,0.13) -- node[above=-1pt]{$\scriptstyle f$} (0.25,0.13);
    \draw[thick, ->] (1.75,-0.13) -- node[below=-1pt]{$\scriptstyle g$} (0.25,-0.13);
    \end{tikzpicture}
} }
are called \textit{homotopic} -- an equivalence relation denoted by $f \,{\sim}\, g$ -- if there exists a 2-simplex $\alpha$ of either of the equivalent types
  \begin{align}
  \raisebox{-2.6em}{\simptex{ \internaltwosimplex[x-y-y][f,=,g][\alpha]}}
  \qquad \text{or} \qquad
  \raisebox{-2.6em}{\simptex{ \internaltwosimplex[x-x-y][=,f,g][\beta]}}
  \end{align}
\end{defi}

\begin{defi} \label{def:nerve-etc}
(i) Let \C\ be a small category. The \textit{nerve} $N\C$ is the quasi-category with $N\C_0 \eq \Obj_\C$ and higher simplices generated by composition of morphisms.
\\[2pt]
(ii) Let $\C$ be a quasi-category. The \textit{homotopy category} $h\C$ of $\C$ is the category with objects given by $\Obj_{h\C} \eq \C_0$ and morphisms given by homotopy classes of 1-simplices.
\\[2pt]
(iii) In a quasi-category \C, a $1$-simplex $f$ is called an \textit{isomorphism} if it is mapped to an isomorphism in the homotopy category. Equivalently, for $f$ an isomorphism there is a $1$-simplex $f\inv$ and a pair
  \begin{align}
  \raisebox{-0.7em}{\simptex[baseline=4mm]{\internaltwosimplex[a-b-a][f,f\inv,=][\alpha]}}
  \qquad \text{and} \qquad
  \raisebox{-0.7em}{\simptex[baseline=4mm]{\internaltwosimplex[b-a-b][f\inv,f,=][\beta]}}
  \end{align}
of 2-simplices.
\end{defi}


\subsection{From monoidal categories to quasi-categories}

Recall the \textit{simplex category} $\Delta$ from Definition \ref{def:Delta}.

\begin{defi}
The \textit{augmented simplex category} $\Delta_*$ is the monoidal category constructed from $\Delta$ by adjoining an initial object $\mathbf{-1} \eq \emptyset$, with the tensor product given by disjoint union and with $\mathbf{-1}$ the monoidal unit.
\end{defi}

Every monoid object in any category is given by a monoidal functor from $\Delta_*$, making $\Delta_*$ the \textit{diagram category} for monoids, with string diagrams as the standard semantic. This mirrors how $\Delta$ is the diagram category for simplicial objects, with its own diagrammatic semantic.

While related, the two semantics have their own strengths and weaknesses. String diagrams are algebraic in nature, and when extended to a third dimension, they readily also model braided monoidal categories. But since objects are 1-dimensional and morphisms 0-dimensional, the standard semantic is not able to describe also higher-dimensional simplices. Meanwhile, simplices are combinatorial and can describe quasi-categories, but the standard semantics are not very useful for describing products of objects. However, alternative simplicial semantics are available when one constructs higher categories by enriching over monoidal categories.

To utilize these observations in our context, we introduce a single-object \C-enriched quasi-category $Q\C$ as follows.

\begin{defi} \label{def:QC} 
Let \C\ be a monoidal category \C. Then $Q\C$ is the quasi-category that is obtained from \C\ by lifting the dimension of every simplex by 1.  More explicitly, we introduce a single object $*$ to serve as the 0-cell, the objects of \C\ become 1-simplices, the 1-morphisms of \C\ become 2-simplices, and 2-morphisms become 3-simplices. The monoidal product in \C\ becomes the (not necessarily associative) morphism composition in $Q\C$. For $n \,{>}\, 3$ we set $Q\C_n$ to be trivial, i.e.\ consisting of only the simplices generated by lower-dimensional ones. (The latter accounts for the fact that	the monoidal categories of our interest have no extra higher morphisms.)
\end{defi}

To elaborate on this semantic, we usually omit labels for the vertices, since they can take only a single value. (We may label vertices by their ordering, though, in case we need to refer to them in some other context.) Further, for $n \,{\in}\, \{2,3\}$ we use the orientation given by the right hand rule to partition the boundary of a simplex into an in- and an out-boundary, and we read a \textit{labeled} $n$-simplex as an $n$-morphism in $Q\C$ from its in-boundary to its out-boundary, while an \textit{unlabeled} simplex stands for the space of possible labels. Thus we have e.g.
  \begin{align}
  \raisebox{-2.5em}{\twosimplex[][x,y,z][f]}
  \!\!=~ f ~~ \in ~~ \Hom_\C(z,x\otimes y) ~=\!\!
  \raisebox{-2.5em}{\twosimplex[][x,y,z]}
  \end{align}
 
There are two ways to compose 2-simplices; they are determined by whether the simplices connect via an initial/final vertex (1-composition over a 0-simplex) or via an edge (2-composition over a 1-simplex). The two possibilities look like
  \begin{align}
  f \otimes g ~=\!\!
  \raisebox{-0.5em}{\simptex[scale=1,baseline=3mm]{
  \internaltwosimplex[][x,y,z][f]
  \ddraw (nz.center) to["$k$" swap] (4,0);
  \ddraw (nz.center) to["$i$"] (3,1.4);
  \ddraw (3,1.4) to["$j$"] (4,0);
  \draw (3,0.5) node{$g$};
  }}
  & & \hspace*{-1.3em} \text{and} \qquad f\circ g ~=
  \raisebox{-3.25em}{\simptex{
  \internaltwosimplex[][i,j,k][f]
  \ddraw (1,-1.4) to["\ensuremath{n}" swap] (2,0);
  \ddraw (0,0) to["$m$" swap] (1,-1.4);
  \draw (1,-0.7) node{$g$};
}}
  \end{align}
respectively.

A crucial observation is now that every planar graph generated by these triangles is graph-dual to a planar string diagram, by way of the prescription
  \begin{align}
  \raisebox{-3.1em}{\simptex{
    \draw[red!50] (1,0.5) to (1,0);
    \ddraw[red!50] (1,0) to (1,-0.8);
    \ddraw[red!50] (-0.2,1.5) to (0.4,1);
    \draw[red!50] (0.4,1) to (1,0.5);
    \ddraw[red!50] (2.2,1.5) to (1.6,1);
    \draw[red!50] (1.6,1) to (1,0.5);
  \internaltwosimplex[][{},{},{}][\alpha]
    \draw (0.31,0.76) node{$i$};
    \draw (1.76,0.86) node{$j$};
    \draw (0.81,-0.25) node{$k$};
  } }
  \end{align}
The unlabeled empty regions of string diagrams become unlabeled vertices, strands become edges, and vertices become surfaces. It is therefore merely a matter of convenience and clarity to decide which semantic to employ in a given argument.


\subsection{From Hochschild homology to fusion quasi-categories}\label{sec:Hochschild2fusion}

In the sequel we assume that \ko\ is an algebraically closed field of characteristic zero and \C\ is a semisimple \ko-linear monoidal category with a finite number of isomorphism classes of simple objects. We choose a set $\S$ of representatives of these classes containing the monoidal unit $1$ and assume further (compare (ii) and (iii) of Definition \ref{def:ourfusion}) that for $i,j \iN \S$ the morphism space $\HomC(i,j)$ has dimension $\delta_{i,j}$ and that to any $i \iN \S$ there exists a unique $\Bar{i} \iN \S$ such that $\HomC(j\oti \Bar i,1)$ and $\HomC(\Bar i\oti j,1)$ have dimension $\delta_{i,j}$, for every $j\iN\S$.
 
To such a category \C\ we can associate its Grothendieck ring $\Gr(\C)$ by taking the projective module generated by the simple objects, with multiplication inherited from the tensor product \Cite{Def.\,4.5.2}{EGno}. By a slight abuse of notation, we here use the symbol $\Gr(\C)$ for the corresponding associative algebra over the ground field \ko.
In this section, we seek to go the opposite way, by considering the minimal full and dense monoidal subcategory of \C\ that contains all the information necessary to compute the \sjs s. To do so, we will utilize simplicial set theory to invoke some homological structures that let us relate the ring-theoretic and category-theoretic properties of the given category.

Recall that the \textit{Hochschild complex} of a ring $A$ with coefficients in an $A$-bimodule $M$ is given by the formal tensor powers of $A$: it is the chain complex
  \be
  M \xleftarrow{~\partial~} M\oti A \xleftarrow{~\partial~} M\oti A^2 \xleftarrow{~\partial~} M \oti A^3 \xleftarrow{~\partial~} \dots 
  \label{eq:HH}
  \ee
with boundary map $\partial \eq \sum_{i=0}^n (-1)^i\partial_i$ with
  \be
  \partial_i (a_0 \,{\otimes} \cdots {\otimes}\, a_n)
  = (a_0 \,{\otimes} \cdots {\otimes}\, a_{i-1} \oti (a_ia_{i+1}) \oti a_{i+2} \,{\otimes} \cdots {\otimes}\, a_n) 
  \ee
for $i \iN \{1,2,...\,,n{-}1\}$ and
  \be
  \partial_n (a_0 \,{\otimes} \cdots {\otimes}\, a_n) = ((a_na_0) \oti a_1 \,{\otimes} \cdots {\otimes}\, a_{n-1}) \,.
  \ee
  
In terms of simplices, each $\partial_i$ amounts to the list of 1-subsimplices of each face of the simplex of which the domain lists the spine, with overall sign given by the orientation of the face. The homology of this chain complex is called the Hochschild homology of $A$ with coefficients in $M$.

 \medskip

To explore how we can make use of this tool, let us as a warmup consider the special case that the fusion category \C\ is pointed. Then the associator of \C\ amounts to a group 3-cocycle on $\Gr(\C)$ \Cite{Ch.\,9.7}{EGno} or, equivalently, to a 3-cocycle in the dual Hochschild complex of $\Gr(\C)$ with coefficients in \ko. The possibility to make use of the dual Hochschild complex rests on the coboundary structure that is naturally induced on functions from the Hochschild complex to the ground field \ko\ (we denote the corresponding chain map by $F$) and from the fact that in the pointed case \ko\ is isomorphic to both the 1-morphism-spaces $\Hm ij{i\otimes j}$ and the 2-morphism spaces $\Hom_\ko(\Hm ij{i\otimes j} \oti \Hm{i\otimes j}k{(i\otimes j)\otimes k}, \Hm i{j\otimes k}{i\otimes(j\otimes k)} \oti \Hm jk{j\otimes k})$.
Thus after fixing a ``framing'', i.e.\ basis choice $\mathrm b\colon \ko \To H$, for each such space $H$, Hochschild cochains may be interpreted as elements of indexed sets of homomorphisms from $\Gr(\C)^{\otimes n}$ to \ko, so that we arrive at
  \begin{align}
  \begin{tikzcd}[ampersand replacement=\&]
  * \& \ar[l] \Gr(\C) \ar[d, "F_1" swap]
  \& \Gr(\C)^{\otimes 2} \ar[l,"\partial"]\ar[d,"F_2" swap]
  \& \Gr(\C)^{\otimes 3} \ar[l,"\partial"]\ar[d,"F_3" swap]
  \& \ar[l] \dots \ar[d,"F_4"]
    \\
  \& \ko \& \ko \ar[d, "\mathrm b" swap] \& \ko \ar[d] \& \ko 
  \\
  * \& \ar[l] Q\C_1 \& Q\C_2\ar[l,"\partial"] \& Q\C_3\ar[l,"\partial"] \& \ar[l] \dots
  \end{tikzcd}
  \label{eq:pointedcase}
  \end{align}
This way we obtain ``framed values'' of the maps $F_i$. Specifically, for $F_2$ we get $\alpha \cdo F_2(i\oti j) \iN \Hm ij{i\otimes j}$ for $i,j\iN\Gr(\C)$ and basis $\mathrm b \eq \alpha \,{\equiv}\, \alpha_{i,j} \colon \ko \To \Hm ij{i\otimes j}$. (We use the symbol ``$\cdot$'' to indicate that the basis choice furnishes an action of \ko\ of $Q\C_2$, whereby  we may think of $\alpha$ as a number.) 
Likewise we have ``framed \sjs s'' given by $(\alpha \oti \beta) \cdo F_3(i\oti j\oti k)\cdo (\bar{\gamma} \oti \bar{\delta})$ with $\alpha \,{\equiv}\, \alpha_{i,j \otimes k}$, $\beta \,{\equiv}\, \beta_{j,k}$, $\gamma \,{\equiv}\, \gamma_{i,j}$ and $\delta \,{\equiv}\, \delta_{i \otimes j,k}$. 
The framed coboundary of $F_2$ is thus
  \be
  \bearl
  (\alpha \oti \beta) \cdo \mathrm d F_2(i\oti j\oti k) \cdo (\cc\gamma \oti \cc\delta)
  = (\alpha \oti \beta) \cdo F_2(\partial(i\oti j\oti k)) \cdo (\cc\gamma \oti \cc\delta)
  \Nxl2
  \qquad\qquad = \big( \alpha \cdo F_2(i\oti jk) \otimes \beta \cdo F_2(j\oti k) \big) 
  \circ \big( \cc\gamma \cdo F_2(i\oti j)\inv_{} 
  \otimes \cc\delta \cdo F_2(ij\oti k)\inv_{} \big) \,.
  \eear
  \ee

When interpreting the cochain structure categorically, some caveats are to be noted.
First of all, 1-cochains do not necessarily frame the 1-simplices even for a pointed fusion category (put differently, $F_1$ cannot be interpreted as an object), as indicated by the lack of a corresponding vertical map in \eqref{eq:pointedcase}.
Second, as there are several notions of sums and products involved, with some of them constructed from each other, some abuse of notation is inevitable. Third, the coboundary of an $n$-cochain is defined computationally by reinterpreting $\mathrm dF_n \,{\equiv}\, F_n\partial$ as an $(n{+}1)$-cochain, and replacing the $n$-framing with the $(n{+}1)$-framing.
 
More explicitly, we may consider the associator-values $F_3$ to be a coboundary of $F_2$, and hence $F_3$ satisfies the pentagon equation since that is the group 3-cocycle equation. We can always choose a frame such that the framed values of $F_2$ are our basis vectors, which we can consider the \textit{composer} maps of the simple objects, in analogy of $F_3$ being the \textit{associator}.

\medskip

In the general case that the fusion category \C\ is not pointed, the reasoning above will typically fail. Specifically, the various morphism spaces will have different dimensions, so that a coboundary computed in the way described above cannot, in general, be directly reframed as a higher morphism. In fact, the numerical coefficients for different basis vectors in the framing will not generally coincide. Instead, we have to consider the framed values of $F_2$ to be vectors and the framed values of $F_3$ to be block matrices, and so on.
In order to extend the formalism, we must then know where the cochains should be valued, know how to compute coboundaries, and be able to interpret the data in terms of higher morphisms. 
To this end we adopt the following conventions.

\begin{defi} \label{def:SCetc}
Let \C\ be a semisimple \ko-linear monoidal category satisfying the additional properties listed at the beginning of this subsection and let $\S$ be the corresponding finite set of representatives for the isomorphism classes of simple objects. Let $Q\C$ be the associated \C-enriched quasi-category as given in Definition \ref{def:QC}.
 \\[3pt]
(i)
By $GQ\C_n$ we denote the groupoid of isomorphisms in $Q\C_n$ -- by which we mean composition-invertible morphisms for $n\,{>}\,1$, and isomorphisms in the sense of Definition \ref{def:nerve-etc}(iii), but restricted to non-zero fillers, for $n \eq 1$.
 \\[3pt]
(ii)
By $\S\C$ we denote the full and dense monoidal subcategory of \C\ that is generated (via composition and tensor products) by the objects in $\S$, i.e.\ the full subcategory that has as objects all expressions that can be built out of elements of $\S$ and of $\oplus$, $\otimes$ and parentheses.
 \\[3pt]
(iii)
By $Q\S\C_n$ we denote the set of $n$-simplices in the quasi-category associated with $\S\C$, i.e.\ the $n$-simplices for which every edge is an object in $\S\C$.
 \\[3pt]
(iv)
By $\RSC$ we denote the free unital non-associative ring generated by the set of isomorphism classes of simple objects of \C.
\end{defi}

We now take the Hochschild complex \eqref{eq:HH} for $A \eq \RSC$ and $M$ be the trivial bimodule $Q\C_0 \,{=}\, *$,
and consider isomorphism-valued spine-filling maps that send the generators of $\RSC$ to their representatives in $\S$. Thereby we have a chain map
  \begin{align}
  \begin{tikzcd}[ampersand replacement=\&]
    * \ar[d,equals]\& \ar[l] \RSC \ar[d, "F_1", two heads] \&
    \RSC^{\times 2} \ar[l,"\partial"]\ar[d,"F_2"] \&
    \RSC^{\times 3} \ar[l,"\partial"]\ar[d,"F_3"] \&
    \ar[l] ...\ar[d,two heads]
    \\
     * 
     \&
     \ar[l] \langle GQ\S\C_1\rangle 
     \&
     \langle GQ\S\C_2\rangle\ar[l,"\partial"] 
     \&
     \langle GQ\S\C_3\rangle\ar[l,"\partial"] 
     \&
     \ar[l] *
    \end{tikzcd}
  \end{align}
Here we denote by $\langle X_n\rangle$ the free $\mathbb Z$-module generated by a a variable $X_n$.

Properties of these chain maps are related to categorical properties of the full and dense monoidal subcategory $\S\C$ as follows:

\begin{enumerate}
    \item
$F_1$ selects the objects of$\S\C$, i.e.\ the representative simple objects and the objects constructed by them according to Definition \ref{def:SCetc}(ii).
    \item
In order for $\S\C$ to be a monoidal subcategory in the standard sense, $F_2$ must map the basis elements to identity morphisms. (Otherwise, $F_2$ yields a \textit{twisted} monoidal product, defined as for twisted group algebras, according to $[x] \,{*}\, [y] \,{=}\, f(x,y)\,[xy]$ with $f$ a group 2-cocycle.
    \item
In order for $\S\C$ to be \textit{strict}, $F_3$ must map the basis elements to identity 2-morphisms.
\item In order for the tensor product of $\S\C$ to satisfy the pentagon equation, $F_4$ must map the basis elements to identity 3-morphisms.
    \item
Strict unitality amounts to the 0-simplex $*$ lifting to the unit 1-simplex by the face map.
    \item
The filling of a 2-horn, considered as an $F_1\partial$-value, determines a value of $F_2$.
    \item
The filling of a 3-horn,
considered as an $F_2\partial$-value, determines a value of $F_3$.
\end{enumerate}

\noindent
By construction, the monoidal category $S\C$ comes equipped with rules for the fillings of horns; the $i{\otimes}j$-horn is filled by $\id_{i\otimes j}$, and any horn with the spine $i\oti j\oti k$ and oriented $F_2$-valued faces is filled by the 2-morphism $F_3(i\oti j\oti k) \eq \FF^{(i\,j\,k)\,i\otimes j \otimes k}_{(j\otimes k),(i\otimes j)}$. It is this filling that provides the coboundary structure $\mathrm d$.
In particular, at level 3 we can note that that $F_3$ is a 3-cocycle in the sense that its coboundary is trivial when evaluated with concrete indices. 
(On the other hand, for coboundaries of indexed 3-morphisms that are linearly independent of $F_3$ there need not exist such a notion of coboundary -- the filling prescription is part of the data that comes with the monoidal structure.)

Making basis choices allows us to compute matrix coefficients by projection to the simple objects. Specifically we recognize
  \begin{align}
          ~\nonumber \\[-2.5em]
  F_2(i\oti j)_{\alpha} ~=&
  \simptex[baseline=4mm]{\internaltwosimplex[][i,j,k][\alpha]} \in \Hom_{\S\C}(k,i\otimes j)
  \\[-5pt]
  \text{and} \qquad F_3(i\oti j\oti k)_{\beta\bar{\delta} \alpha\bar{\gamma}} ~= &
  \hspace*{-0.5em} \raisebox{-1.6em}{
  \simptex[baseline=4mm,scale=0.5]{\internalthreesimplex[][i,j,k,p,q,l][\gray\alpha,\gray\beta,\gamma,\gray\delta]}}
  \end{align}

\begin{exa}
Let us also mention how the case of \C\ being pointed is formulated in the so-obtained setting. We may identify the one-dimensional spaces $\Hm ijk$ with the ground field \ko, with the frame of $F_2$ as the chosen basis.
Then $F_3$ takes values in the multiplicative group of \ko, and covector duality corresponds to inversion in \ko.
Thereby addition in the ring $\RSC$ coincides with multiplication in \ko, and the filling procedure provides an explicit cochain complex on the chain maps $F_*$ when restricted to the simple objects. This endows $Q\S\C_1$ with the structure of a twisted group algebra over \ko, while 
  \be
  \begin{array}{rcl}
  F_1(j)\, F_1(i) = F_2(i,j)\, F_1(ij) &\Longrightarrow& \mathrm d F_1 = F_2 \,,
  \Nxl3
  F_2(j,k)\, F_2(i,jk) = F_3(i,j,k)\, F_2(ij,k)F_2(i,j) &\Longrightarrow& \mathrm d F_2 = F_3 \,,
  \Nxl3
  F_3(j,k,l)\, F_3(i,jk,l)\, F_3(i,j,k) = F_4(i,j,k,l)\, F_3(ij,k,l)\, F_3(i,j,kl)\, &\Longrightarrow& \mathrm d F_3 = F_4 
  \eear
  \ee
provides the prescription for computing the coboundaries of $F_1$, $F_2$ and $F_3$.
\end{exa}

\newpage

 \newcommand\wb{\,\linebreak[0]} \def\wB {$\,$\wb}
 \newcommand\Bi[2]    {\bibitem[#2]{#1}} 
 \newcommand\Epub[2]  {{\em #2}, {\tt #1}}
 \newcommand\Erra[3]  {\,[{\em ibid.}\ {#1} ({#2}) {#3}, {\em Erratum}]}
 \newcommand\inBO[9]  {{\em #9}, in:\ {\em #1}, {#2}\ ({#3}, {#4} {#5}), p.\ {#6--#7} {\tt [#8]}}
 \newcommand\inBOo[8] {{\em #8}, in:\ {\em #1}, {#2}\ ({#3}, {#4} {#5}), p.\ {#6--#7}}
 \newcommand\J[7]     {{\em #7}, {#1} {#2} ({#3}) {#4--#5} {{\tt [#6]}}}
 \newcommand\JO[6]    {{\em #6}, {#1} {#2} ({#3}) {#4--#5} }
 \newcommand\JP[7]    {{\em #7}, {#1} ({#3}) {{\tt [#6]}}}
 \newcommand\BOOK[4]  {{\em #1\/} ({#2}, {#3} {#4})}
 \newcommand\BOOP[2]  {{\em #2}, book in preparation \mbox{(#1)}}
 \newcommand\Lect[2]  {{\em #2}, Lecture notes (#1)}
 \newcommand\Mast[2]  {{\em #2}, Master thesis (#1)}
 \newcommand\PhD[2]   {{\em #2}, Ph.D.\ thesis #1}
 \newcommand\Prep[2]  {{\em #2}, preprint {\tt #1}}
 \def\aagt  {Alg.\wB\&\wB Geom.\wb Topol.}     
 \def\adma  {Adv.\wb Math.}
 \def\alnt  {Algebra\wB\&\wB Number\wB Theory}          
 \def\anip  {Ann.\wb Inst.\wB Henri\wB Poin\-car\'e}
 \def\anop  {Ann.\wb Phys.}
 \def\apcs  {Applied\wB Cate\-go\-rical\wB Struc\-tures}
 \def\atmp  {Adv.\wb Theor.\wb Math.\wb Phys.}   
 \def\anma  {Ann.\wb Math.}
 \def\bams  {Bull.\wb Amer.\wb Math.\wb Soc.}
 \def\camb  {Ca\-nad.\wb Math.\wb Bull.}
 \def\cocm  {Com\-mun.\wb Con\-temp.\wb Math.}
 \def\coma  {Con\-temp.\wb Math.}
 \def\comp  {Com\-mun.\wb Math.\wb Phys.}
 \def\epjb  {Eur.\wb Phys.\wb J.\ B}
 \def\fiic  {Fields\wB Institute\wB Commun.}
 \def\fuaa  {Funct.\wb Anal.\wb Appl.}
 \def\ijmp  {Int.\wb J.\wb Mod.\wb Phys.\ A}
 \def\imrn  {Int.\wb Math.\wb Res.\wb Notices}
 \def\injm  {Int.\wb J.\wb Math.}
 \def\jams  {J.\wb Amer.\wb Math.\wb Soc.}
 \def\jems  {J.\wb Europ.\wb Math.\wb Soc.}
 \def\jfst  {J.\wb Fac.\wb Sci.\wb Univ.\wB Tokyo}
 \def\jgap  {J.\wb Geom.\wB and\wB Phys.}
 \def\jomp  {J.\wb Math.\wb Phys.}
 \def\jhrs  {J.\wB Homotopy\wB Relat.\wB Struct.}
 \def\jktr  {J.\wB Knot\wB Theory\wB and\wB its\wB Ramif.}
 \def\jmsj  {J.\wb Math.\wb Soc.\wB Japan}
 \def\joal  {J.\wB Al\-ge\-bra}
 \def\jopa  {J.\wb Phys.\ A}
 \def\jpaa  {J.\wB Pure\wB Appl.\wb Alg.}
 \def\jram  {J.\wB rei\-ne\wB an\-gew.\wb Math.}
 \def\lemp  {Lett.\wb Math.\wb Phys.}
 \def\momj  {Mos\-cow\wB Math.\wb J.}
 \def\leni  {Lenin\-grad\wB Math.\wb J.}
 \def\marl  {Math\wb. Res.\wb Lett.}
 \def\masc  {Math.\wb Scand.}
 \def\nejp  {New J.\wB Phys.}
 \def\npbp  {Nucl.\wb Phys.\ B (Proc.\wb Suppl.)}
 \def\nupb  {Nucl.\wb Phys.\ B}
 \def\pajm  {Pa\-cific\wB J.\wb Math.}
 \def\phlb  {Phys.\wb Lett.\ B}
 \def\quto  {Quantum Topology}
 \def\phrb  {Phys.\wb Rev.\ B}
 \def\phrx  {Phys.\wb Rev.\ X}
 \def\tams  {Trans.\wb Amer.\wb Math.\wb Soc.}
 \def\rvmp  {Rev.\wb Math.\wb Phys.}
 \def\sema  {Selecta\wB Mathematica}
 \def\slnm  {Sprin\-ger\wB Lecture\wB Notes\wB in\wB Mathematics}
 \def\topo  {Topology}
 \def\trgr  {Trans\-form.\wB Groups}


\end{document}